\documentclass[10pt]{article}

\usepackage[utf8]{inputenc}

\usepackage[T1]{fontenc}
\usepackage{soulutf8}
\sethlcolor{yellow}

\usepackage[a4paper, margin=2.7cm]{geometry}
\usepackage[nottoc]{tocbibind}
\usepackage{mathtools, amsthm, amsfonts, amssymb}

\usepackage{mathrsfs}  

\usepackage{enumitem} 

\usepackage{titlesec} 
\usepackage[colorlinks=true,linkcolor=blue,citecolor=blue,urlcolor=blue,breaklinks]{hyperref}

\usepackage{graphicx} 
\usepackage{mathtools}

\usepackage{mathabx}

\usepackage{esint} 
\usepackage{bbm}

\usepackage{bigints}
\usepackage[dvipsnames]{xcolor}
\usepackage[toc,page]{appendix}
\usepackage{comment}

\usepackage{multirow}

\usepackage{tikz}
\usetikzlibrary{decorations.pathreplacing}

\numberwithin{equation}{section}

\DeclareMathOperator*{\fiint}{\ensuremath{\iint\text{\kern-1.36em{\raisebox{5.87pt}{\rotatebox{-93}{$\setminus$}}}}}}

\DeclareMathOperator*{\supess}{\ensuremath{ess{\text{ }}sup}}

\newtheorem{theorem}{Theorem}[section]
\newtheorem{lemma}[theorem]{Lemma}
\newtheorem{corollary}[theorem]{Corollary}
\newtheorem{proposition}[theorem]{Proposition}
\newtheorem{assumption}[theorem]{Assumption}
\allowdisplaybreaks

\newtheorem{definition}[theorem]{Definition}

\newcommand{\Tr}{{\mathrm{Tr}}}

\newcommand{\g}{\gamma}
\newcommand{\wg}{\widetilde{\gamma}}

\newtheorem{remark}[theorem]{Remark}

\providecommand\given{}

  \newcommand\SetSymbol[1][]{
   \nonscript\:#1\vert
   \allowbreak
   \nonscript\:
   \mathopen{}}
  \DeclarePairedDelimiterX\Set[1]\{\}{
   \renewcommand\given{\SetSymbol[\delimsize]}
   #1
}

\newcommand{\interior}[1]{
 {\kern0pt#1}^{\mathrm{o}}
}

\usepackage[style=numeric-comp,maxbibnames=9,natbib=true]{biblatex}
\AtEveryBibitem{
 \clearfield{issn} 
 \clearfield{doi}

 \ifentrytype{online}{}{
  \clearfield{url}
 }
}

\renewbibmacro*{name:andothers}{
 \ifboolexpr{
  test {\ifnumequal{\value{listcount}}{\value{liststop}}}
  and
  test \ifmorenames
 }
  {\ifnumgreater{\value{liststop}}{1}
    {\finalandcomma}
    {}
   \andothersdelim\bibstring[\emph]{andothers}}
  {}}

\usepackage{hyperref}
\newlist{primenumerate}{enumerate}{1}
\setlist[primenumerate,1]{label={(H\arabic*$'$})}

\title{Existence of minimizers for the Dirac--Fock model of crystals}

\author{
Isabelle Catto\footnote{ \textsc{Isabelle Catto, CEREMADE, Université Paris-Dauphine, Université PSL, CNRS, 75016 Paris, France}
 \textit{E-mail address}: \texttt{\href{mailto:catto@ceremade.dauphine.fr}{catto@ceremade.dauphine.fr}}}
\and Long Meng\footnote{ \textsc{Long Meng, CERMICS, \'Ecole des ponts ParisTech, 6 and 8 av. Pascal, 77455 Marne-la-Vall\'ee, France}
 \textit{E-mail address}:\texttt{\href{long.meng@enpc.fr}{long.meng@enpc.fr}}}
\and \'Eric Paturel\footnote{ \textsc{\'Eric Paturel, Laboratoire de Math\'ematiques J. Leray (UMR CNRS 6629), Universit\'e de Nantes, 44322 Nantes, France}
 \textit{E-mail address}: \texttt{\href{mailto:eric.paturel@univ-nantes.fr}{eric.paturel@univ-nantes.fr}}}
\and \'Eric S\'er\'e\footnote{ \textsc{\'Eric S\'er\'e, CEREMADE, Université Paris-Dauphine, Université PSL, CNRS, 75016 Paris, France}
 \textit{E-mail address}: \texttt{\href{mailto:sere@ceremade.dauphine.fr}{sere@ceremade.dauphine.fr}}}
 }
 
\date{}

\begin{document}

\maketitle
\begin{abstract} Whereas many different models exist in the mathematical and physical literature for ground states of non-relativistic crystals, the relativistic case has been much less studied and we are not aware of any mathematical result on a fully relativistic treatment of crystals. 
In this paper, we introduce a mean-field relativistic energy for crystals in terms of periodic density matrices. This model is inspired both from a recent definition of the Dirac--Fock ground state for atoms and molecules, due to one of us, and from the non-relativistic Hartree--Fock model for crystals. We prove the existence of a ground state when the number of electrons per cell is not too large.
\end{abstract}
\section{Introduction}
For solids with heavy atoms, relativistic shifts may affect the bonding properties and the optical properties. It is shown in \cite{relapyy} that the yellow color of gold is a result of relativistic effects. Furthermore, by studying the relativistic band structure in solids, it is shown in \cite{relachrist,Christen} that the relativistic shifts of the $5d$ bands relative to the $s-p$ bands in gold change the main interband edge by more than  1  eV.

A natural way to build quantum models for the crystal phase is to consider the so-called thermodynamic limit of quantum molecular models. Roughly speaking it consists in considering a finite but large piece of an (infinite and neutral) crystal. The thermodynamic law predicts that the ground state energy of the obtained large neutral molecule is proportional to the volume of this finite piece (which turns out to be also proportional to the total number of particles composing the molecule). The energy for the whole crystal is then identified with the limit--if it exists--of the energy per unit volume (or equivalently per particle) of the large molecule when the size of the considered piece goes to infinity. This method was applied successfully by different authors for several well-known models from quantum chemistry \cite{lieb1977thomas,catto1998mathematical,catto2001thermodynamic,catto2002some}--see also \cite{catto2000recent} for a review--but always for non-relativistic crystals.

The Dirac--Fock model (DF) was introduced in atomic physics by Swirles \cite{Swirles} in 1935. It is widely used in relativistic quantum chemistry, and gives numerical results on atoms and molecules in excellent agreement with experimental data \cite{kim1967relativistic,grant1970relativistic,desclaux1973relativistic}. Its relation with QED was investigated by Mittleman \cite{mittleman1981theory}. Mittleman's approach was studied mathematically in \cite{EstSer-02a, BarEstSer-05, BarFarHelSie-05, BarHelSie-06,meng}. To our knowledge the Dirac--Fock model has not been extended to crystals: there exist fully relativistic treatments of crystals in the physics literature, but they use the Kohn--Sham approach (see \cite{RelatKS,Kadek} and the references therein).

The first rigorous existence results for the atomic and molecular Dirac--Fock equations were obtained in \cite{esteban1999solutions,paturel2000solutions}. Compared to the non-relativistic models, the situation is different: the existence of bound states has only been proved when the total number of protons does not exceed $124$, for the physical value $\alpha\approx 1/137$ of the fine structure constant. Moreover, the Dirac--Fock energy functional is strongly indefinite and the notion of ground state has to be handled very carefully~\cite{esteban1999solutions}. These difficulties exclude a thermodynamic limit approach to derive the Dirac--Fock model for crystals.

In \cite{esteban2001nonrelativistic} it was shown that certain solutions of the (relativistic) Dirac--Fock equations converge towards the energy-minimizing solutions of the (non-relativistic) Hartree--Fock equations when the speed of light tends to infinity. This validates \textit{a posteriori} the notions of ground state solutions and ground state energy for the Dirac--Fock equations. In the approach of \cite{esteban2001nonrelativistic}, the multi-electronic state is modeled by a Slater determinant of mono-electronic wavefunctions. On the other hand, Huber and Siedentop use a density matrix formulation and a fixed-point iteration to define and construct ground states of the Dirac--Fock model~ \cite{huber2007solutions}. Unfortunately their assumptions do not cover the physical value of $\alpha$. Recently, in \cite{Ser09} one of us gave a new definition and an existence proof for the ground state of the Dirac--Fock model in atoms and molecules, under assumptions covering the physical value of $\alpha$, thanks to a density matrix formulation and a retraction technique combined with a minimization principle.
Inspired by this work and by the analysis of the periodic Hartree--Fock model due to Le Bris, Lions, and one of us \cite{catto2001thermodynamic}, we propose a definition for the ground state of the Dirac--Fock model for crystals which is a relativistic analogue of Lieb's variational principle for the Hartree--Fock model \cite{MR601336,MR1175492}, and we prove the existence of minimizers. In addition, we show that these minimizers solve a self-consistent equation. Our method can be used to calculate the ground state of neutral crystals with at most $17$ electrons per cell. However, some estimates used in this paper are not optimal, and we strongly believe that this limiting bound can be improved. 

 The minimization problem under consideration in this paper combines several difficulties related to compactness issues. The Dirac operator, hence the Dirac--Fock energy functional, is not bounded from below and the kinetic energy term is of the same order as the Coulomb-type potential energy terms, a standard feature of Coulomb--Dirac--Fock type models. Our proof of existence of minimizers for crystals is neither a straightforward adaptation of the one for atoms and molecules in \cite{Ser09} nor of the one for crystals in Hartree--Fock theory in \cite{catto2001thermodynamic}: a major issue arises from the compactness of the density matrices and of the self-consistent operators in the momentum variable $\xi$, resulting from the Bloch decomposition of the space. Compactness in the momentum variable is crucial to deal with the (non-linear) exchange term in the DF periodic functional and with the nonlinear constraint ensuring that the electrons lie in the positive spectral subspace of the self-consistent periodic Dirac--Fock operator. Our results rely on a careful analysis of the periodic exchange potential. In passing, we have corrected some wrong estimates on the exchange term in \cite{catto2001thermodynamic} and improved the regularity results therein (see Appendix B). Furthermore, we provide an asymptotically optimal constant for the Hardy inequality associated with the periodic Coulomb potential that is new in the literature, as far as we know.

In addition, compared with existing results for crystals' ground state energy, such as the Hartree--Fock model \cite{catto2001thermodynamic}, we provide a new method to prove the existence of minimizers for crystals: based on the spectral analysis of the self-consistent periodic DF operator, we build minimizing sequences that feature both a uniform dependence with respect to the momentum $\xi$ and a better regularity in the space variables, and we rely on it to improve the relative compactness of subsequences in the periodic energy space.

 Before ending this section, let us mention the Bogoliubov--Dirac--Fock (BDF) model proposed by Chaix and Iracane in \cite{CI} as an alternative to Dirac--Fock for heavy atoms and molecules, and later studied in a series of mathematical works (see the review paper \cite{HLSS} and references therein, see also the more recent works \cite{HLS-08,GLS-09,GLS-11}). Compared with DF, the BDF model has several advantages: the corresponding energy is bounded from below, so the notion of ground state becomes straightforward; vacuum polarization effects are taken into account; the derivation of the model as a mean-field approximation of no-photon QED is more convincing. However the mathematical definition of the BDF energy involves a rather complex functional framework as well as an ultraviolet regularization, and a renormalization procedure is needed to interpret the equations. Thus, in the present work we restrict ourselves to the conceptually simpler DF model, and the study of relativistic crystals in the BDF approximation is left for future research.

\section{General setting of the model and main result}\label{3.sec:2}
\subsection{Preliminaries and functional framework}
Throughout the paper, we choose units for which $m=c=\hbar=1$, where $m$ is the mass of the electron, $c$ the speed of light and $\hbar$ the Planck constant. 
For the sake of simplicity, we only consider the case of a cubic crystal with a single point-like nucleus per unit cell, which is located at the center of the cell. The reader should however keep in mind that the general case could be handled as well. Let $\ell>0$ denote the length of the elementary cell $Q_\ell=(-\frac{\ell}{2},\frac{\ell}{2}]^3$. The nuclei with positive charge $z$ are treated as classical particles with infinite mass that are located at each point of the lattice $\ell\,\mathbb{Z}^3$. The electrons are treated quantum mechanically through a periodic density matrix. The electronic density is modeled by a $Q_\ell$-periodic function whose $L^1$ norm over the elementary cell equals the ``number of electrons'' $q$ (the electrons' charge per cell is equal to $-q$). When $q=z$, electrical neutrality per cell is ensured. 

In this periodic setting, the $Q_\ell$-periodic Coulomb potential $G_\ell$ resulting from a distribution of point particles of charge $1$ that are periodically located at the centers of the cubic cells of the lattice is defined, up to a constant, by
\begin{equation}\label{eq:def-G}
  -\Delta G_\ell=4\pi\left[-\frac{1}{\ell^3}+\sum_{k\in\mathbb{Z}^3}\delta_{\ell k}\right].
\end{equation}
By convention, we choose $G_\ell$ such that
\begin{equation}\label{eq:constante-G}
  \int_{Q_\ell}G_\ell\,dx=0.
\end{equation}
The function $G_\ell$ is actually the Green function of the periodic Laplace operator on $Q_\ell$. The Fourier series of $G_\ell$ writes
\begin{equation}\label{3.eq:2.6}
  G_\ell(x)=\frac{1}{\pi \ell}\sum_{p\in\mathbb{Z}^3\setminus\{0\}}\frac{e^{\frac{2i\pi}{\ell}p\cdot x}}{|p|^2}, \quad\text {for every } x\in \mathbb{R}^3.
\end{equation}
Under the convention \eqref{eq:constante-G}, the periodic Coulomb potential changes sign, but is bounded from below~(see Lemma~\ref{lem:Gbound} in Appendix \ref{sec:A}).

\begin{remark}
  The size $\ell$ of the unit cell does not play a specific role here. It is however involved in the study of the Hardy-type inequalities for the periodic Coulomb potential (see Section \ref{sec:Coulomb}). When $\ell$ goes to infinity, one expects to recover the Dirac--Fock model for atoms. 
\end{remark}
The free Dirac operator is defined by $
D^0=-i\sum_{\rm{r}=1}^3\alpha_{\rm{r}}\partial_{\rm{r}}+\beta$, 
with $4\times4$ complex matrices $\alpha_1,\alpha_2,\alpha_3$ and $\beta$, whose standard forms are $
\beta=\begin{pmatrix} 
\mathbbm{1}_2 & 0 \\
0 & -\mathbbm{1}_2 
\end{pmatrix}$, 
$\alpha_{\rm{r}}=\begin{pmatrix}
0&\sigma_{\rm{r}}\\
\sigma_{\rm{r}}&0
\end{pmatrix}$ 
where $\mathbbm{1}_2$ is the $2\times 2$ identity matrix and the $\sigma_{\rm{r}}$'s, for $\rm{r}\in\{1,2,3\}$, are the well-known $2\times2$ Pauli matrices $
\sigma_1=\begin{pmatrix}
0&1\\
1&0
\end{pmatrix},\,
\sigma_2=\begin{pmatrix}
0&-i\\
i&0
\end{pmatrix}$, 
$\sigma_3=\begin{pmatrix}
1&0\\
0&-1
\end{pmatrix}.$

The operator $D^0$ acts on $4-$spinors; that is, on functions from $\mathbb{R}^3$ to $\mathbb{C}^4$. It is self-adjoint in $L^2(\mathbb{R}^3;\mathbb{C}^4)$, with domain $H^1(\mathbb{R}^3;
\mathbb{C}^4)$ and form domain $H^{1/2}(\mathbb{R}^3;\mathbb{C}^4)$ (denoted by $L^2$, $H^1$ and $H^{1/2}$ in the following, when there is no ambiguity). Its spectrum is $\sigma(D^0)=(-\infty,-1]\cup[+1,+\infty)$. Following the notation in \cite{esteban1999solutions,paturel2000solutions}, we denote by $\Lambda^+$ and $\Lambda^-=\mathbbm{1}_{L^2}-\Lambda^+$ respectively the two orthogonal projectors on $L^2(\mathbb{R}^3;\mathbb{C}^4)$ corresponding to the positive and negative eigenspaces of $D^0$; that is
\[
\begin{cases}
D^0\Lambda^+=\Lambda^+D^0=\Lambda^+\sqrt{1-\Delta}=\sqrt{1-\Delta}\,\Lambda^+;\\
D^0\Lambda^-=\Lambda^-D^0=-\Lambda^-\sqrt{1-\Delta}=-\sqrt{1-\Delta}\,\Lambda^-.
\end{cases}
\]
According to the Floquet theory \cite{reed1978b}, the underlying Hilbert space $L^2(\mathbb{R}^3;\mathbb{C}^4)$ is unitarily equivalent to $L^2(Q_\ell^*)\bigotimes L^2(Q_\ell;\mathbb{C}^4)$, where $Q_\ell^*=[-\frac{\pi}{\ell},\frac{\pi}{\ell})^3$ is the so-called reciprocal cell of the lattice, with volume $|Q_\ell^*|=(2\pi)^3/\ell^3$ (in the physics literature $Q_\ell^*$ is known as the first Brillouin zone). The Floquet unitary transform $U:L^2(\mathbb{R}^3;\mathbb{C}^4)\to L^2(Q_\ell^*)\bigotimes L^2(Q_\ell;\mathbb{C}^4) $ is given by
\begin{equation}\label{eq:def_Bloch_transform}
(U\phi)_\xi=\sum_{k\in\mathbb{Z}^3}e^{-i\ell k\cdot \xi}\phi(\cdot+\ell\,k)
\end{equation}
for every $\xi\in Q^*_\ell$ and $\phi$ in $L^2(\mathbb{R}^3;\mathbb{C}^4)$. For every $\xi\in Q^*_\ell$, the function $(U \phi)_\xi$ belongs to the space
\[
L^2_\xi(Q_\ell;\mathbb{C}^4)=\Set*{\psi\in L^2_{\text{loc}}(\mathbb{R}^3 ; \mathbb{C}^4)\given e^{-i\xi\cdot x}\psi\,\,\text{is }Q_\ell\text{-periodic}},
\]
which will be denoted by $L^2_\xi$ in the sequel. Functions $\psi$ of this form are called Bloch waves or $Q_\ell$-quasi-periodic functions with quasi-momentum $\xi\in Q^*_\ell$. They satisfy 
\[
\psi(\cdot+\ell\,k)=e^{i\ell\,k\cdot\xi}\psi(\cdot), \text{ for every }k\in \mathbb{Z}^3
.\]
For any function $\phi_\xi\in L^2_\xi$, using the definition of Fourier series expansion for $Q_\ell$-periodic functions, we write 
\begin{align}\label{eq:fourier}
  \phi_\xi(x)=\sum_{k\in \mathbb{Z}^3}\widehat{\phi}_{\xi}(k)\,e^{i\left(\frac{2\pi}{\ell}k+\xi\right)\cdot x},\text{ a.e. } x\in \mathbb{R}^3,
\end{align}
with coefficients 
\[
\widehat{\phi}_{\xi}(k)=\frac{1}{\ell^3}\int_{Q_\ell}\phi_\xi(y)e^{-i\left(\frac{2\pi}{\ell}k+\xi\right)\cdot y}\,dy\in \mathbb{C}^4.
\]
The Hilbert space $L^2_\xi$ is endowed with the norm
\[
\Vert\phi\Vert_{L^2_\xi}:= \left(\ell^3\sum_{k\in\mathbb{Z}^3}
\vert\widehat{\phi}_{\xi}(k)\vert^2\right)^{1/2}=\left(\int_{Q_\ell}|\phi_\xi(x)|^2\,dx\right)^{1/2}=\Vert\phi_\xi\Vert_{L^2(Q_\ell)}.\]
Here, and in the whole paper, we use the same notation $|\cdot|$ for the canonical Euclidean norm in $\mathbb{R}^n$, $\mathbb{C}^n$ or $ \mathcal{M}_n(\mathbb{C})$. When applied to self-adjoint operators, $|T|$ means the absolute value of $T$.\medskip

For every real number $s$, we also define
\[
H_\xi^s(Q_\ell;\mathbb{C}^4):=L_\xi^2(Q_\ell;\mathbb{C}^4)\cap H_{\text{loc}}^s(\mathbb{R}^3;\mathbb{C}^4) 
\]
endowed with the norm 
\[
\Vert\phi_\xi\Vert_{H^s_\xi}=\left(\ell^3\sum_{k\in \mathbb{Z}^3}\bigg(1+\Big|\frac{2\pi}{\ell}k+\xi\Big|^2\bigg)^{s}\,|\widehat{\phi}_{\xi}(k)|^2\right)^{1/2}
.\]
To simplify the notation, we simply write here and below $H_\xi^s$ when there is no ambiguity.

Operators $\mathcal{L}$ on $L^2(\mathbb{R}^3;\mathbb{C}^4)$ that commute with the translations of $\ell\,\mathbb{Z}^3$ can be decomposed accordingly into a direct integral of operators $\mathcal{L}_\xi$ acting on $L^2_\xi$ and defined by 
\begin{equation}\label{3.eq:2.0}
\mathcal{L}_\xi(U\phi)_\xi=(U \mathcal{L}\phi)_\xi \text{ for every } \phi \in L^2(\mathbb{R}^3;\mathbb{C}^4), \text{ a.e. }\xi\in Q_\ell^*
\end{equation}
(see \cite{reed1978b} for more details). We use the notation $\mathcal{L}=\fint_{Q_\ell^*}^\oplus \mathcal{L}_\xi d\xi$, with the shorthand $\fint_\Omega$ for $\frac{1}{|\Omega|}\int_\Omega$, to refer to this decomposition. In particular, for the free Dirac operator $D^0$ we have 
\begin{equation}\label{eq:somme_directe_Dirac}
  D^0=\fint_{Q_\ell^*}^\oplus D_\xi \,d\xi,
\end{equation}
where the $D_\xi$'s are self-adjoint operators on $L^2_\xi$ with domains $H_\xi^1$ and form-domains $H_\xi^{1/2}$. Note that $D_\xi^{\,2}=1-\Delta_\xi$, where $-\Delta=\fint_{Q_\ell^*}^\oplus-\Delta_\xi d\xi$. For every function $\phi_\xi\in H^{1}_\xi$, the operator $D_\xi$ is also defined by
\[
D_{\xi}\,\phi_\xi(x)=\sum_{k\in \mathbb{Z}^3}\left[\sum_{\rm{r=1}}^3\Big(\frac{2\pi}{\ell}k_{\rm{r}}+\xi_{\rm{r}}\Big)\cdot\alpha_{\rm{r}}+\beta\right]\widehat{\phi}_{\xi}(k)\,e^{i\big(\frac{2\pi k}{\ell}+\xi\big)\cdot x}.
\]
In particular, 
\begin{equation}\label{eq:DFourier}
(\phi_\xi,|D_\xi|\phi_\xi)_{L^2_\xi}=\ell^3\sum_{k\in \mathbb{Z}^3}\sqrt{1+\left|\xi+\frac{2\pi}{\ell}k\right|^2}\;|\widehat{\phi}_{\xi}(k)|^2.
\end{equation}
For every $\xi\in Q_\ell^*$, the positive spectrum of $D_\xi$ is composed of a non-decreasing sequence of real eigenvalues $(d^+_j(\xi))_{j\geq 1}$ counted with multiplicity. Each function $\xi\mapsto d_j^+(\xi)$ is continuous and $Q_\ell^*$-periodic, and one has
$d_j^+(Q_\ell^*)\in[c_*(j),c^*(j)]$ with
\begin{equation}\label{bornes}
c_*(j):=\min_{\xi\in Q_\ell^*} d_j^+(\xi)\;\;\hbox{ and }\;\; c^*(j):=\max_{\xi\in Q_\ell^*} d_j^+(\xi).
\end{equation}
Note that
\begin{align*}
 c_*(j)\geq 1
,\quad \lim_{j\to+\infty }c_*(j)=+\infty.
\end{align*}
In the same manner, the negative spectrum of $D_\xi$ is composed of the non-increasing sequence of real eigenvalues $d^-_j(\xi)=-d_j^+(\xi)$. Finally, one has
\begin{equation}\label{spectrum}
\bigcup_{\xi\in Q_\ell^*}\sigma(D_\xi)=\bigcup_{j\geq 1}\left[-c^*(j),-c_*(j)\right]\cup\left[c_*(j),c^*(j)\right]=\sigma(D^0)=(-\infty,-1]\cup[+1,+\infty).
\end{equation}
As in the Hartree--Fock model for crystals~\cite{catto2001thermodynamic}, the electrons will be modeled by an operator on $L^2(\mathbb{R}^3;\mathbb{C}^4)$, called the one-particle density matrix, that reflects their periodic distribution in the nuclei lattice.

We now introduce various functional spaces for linear operators on $L^2(Q_\ell;\mathbb{C}^4)$ and for operators on $L^2(\mathbb{R}^3;\mathbb{C}^4)$ that commute with translations. 
Let $\mathcal{B}\left(E \right)$ 
be the set of bounded operators from a Banach space $E$ to itself. We use the shorthand $\mathcal{B}(L^2_\xi)$ for $\mathcal{B}(L^2_\xi(Q_\ell;\mathbb{C}^4))$. The space of bounded operators on $\fint_{Q_\ell^*}^\oplus L^2_\xi\,d\xi= L^2(Q_\ell^*)\otimes L^2(Q_\ell;\mathbb{C}^4)$ which commute with the translations of $\ell \mathbb{Z}^3$ is denoted by $Y$. It is isomorphic to $L^\infty(Q_\ell^*;\mathcal{B}(L^2_\xi))$. Moreover, for every $h=\fint_{Q_\ell^*}^\oplus h_\xi\,d\xi \in Y$, 
\[
\Vert h\Vert_Y:=\supess_{\xi\in Q_\ell^*}\Vert h_\xi\Vert_{\mathcal{B}(L^2_\xi)}=\|h\|_{\mathcal{B}(L^2(\mathbb{R}^3;\mathbb{C}^4))}
\]
(see \cite[Theorem XIII.83]{reed1978b}). For $s\in [1,\infty)$ and $\xi\in Q_\ell^*$, we define
 \[
\mathfrak{S}_{s}(\xi):=\Set*{h_\xi\in\mathcal{B}(L^2_\xi)\given \Tr_{L^2_\xi}(|h_\xi|^s)<\infty}
\]
endowed with the norm
\[
\|h_\xi\|_{\mathfrak{S}_s(\xi)}:=\left(\Tr_{L^2_\xi}(|h_\xi|^s)\right)^{1/s}.
\]
We denote by $\mathfrak{S}_\infty(\xi)$ the space of compact operators on $L^2_\xi$, endowed with the norm inherited from $\Vert\cdot\Vert_{\mathcal{B}(L^2_\xi)}$. Similarly, for $t\in [1,+\infty]$, we define 
\begin{equation}\label{def:Sigmast}
 \mathfrak{S}_{s,t}:=\Set*{ h=\fint^\oplus_{ Q^*_\ell} h_\xi \,d\xi \given h_\xi\in \mathfrak{S}_s(\xi)\text{ a.e. }\xi\in Q_\ell^*, \|h_\xi\|_{\mathfrak{S}_s(\xi)}\in L^t(Q_\ell^*) }
\end{equation}
endowed with the norm 
\begin{equation}\label{def:norme-Sigmast}
\|h\|_{\mathfrak{S}_{s,t}}:=\left(\fint_{Q_\ell^*}\|h_\xi\|_{\mathfrak{S}_s(\xi)}^t d\xi\right)^{1/t}\quad \hbox{for }\,1\leq t<+\infty
\end{equation}
and 
\begin{equation}\label{def:norme-compact}
\|h\|_{\mathfrak{S}_{s,\infty}}:=\supess_{\xi \in Q_\ell^*} \|h_\xi\|_{\mathfrak{S}_s(\xi)}.
\end{equation}
In particular $\mathfrak{S}_{\infty,\infty}=L^\infty(Q_\ell^*;\mathfrak{S}_\infty(L^2_\xi))\subset Y$ is endowed with the norm of $Y$. 

In the sequel of this paper, we work with periodic one-particle density matrices belonging to subspaces $\mathfrak{S}_{1,p}$ of $\mathfrak{S}_{1,1}$, for $1\leq p\leq +\infty$. On such spaces, we can define the trace per unit cell as
\[
\widetilde{\Tr}_{L^2}(h):=\fint_{Q_\ell^*}\Tr_{L^2_\xi}(h_\xi)\,d\xi.
\]
Here, $\Tr_{L^2_\xi}$ means the usual trace of operators on the Hilbert space $L^2_\xi(Q_\ell;\mathbb{C}^4)$. It coincides with the trace of the operator with kernel $\Tr_4(h_\xi(\cdot,\cdot))$ on $L^2_\xi(Q_\ell;\mathbb{C})$ with $\Tr_4$ standing for the trace of a $4\times4$ matrix. The $\;\widetilde\quad$ reminds us that $\gamma$ is not trace-class on $L^2(\mathbb{R}^3)$. 

The trace per unit cell allows to define duality pairings between spaces $\mathfrak{S}_{s,t}$ using the classical duality properties in Schatten's spaces~\cite{SimonTrace}. More precisely, if $(s,s')$ and $(t,t')$ are in $[1,+\infty]^2$ with $1/s+1/s'=1$ and $1/t+1/t'=1$, then one can define a duality pairing $\langle\cdot,\cdot\rangle$ between $\mathfrak{S}_{s,t}$ and $\mathfrak{S}_{s',t'}$ as follows. For $h\in \mathfrak{S}_{s,t}$ and $h'\in \mathfrak{S}_{s',t'}$, the product $hh'$ is in $\mathfrak{S}_{1,1}$ and one sets
\[\langle h,h'\rangle:=\widetilde{\Tr}_{L^2}[h h'].\]
One has
\[
| \langle h,h'\rangle|\leq \Vert hh'\Vert_{\mathfrak{S}_{1,1}}\leq\Vert h\Vert_{\mathfrak{S}_{s,t}}\Vert h'\Vert_{\mathfrak{S}_{s',t'}}
.
\]

We also define
\[
X^\alpha(\xi)=\Set*{h\in\mathcal{B}(L^2_\xi)\given |D_\xi|^{\alpha/2}h_\xi|D_\xi|^{\alpha/2}\in \mathfrak{S}_1(\xi)}
\]
endowed with the norm
\[
\|h_\xi\|_{X^\alpha(\xi)}=\left\Vert|D_\xi|^{\alpha/2}h_\xi|D_\xi|^{\alpha/2}\right\Vert_{\mathfrak{S}_1(\xi)}
\]
and
\[
X^\alpha_t:=\Set*{h=\fint^\oplus_{ Q^*_\ell} h_\xi \,d\xi \given h_\xi\in \mathfrak{S}_1(\xi) \text{ a.e. }\xi\in Q_\ell^*, \left\||D_\xi|^{\alpha/2}h_\xi|D_\xi|^{\alpha/2}\right\|_{\mathfrak{S}_1(\xi)}\in L^t(Q_\ell^*)}
\]
endowed with the norm
\begin{equation*}
\|h\|_{X_t^\alpha}:=\left\Vert|D^0|^{\alpha/2}h|D^0|^{\alpha/2}\right\Vert_{\mathfrak{S}_{1,t}}.
\end{equation*}
For any two functional spaces $A$ and $B$ the norm of the intersected space is defined by 
\[
\|\gamma\|_{A\cap B}:=\max\{\|\gamma\|_{A};\|\gamma\|_{B}\}, \quad \forall \gamma\in A\cap B.
\]
For future convenience, we use the notation $X(\xi)$ for $X^1(\xi)$. We also set $X := X_1^1$ and
\[
Z:=\Set*{\gamma\in X\cap Y\, \given \,\gamma^*=\gamma\,}.
\]
We endow $Z$ with the norm inherited from $X\cap Y$, that is, we take
\[\Vert\gamma\Vert_{Z}:=\max\{\|\gamma\|_{X},\|\gamma\|_{Y}\}\;,\quad \forall \gamma\in Z.
\]
With this norm, $Z$ is a Banach space. The functional spaces $\mathfrak{S}_{1,1}$, $X$, $Y$ and $Z$ will play an essential role in the whole paper, while the functional space $\mathfrak{S}_{1,\infty}$ and its subspace $X^2_\infty$ are mainly used in Section \ref{3.sec:5}. In addition, we will also use the functional space $\mathfrak{S}_{\infty,1}$ in Section \ref{3.sec:5} since $\mathfrak{S}_{1,\infty}$ is its dual space.

\begin{definition}[Periodic one-particle density matrices]\label{def:one-part-density} We denote by $\Gamma$ the following set of $Q_\ell$-periodic one-particle density matrices:
\[
\Gamma:=\Set*{\gamma\in X\given \gamma^*=\gamma, \quad 0\leq \gamma\leq\mathbbm{1}_{L^2(\mathbb{R}^3)}}
\subset Z.
\]
\end{definition}

\begin{remark}\label{rk:noyau}
For $\gamma\in\Gamma$ and for almost every $\xi$ in $Q_\ell^*$, the operator $\gamma_\xi$ is compact on $L^2_\xi$ and admits a complete set of eigenfunctions $(u_n(\xi,\cdot))_{n\geq 1}$ in $L^2_\xi$ (actually lying in $H^{1/2}_\xi${}), corresponding to a non-decreasing sequence of eigenvalues $0\leq \mu_n(\xi)\leq 1$ (counted with their multiplicity). This is expressed as 
\begin{equation}
\label{eq:decomp-gamma}
\gamma_\xi=\sum_{n\geq1}\mu_n(\xi) \left| u_n(\xi,\cdot)\right\rangle\,\left \langle u_n(\xi,\cdot)\right|,\;\langle u_n(\xi,\cdot),u_m(\xi,\cdot)\rangle_{L^2_\xi}=\delta_{n,m}
\end{equation}
where $|u \rangle\,\langle u|$ denotes the projector onto the vector space spanned by the function $u$. Equivalently, for almost every $\xi$ in $Q_\ell^*$ and for any $(x,y)\in \mathbb{R}^3\times \mathbb{R}^3$, the Hilbert--Schmidt kernel writes
\begin{equation}\label{3.eq:2.1}
\gamma_{\xi}(x,y)=\sum_{n\geq1}\mu_n(\xi)u_n(\xi,x) u_n^*(\xi,y).
\end{equation}
In the above equation, $u_n(\xi,\cdot)$ is a column vector with four coefficients and the superscript $^*$ refers to transposition composed with complex conjugation of the coefficients. Thus,
$\gamma_\xi(x,y)$ is a $4\times 4$ complex matrix, and for every function $\varphi\in L^2_\xi$, 
\[
(\gamma_{\xi}\varphi)(x)=\int_{Q_\ell}\gamma_{\xi}(x,y)\varphi(y)\,dy=\sum_{n\geq1}\mu_n(\xi)u_n(\xi,x)\int_{Q_\ell}u_n^*(\xi,y)\varphi(y)\,dy.
\]
By definition of the trace of an operator, 
\[
\Tr_{L^2_\xi}(\gamma_\xi)=\sum_{n\geq 1}\mu_n(\xi).
\].
\end{remark}

\begin{definition}[Integral kernel and electronic density]
Let $\gamma$ belong to $\Gamma$. Then we can define in a unique way an integral kernel $\gamma(\cdot,\cdot)\in L^2(Q_\ell\times \mathbb{R}^3) \cap L^2(\mathbb{R}^3\times Q_\ell)$ with $\gamma(\cdot+k,\cdot+k)=\gamma(\cdot,\cdot)$ for any $k\in \mathbb{Z}^3$ and a $Q_\ell$-periodic density $\rho_\gamma$ associated to $\gamma$ by 
\begin{equation}\label{3.eq:2.12}
  \gamma(x,y)=\fint_{Q_\ell^*}\gamma_\xi(x,y)\,d\xi 
\end{equation}
and 
\begin{equation}\label{3.eq:2.2}
  \rho_\gamma(x)=\fint_{Q_\ell^*}\Tr_4\gamma_\xi(x,x)\,d\xi. 
\end{equation}
The function $\rho_\gamma$ is non-negative and belongs to $L^1(Q_\ell;\mathbb{R})$. Indeed, using the decomposition \eqref{3.eq:2.1}, we have 
\begin{equation}\label{eq:def-rho}
\rho_\gamma(x)=\fint_{Q_\ell^*}\sum_{n=1}^\infty\mu_n(\xi)\,|u_n(\xi,x)|^2\,d\xi
\end{equation}
and
\begin{equation*}
\int_{Q_\ell}\rho_\gamma(x)\,dx=\fint_{Q_\ell^*}\sum_{n=1}^\infty\mu_n(\xi)\,d\xi=\fint_{Q_\ell^*}\Tr_{L^2_\xi}(\gamma_\xi)\,d\xi.
\end{equation*}
In the physical setting we are interested in, the value of the above integral is the number of electrons per cell.

By the Cauchy--Schwarz inequality, it is easily checked that
\begin{equation}\label{eq:CS-gamma}
  \vert \gamma(x,y)\vert^2 \leq \rho_\gamma(x)\,\rho_\gamma(y), \quad \text { a.e. }x,y \in \mathbb{R}^3.
\end{equation}
Note that, when $h$ is a $Q_\ell$-periodic trace-class operator but is not necessarily a positive operator, we still may define $\rho_h$ with the help of \eqref{3.eq:2.2}, but \eqref{eq:CS-gamma} becomes $|h(x,y)|^2\leq \rho_{|h|}(x)\rho_{|h|}(y)$ where $|h|=\sqrt{h^*h}$.
\end{definition}

We can now introduce the periodic Dirac--Fock functional.

\subsection{The periodic Dirac--Fock model}

For $\gamma\in Z$, we define the periodic Dirac--Fock functional 
\begin{equation}
\begin{aligned}
\mathcal{E}^{DF}(\gamma)
{}&{}\quad=\fint\limits_{Q_\ell^*}\Tr_{L^2_\xi}[D_\xi\gamma_\xi]\,d\xi-\alpha z\int\limits_{Q_\ell}G_\ell(x)\rho_\gamma(x)\,dx\notag\\
{}&{}\qquad {} +\frac{\alpha}{2}\iint\limits_{Q_\ell\times Q_\ell}\rho_\gamma(x)G_\ell(x-y)\rho_\gamma(y)\,dxdy \label{3.energy}\\
{}&{}\qquad {} -\frac{\alpha}{2}\fiint\limits_{Q_\ell^*\times Q_\ell^*}\,d\xi d\xi'\iint\limits_{Q_\ell\times Q_\ell}\Tr_{4}{[\gamma_\xi(x,y)\gamma_{\xi'}(y,x)]}W_\ell^\infty(\xi-\xi',x-y)\,dxdy.\notag
\end{aligned}
\end{equation}
This functional is well-defined on $Z$ (see Remark \ref{rem:well-define} below). In the above definition of the energy functional, the so-called fine structure constant $\alpha$ is a dimensionless positive constant (the physical value is approximately 1/137). Note that $D_\xi\gamma_\xi$ is not a trace-class operator, so $\Tr_{L^2_\xi}[D_\xi\gamma_\xi]$ is not really a trace, it is just a notation for the rigorous mathematical object $\Tr_{L^2_\xi}[\vert D_\xi\vert^{1/2}\gamma_\xi\vert D_\xi\vert^{1/2}\mathrm{sign}(D_\xi)]$. We will make this abuse of notation throughout the paper.\medskip

The last term in \eqref{3.energy} is called the ``exchange term ''. The potential $W_\ell^\infty$ that enters its definition is given by
\begin{equation}\label{3.eq:2.8}  W_\ell^\infty(\eta,x)=\sum_{k\in\mathbb{Z}^3}\frac{e^{i\ell\,k\cdot\eta}}{|x+\ell\,k|}=\frac{4\pi}{\ell^3}\sum_{k\in \mathbb{Z}^3}\frac{1}{\left|\frac{2\pi k}{\ell}-\eta\right|^2}\,e^{i\left(\frac{2\pi k}{\ell}-\eta\right)\cdot x}
\end{equation}
(see \cite{catto2001thermodynamic} for a formal derivation of the exchange term from its analogue for molecules). It is $Q_\ell^*$-periodic with respect to $\eta$ and quasi-periodic with quasi-momentum $\eta$ with respect to $x$. 
For every $\gamma\in Z$, we now define the mean-field periodic Dirac operator 
\[
D_\gamma=\fint^\oplus_{Q^*_\ell} D_{\gamma,\xi}\,d\xi \quad \text{ with }\quad 
D_{\gamma,\xi}:= D_\xi-\alpha z\,G_\ell+\alpha V_{\gamma,\xi}\]
where
\begin{equation}\label{eq:def-V}
V_{\gamma,\xi}=\rho_{\gamma}\ast G_\ell-W_{\gamma,\xi}.
\end{equation}
Here,
\begin{equation}\label{eq:convol-G}
\rho_{\gamma}\ast G_\ell(x)= \int_{Q_\ell}G_\ell(y-x)\,\rho_{\gamma}(y)\,dy= \widetilde{\Tr}_{L^2}[G_\ell(\cdot-x)\,\gamma] 
\end{equation}
and
\begin{equation}\label{eq:W-xi}
W_{\gamma,\xi}\psi_\xi(x)=\fint_{Q_\ell^*}d\xi'\int_{Q_\ell}W_\ell^\infty(\xi'-\xi,x-y)\,\gamma_{\xi'}(x,y)\,\psi_\xi(y)\,dy.
\end{equation}
(In \eqref{eq:convol-G} we keep the notation $\cdot\ast \cdot$ for the convolution of periodic functions on $Q_\ell$.)\medskip

Let us explain the relation between $\mathcal{E}^{DF}$ and $D_{\gamma}$. The periodic DF energy may be rewritten
\[
\mathcal{E}^{DF}(\gamma)=\widetilde{\Tr}_{L^2}[(\mathcal{D}^0-\alpha \,G)\,\gamma+\frac{\alpha}{2}V_\gamma\,\gamma].
\]
It is smooth on $Z$, and its differential at $\gamma\in Z$ is the linear form
\[d\mathcal{E}^{DF}(\gamma)\,:\, Z\ni h \mapsto\fint_{Q_\ell^*}\Tr_{L^2_{\xi}}[D_{\gamma,\xi}h_{\xi}]\,d\xi=\widetilde{\Tr}_{L^2}[D_\gamma\,h]. 
\]

We introduce the following set of periodic density matrices :
\[
\Gamma_{q}:=\Set*{\gamma\in \Gamma \given \|\gamma\|_{\mathfrak{S}_{1,1}}= q}
\]
and 
\[
\Gamma_{\leq q}:=\Set*{\gamma\in \Gamma \given \|\gamma\|_{\mathfrak{S}_{1,1}}\leq q}.
\]
Here $q$ is a positive real number. The elements of $ \Gamma_{q}$ (resp. $\Gamma_{\leq q}$) are Dirac--Fock density matrices with particle number per unit cell equal to $q$ (resp. at most $q$). 
\medskip

Our goal is to define the ground state despite the fact that the energy functional $\mathcal{E}^{DF}$ is strongly indefinite on $\Gamma_{\leq q}$, due to the unboundedness of the Dirac operator $D^0$. 

\subsection{Ground state energy and main result}
We follow Dirac's interpretation of the negative energy states of Dirac--Fock models: Such states are supposed to be occupied by virtual electrons that form the Dirac sea. Therefore, by the Pauli exclusion principle, the states of physical electrons are orthogonal to all the negative energy states. The ground- energy and state should thus be defined on the positive spectral subspaces of the corresponding Dirac--Fock operator.
Let 
\[
P_{\gamma}^{\pm}:=\fint^\oplus_{Q^*} P_{\gamma,\xi}^{\pm}\,d\xi \quad \text { with }\quad P_{\gamma,\xi}^{\pm}:=\mathbbm{1}_{\mathbb{R}_{\pm}}(D_{\gamma,\xi}).
\]
Note that by definition $P_{0,\xi}^{\pm}=\mathbbm{1}_{\mathbb{R}_{\pm}}(D_\xi-\alpha zG_\ell)$. We define the set
\begin{align}\label{eq:gamma+}
  \Gamma_{q}^+:=\Set*{\gamma\in \Gamma_{q}\given \gamma=P^+_\gamma \gamma P^+_\gamma}
\end{align}
and the ground state energy 
\begin{align}\label{eq:I-q}
  I_q:=\inf_{\gamma\in \Gamma_q^+}\mathcal{E}^{DF}(\gamma).
\end{align}
We need the following assumption.
\begin{assumption}\label{3.ass:2.1}
Let $q^+:=\max\{q;1\}$ and $\kappa:=\alpha\,\big(C_G z+C_{EE}'q^+\big)$. We also introduce the positive constants $e_0:=(1-\kappa)^{-1}c^*(\lceil q\rceil)$ and $a:=\frac{\alpha}{2}C_{EE}\,(1-\kappa)^{-1/2}\lambda_0^{-1/2}$ (well-defined if $\kappa<1$). Here we have used the standard notation $\lceil q\rceil:=\min\{m\in\mathbb{N}\;\vert\;m\geq q\}$ and  $c^*(\cdot )$ is given by formula \eqref{bornes}.\medskip

\noindent
We demand that
\begin{enumerate}
  \item $\kappa<1-\frac{\alpha}{2} C_{EE}q^+$ ;
  
\item $2a\,\sqrt{\max\{(1-\kappa-\frac{\alpha}{2} C_{EE}q^+)^{-1}e_0 q;1\}q^+}<1$.
\end{enumerate}
 The positive constants $C_G$, $C_{EE}$, $C_{EE}'$ and $\lambda_0$ are defined respectively in Lemmas \ref{3.lem:4.2},  \ref{3.lem:5.1} and \ref{lem:bottom} below.
\end{assumption}
Our main result is the following.

\begin{theorem}[Existence of a ground state]\label{3.th:2.7}
When $\alpha$, $q$, $z$ and $\ell$ satisfy Assumption \ref{3.ass:2.1}, there exists $\gamma_*\in\Gamma_{q}^+$ such that
\begin{equation}\label{3.eq:2.21'}
  \mathcal{E}^{DF}(\gamma_*)=I_q=\min_{\gamma\in\Gamma_{ q}^+}\mathcal{E}^{DF}(\gamma).
\end{equation}
Besides, $\gamma_*$ solves the following nonlinear self-consistent equation
\begin{equation}\label{3.eq:2.22'}
\gamma=\mathbbm{1}_{[0,\nu)}(D_\gamma)+\delta
\end{equation}
where $0\leq \delta \leq \mathbbm{1}_{\{\nu\}}(D_\gamma)$ and $\lambda_0\leq \nu\leq e_0$, with $e_0$ being defined in Assumption \ref{3.ass:2.1}, and $\lambda_0\ge 1-\kappa>0$ in Lemma~\ref{lem:bottom}. 
\end{theorem}

\begin{remark}[Projectors]
According to \cite{MR1175492,MR601336,ghimenti2009properties} any ground state of the Hartree--Fock model (both for the molecules and crystals) is a projector. However we do not know whether the ground states of Dirac--Fock model are projectors in general.
\end{remark}

\begin{remark}\label{3.rem:2.15}
In Solid State Physics, the length of the unit cell is about a few \AA ng\-str\"{o}ms. In our system of units, $\hbar=m=c=1$, thus $\alpha\approx \frac{1}{137}$ and $\ell\approx 1000$. Under the condition $q=z$ for electrical neutrality, Assumption \ref{3.ass:2.1} is satisfied for $q\leq 17$. The proof is detailed in Appendix \ref{sec:D}. Our estimates are far from optimal: The ideas of this paper are expected to apply to higher values of $q$. 
\end{remark}

\section{Sketch of proof of Theorem \ref{3.th:2.7}}\label{sec:3}
We are convinced that the constraint set $\Gamma^+_q$ is not convex, and we are not able to prove that it is closed for the weak-$^*$ topology of $Z$. This is the source of considerable difficulties. Mimicking \cite{Ser09}, we shall use a retraction technique as for the Dirac--Fock model for atoms and molecules. This imposes to search the ground state in the set $\Gamma_{\leq q}^+$ defined by
\[
\Gamma_{\leq q}^+:=\Set*{\gamma\in \Gamma_{\leq q}\given \gamma=P^+_\gamma \gamma P^+_\gamma}.
\]
However, under the above constraint, the minimizers may not be situated in $\Gamma_{q}^+$. To overcome this problem, we next subtract a penalization term $\epsilon_{P} \,\widetilde{\Tr}_{L^2}(\gamma)$ from the DF energy functional, for some parameter $\epsilon_{P}>0$ to be chosen later, and we first study a minimization problem for the penalized functional with relaxed constraint. We introduce the infimum
\begin{align}\label{3.eq:5.0'}
  J_{\leq q}:=\inf_{\gamma\in\Gamma_{\leq q}^+}\left[\mathcal{E}^{DF}(\gamma)-\epsilon_{P}\widetilde{\Tr}_{L^2}(\gamma)\right].
\end{align}
If this infimum is attained at some
$\gamma_*\in \Gamma_{\leq q}^+$, $\gamma_*$ will be called a {\it minimizer for $J_{\leq q}$}. We are going to see that for a suitably chosen value of $\epsilon_P$, $J_{\leq q}$ is attained and that every minimizer for $J_{\leq q}$ lies in $\Gamma_{q}^+$, thus is a minimizer for $I_q$.

\medskip
For the study of the penalized problem $J_{\leq q}$, we need an analogue of Assumption~\ref{3.ass:2.1}: 
\begin{assumption}\label{3.ass:2.1'}
Let $q^+=\max\{q;1\}$, $\kappa:=\alpha\,(C_Gz+C'_{EE} q^+)$ and $a:=\frac{\alpha}{2}C_{EE}\,(1-\kappa)^{-1/2}\lambda_0^{-1/2}$ (well-defined if $\kappa<1$). We assume that
\begin{enumerate}
  \item $\kappa<1-\frac{\alpha}{2} C_{EE}q^+$~;
  
  \item $2a\,\sqrt{\max\{(1-\kappa-\frac{\alpha}{2} C_{EE}q^+)^{-1}\epsilon_{P}\,q;1\}q^+}<1$.
\end{enumerate}
\end{assumption}
The relation between assumptions~\ref{3.ass:2.1} and \ref{3.ass:2.1'} is given by the following lemma:
\begin{lemma}[Choice of $\epsilon_P$]
\label{lem:assumptions} Assume that Assumption~\ref{3.ass:2.1} on $q$ and $z$ holds. Then, there is a constant $\epsilon_{P}>e_0$ such that Assumption~\ref{3.ass:2.1'} is satisfied.
\end{lemma}
\begin{proof}
One just needs to take $\epsilon_{P}=e_0+\varepsilon$ with $\varepsilon$ positive and small enough.
\end{proof}
We now state an existence result:
\begin{theorem}[Existence of a minimizer for the penalized problem] \label{3.th:2.7'}
We suppose that Assumption \ref{3.ass:2.1'} on $q,z,\epsilon_P$ holds and recall the notation $e_0:=(1-\kappa)^{-1} c^*(\lceil q \rceil)$. If $\epsilon_{P}>e_0$, then there exists $\gamma_*\in\Gamma_{\leq q}^+$ such that
\begin{equation}\label{3.eq:2.10}
  \mathcal{E}^{DF}(\gamma_*)-\epsilon_{P}\widetilde{\Tr}_{L^2}(\gamma_*)=J_{\leq q}.
\end{equation}
Besides, $\widetilde{\Tr}_{L^2}(\gamma_*)=\fint_{Q_\ell^*} \Tr_{L^2_\xi}(\gamma_{{*},\xi})\,d\xi=q$ and $\gamma_*$ solves the following nonlinear self-consistent equation
\begin{equation}\label{3.eq:2.21}
\begin{aligned}
\gamma=\mathbbm{1}_{[0,\nu)}(D_\gamma)+\delta
\end{aligned}
\end{equation}
where $0\leq \delta \leq \mathbbm{1}_{\{\nu\}}(D_\gamma)$ and $\nu\in [\lambda_0,e_0]$ is the Lagrange multiplier due to the charge constraint $
\Tr_{L^2
}(\gamma)\leq q$. 
\end{theorem}
Theorem \ref{3.th:2.7} is a direct consequence of Lemma \ref{lem:assumptions} and Theorem \ref{3.th:2.7'}. Indeed, if Assumption~\ref{3.ass:2.1} on $q,z$ holds, Lemma \ref{lem:assumptions} guarantees the existence of $\epsilon_{P}$ such that the assumptions of Theorem \ref{3.th:2.7'} are satisfied. Then this theorem provides a minimizer $\gamma_*$ for $J_{\leq q}$ which lies in $\Gamma_q^+$, hence $J_{\leq q}+\epsilon_P q={\mathcal E}^{DF}(\gamma_*)\geq I_q$. On the other hand, for each $\gamma\in\Gamma_q^+$ one has the inequality $\mathcal {\mathcal E}^{DF}(\gamma)\geq J_{\leq q}+\epsilon_P q$, so $I_q\geq J_{\leq q}+\epsilon_P q$. As a consequence, we get ${\mathcal E}^{DF}(\gamma_*)=J_{\leq q}+\epsilon_P q=I_q$ which is the same as \eqref{3.eq:2.21'}. Moreover $\gamma_*$ satisfies \eqref{3.eq:2.21} which is the same as \eqref{3.eq:2.22'}.\medskip

Therefore, in the sequel of this paper we focus on the proof of Theorem~\ref{3.th:2.7'}. 
Before going further, we explain the difficulties we face and the strategy we adopt to solve them, by comparing with the Hartree--Fock case \cite{catto2001thermodynamic}. The method used in \cite{catto2001thermodynamic} is based on some properties of the Schr\"odinger operator $-\Delta$: 
\begin{enumerate}
  \item This operator is non-negative. Hence the Hartree--Fock model for crystals is well-defined and the kinetic energy is weakly lower semi-continuous w.r.t. the density matrix ; 
  \item The exchange potential $W_{\ell}^\infty$ is rather easily controlled by the Schr\"odinger operator $-\Delta$. 
\end{enumerate}
In \cite{catto2001thermodynamic}, these properties allow to deduce bounds on the minimizing sequence of density matrices w.r.t. the $\xi$, $x$ and $y$ variables, and to pass to the limit in the different terms of the energy functional, in particular in the exchange term which is the most intricate one. In the proof, the strong convergence of the density matrix kernels $\gamma_n(x,y)=\fint_{Q_\ell^*}\gamma_{n,\xi}(x,y)\,d\xi$ plays an important role. In addition, the charge constraint in the periodic Hartree--Fock model is linear with respect to the density, and there is no possible loss of charge in passing to the limit.

In the Dirac--Fock model for crystals, two additional difficulties occur. First of all, the Dirac operator does not control the potential energy terms, which are of the same order. Secondly, the convergence of the nonlinear constraint $\fint_{Q_\ell^*}^\oplus P^+_{\gamma,\xi}\gamma_\xi d\xi=\fint_{Q_\ell^*}^\oplus \gamma_\xi d\xi$ requires stronger compactness properties of the sequence of density matrices with respect to the $\xi$ variable (of course, this second difficulty does not exist for the Dirac--Fock model of atoms and molecules, since in that case the $\xi$ variable is absent). Therefore the proof of existence of minimizers in the periodic Hartree--Fock setting cannot be applied \textit{mutatis mutandis}. The functional space $Z$ is natural to give a sense to the energy functional and to the constraints, but the weak convergence of minimizing sequences in $Z$ is not sufficient to deal with the exchange term and the non-linear constraints. 
 The whole paper (except Section 5 about the retraction) is devoted to solving the difficulties arising from the $\xi$ variable. \medskip

{\bf Strategy for the proof of Theorem \ref{3.th:2.7'}.} Our strategy rather relies on the spectral analysis of the periodic Dirac--Fock operator, which is new, to our knowledge, for the proof of existence of minimizers in the periodic case. Thanks to this spectral analysis, in Lemma \ref{3.lem:6.4} together with Lemma \ref{3.lem:6.3} (see also Remark \ref{rem:4.13}), we can prove that every minimizer for $J_{\leq q}$ actually lies in $\mathfrak{S}_{1,\infty}$, and is situated in $ {\bf{\overline B}}$, where
\begin{align}\label{eq:BR}
  {\bf{\overline B}}:=\Set*{\gamma\in Z\; \given \;\|\gamma\|_{\mathfrak{S}_{1,\infty}}\leq j_1}
\end{align}
and $j_1$ is an integer defined in Subsection \ref{ssec:spectral}.\medskip

The key point in the proof of the existence of minimizers for $J_{\leq q}$ is that for any minimizing sequence $(\gamma_n)$ for $J_{\leq q}$, we are able to construct another minimizing sequence $(\widetilde{\gamma}_n)$ satisfying the same regularity estimate as the minimizers, that is, $ \widetilde{\gamma}_n\in {\bf{\overline B}}$. This is the content of Lemma \ref{3.lem:6.5}. This estimate helps considerably to solve the problem of passing to the limit in the constraint $\fint_{Q_\ell^*}^\oplus P^+_{\gamma,\xi}\gamma_\xi P^+_{\gamma,\xi}d\xi=\fint_{Q_\ell^*}^\oplus \gamma_\xi d\xi$.\medskip

{\bf Organization of the paper.} The next sections are devoted to the proof of Theorem \ref{3.th:2.7'}. Our paper is organized as follows.\medskip

In Section \ref{3.sec:3}, we collect some fundamental estimates on the potentials $G_\ell$ and $W_\ell^\infty$, that ensure in particular that the DF periodic energy functional is well-defined and smooth on $Z$. In Subsection \ref{ssec:spectral}, we study the spectral properties of the Dirac--Fock operators $D_{\gamma,\xi}$ for every $\xi\in Q_\ell^*$. Relying on them, we study in Subsection \ref{sec:5.1} the properties of minimizers for a linearized Dirac--Fock energy. Finally, we collect the first estimates on minimizing sequences for $J_{\leq q}$.\medskip

In Section \ref{4.sec:5}, we study the Euler-Lagrange equation associated to $J_{\leq q}$. We conclude that each minimizer for $J_{\leq q}$ is in $\Gamma_{q}^+$ and solves a self-consistent equation (it is a ground state of its own mean-field Hamiltonian) and that minimizing sequences are approximate solutions. In Hartree--Fock type models for molecules \cite{lions1987solutions} or crystals \cite{ghimenti2009properties}, it is a standard fact that the approximate minimizers are also approximate ground states of their mean-field Hamiltonian, and the proof relies on the convexity of the constraint set. However, in the Dirac--Fock model (both for molecules and crystals), the constraint set $ \Gamma^+_{\leq q}$ does not seem to be convex. By using a retraction technique, a similar difficulty was recently overcome by one of us for the Dirac--Fock model of molecules \cite{Ser09}. Adapting the method of \cite{Ser09}, we define a set $\mathcal{V}$ which is relatively open in $\Gamma_{\leq q}$ for the norm of $Z$, and we build a regular map $\theta :\overline{\mathcal{V}}\to \overline{\mathcal{V}}\cap \Gamma_{\leq q}^+$ such that $\theta(\gamma)=\gamma,\;\forall \gamma\in \overline{\mathcal{V}}\cap \Gamma_{\leq q}^+$. Here, $\overline{\mathcal{V}}$ is the closure of $\mathcal{V}$ in $Z$. Under our assumptions, there exist minimizing sequences for $J_{\leq q}$ lying in $\mathcal{V}\cap \Gamma_{\leq q}^+$,
hence the equality
\[
J_{\leq q}=\inf_{\gamma\in \mathcal{V}}\left\{\mathcal{E}^{DF}(\theta(\gamma))-\epsilon_P\widetilde{\Tr}_{L^2}[\theta(\gamma)]\right\}
\]
which allows us to prove that the terms of minimizing sequences are approximate ground states of their mean-field Hamiltonian.\medskip

Then, in Section \ref{3.sec:5}, we build modified minimizing sequences lying in ${\bf{\overline B}}$. Finally, we prove the convergence of such sequences to a minimizer for $J_{\leq q}$, and this ends the proof of Theorem \ref{3.th:2.7'}.\medskip

Assumption \ref{3.ass:2.1} involves optimal constants in Hardy-type inequalities introduced in Subsection \ref{sec:Coulomb}. Therefore, in Appendix \ref{sec:A}-\ref{sec:C}, we prove Lemma \ref{3.lem:4.2}, Lemma \ref{3.lem:4.5} and Lemma \ref{3.lem:5.1} respectively. Finally, in Appendix \ref{sec:D}, we calculate the maximum number of electrons per cell allowed by the model, relying on approximate values of the constants obtained in Appendices \ref{sec:A}-\ref{sec:C}. 

\section{Fundamental estimates}\label{3.sec:3}
In this section, we give Hardy-type inequalities for the periodic Coulomb potential and provide estimates on the interaction potential between electrons in crystals. Then we study the spectrum of the periodic self-consistent Dirac--Fock operators. Finally, using this spectral analysis, we derive properties of the minimizers for a linearized problem, and {\it a priori} bounds on minimizing sequences for $J_{\leq q}$ .

\subsection{Hardy-type estimates on the periodic Coulomb potential}\label{sec:Coulomb}

First of all, and this is a major difference with the usual Coulomb potential $\frac{1}{|x|}$ in $\mathbb{R}^3$, the periodic Coulomb potential $G_\ell$ may not be positive, since it is defined up to constant, but it is bounded from below (see Lemma \ref{lem:Gbound} in Appendix \ref{sec:A}). Nevertheless, the operator of convolution with $G_\ell$ is positive on $L^2(Q_\ell)$ in virtue of \eqref{3.eq:2.6}. Moreover, we have the following Hardy-type estimates concerning the operator of multiplication by the periodic potential $G_\ell$. 

\begin{lemma}[Hardy-type inequalities for the periodic Coulomb potential]\label{3.lem:4.2}
There exist positive constants $C_H=C_H(\ell)>0$ that only depends on $\ell$ and such that

\begin{equation}\label{3.eq:4.1a}
  G_\ell\leq |G_\ell|\leq C_H\,|D^0|
\end{equation}
in the sense of operators on $L^2(Q_\ell^*)\bigotimes L^2(Q_\ell;\mathbb{C}^4)$.

Moreover, there exists a positive constant $C_G=C_G(\ell)$ with $C_G\geq C_H$ that only depends on $\ell$ and such that
\begin{equation}\label{3.eq:4.3}
  \|G_\ell\,|D^0|^{-1}\|_Y = C_G.
\end{equation}
\end{lemma}

\begin{remark}
In \eqref{3.eq:4.1a}, the inequality $A\leq B$ is equivalent to : For almost every $\xi\in Q_\ell^*$,  $A_\xi\leq B_\xi$ in the sense of operators on $L^2_\xi$. 
\end{remark}

\begin{remark}\label{rk:constantes}
The constant $C_G(\ell)$ is estimated in \eqref{const:CG} in Appendix \ref{sec:A} below. While it is far from optimal when $\ell$ is small, it converges to $2$ when $\ell$ goes to infinity; that is, to the value of the optimal constant for the Coulomb potential on the whole space. By interpolation,
\begin{equation}\label{const:CH}
C_H\leq C_G.
\end{equation} 
Therefore, \eqref{3.eq:4.1a} holds with $C_H$ being replaced by $C_G$. However, $C_H$ is expected to converge to $\pi/2$ as $\ell$ goes to infinity; that is, to the best constant in the Kato--Herbst inequality on the whole space \cite{kato2013perturbation,Herbst}. 
\end{remark}
A by-product of Lemma \ref{3.lem:4.2} is the following. 
\begin{corollary}[Estimates on the direct term]\label{3.cor:4.4}
For any $\gamma\in X$, we have
\begin{align}\label{3.eq:4.0}
  \|\rho_\gamma\ast G_\ell\|_{Y}&\leq C_H\,\|\gamma\|_{X}
\shortintertext{and}\notag\\
\|(\rho_\gamma\ast G_\ell)\,|D^0|^{-1}\|_Y&\leq C_G\,\|\gamma\|_{\mathfrak{S}_{1,1}}.
\label{3.cor:4.4'}
\end{align}
\end{corollary}
\begin{proof}
For every $x\in \mathbb{R}^3$ and $\gamma\in X$, we have
\[
\begin{aligned}
\MoveEqLeft\left|\rho_{\gamma} \ast G_\ell(x)\right|=\left|\widetilde{\Tr}_{L^2} \big[G_\ell(x-\cdot)\,\gamma\big]\right|\\
& = \left|\widetilde{\Tr}_{L^2} \big[|D^0|^{-1/2}G_\ell(x-\cdot)|D^0|^{-1/2}\,|D^0|^{1/2}\gamma|D^0|^{1/2}\big]\right| \\
&\leq \Vert |D^0|^{-1/2}|G_\ell(x-\cdot)||D^0|^{-1/2}\Vert_{Y}\,\Vert |D^0|^{1/2}\gamma|D^0|^{1/2}\,\Vert_{\mathfrak{S}_{1,1}} \leq C_H\, \|\gamma\|_{X}.
\end{aligned}
\]
Indeed, the bound \eqref{3.eq:4.1a} in Lemma \ref{3.lem:4.2} yields
\[
\left\||G_\ell(\cdot-x)|^{1/2}|D^0|^{-1/2}\right\|_Y\leq (C_H)^{\,1/2}
\]
uniformly in $x$. 

We now turn to the proof of \eqref{3.cor:4.4'}.
For every $\xi\in Q_\ell^*$ and $\varphi_\xi$ in $L^2_\xi$, we have
\begin{align}
\MoveEqLeft\left\Vert(\rho_{\gamma}* G_\ell)\,|D_\xi|^{-1} \varphi_\xi\right\Vert_{L^2_\xi}
 \leq \int_{Q_\ell} |\rho_{\gamma}(x)|\,\left\Vert G_\ell(\cdot-x)\,|D_\xi|^{-1} \varphi_\xi\right\Vert_{L^2_\xi}\,dx
\notag\\
& \qquad\qquad\leq \sup_{x\in \mathbb{R}^3}\left\Vert G_\ell(\cdot-x)\,|D_\xi|^{-1} \varphi_\xi\right\Vert_{L^2_\xi} \int_{Q_\ell}|\rho_{\gamma}(x)|\,dx \leq C_G\,\Vert \gamma\Vert_{\mathfrak{S}_{1,1}} \,\Vert\varphi_\xi\Vert_{L^2_\xi}.\label{eq:norme}
\end{align}
In \eqref{eq:norme}, we have used the bound \eqref{3.eq:4.3} of Lemma \ref{3.lem:4.2} and the obvious fact that it remains true for $G_\ell(\cdot-x)$ for any $x\in \mathbb{R}^3$.

\end{proof}
Now, we consider the operator $W_\gamma:=\fint^\oplus_{Q_\ell^*} W_{\gamma,\xi}\,d\xi$ which enters the definition of the exchange term. The operators $W_{\gamma,\xi}$ have been defined in Formula \eqref{eq:W-xi}, which involves the integral kernel $W_\ell^\infty$ given in \eqref{3.eq:2.8}. We can separate the singularities of $W_\ell^\infty$ with respect to $\eta\in 2Q_\ell^*$ and $x\in 2Q_\ell$ as follows 
\begin{equation}\label{3.eq:4.7'}
  W_\ell^\infty(\eta,x)=W_{\geq m,\ell}^\infty(\eta,x)+W_{<m,\ell}^\infty(\eta,x),\quad\forall\, m\in\mathbb{N}, m\geq 2,
\end{equation}
with 
\[
W_{\geq m,\ell}^\infty(\eta,x)=\frac{4\pi}{\ell^3}\sum_{\substack{|k|_\infty\geq m\\ k\in \mathbb{Z}^3}}\frac{1}{\left|\frac{2\pi k}{\ell}-\eta\right|^2}\,e^{i\left(\frac{2\pi k}{\ell}-\eta\right)\cdot x}
\]
and
\[
W_{<m,\ell}^\infty(\eta,x)=\frac{4\pi}{\ell^3}\sum_{\substack{|k|_\infty< m\\ k\in \mathbb{Z}^3}}\frac{1}{\left|\frac{2\pi k}{\ell}-\eta\right|^2}\,e^{i\left(\frac{2\pi k}{\ell}-\eta\right)\cdot x}
\]
where $|k|_\infty:=\max\{|k_1|;|k_2|;|k_3|\}$. It is easy to see that the singularity of $W_{<m,\ell}^\infty$ behaves like $\frac{1}{|\eta|^2}$. We will show in Appendix \ref{sec:B} that the singularity of $W_{\geq m,\ell}^\infty$ behaves like $\frac{1}{|x|}$ or equivalently $G_\ell(x)$, and we will obtain the following estimates on the operator $W_\gamma$. 
\begin{lemma}[Estimates on $W_{\gamma}$] \label{3.lem:4.5}
If $\gamma\in Z$, then $W_\gamma\in Y$ and there exist positive constants $C_W=C_W(\ell)$, $C_W'=C_W'(\ell)$, $C_W''=C_W''(\ell)$ that only depend on $\ell$, such that
\begin{align}\label{3.eq:4.6}
  \|W_{\gamma}\|_Y &\leq C_W\,\|\gamma\|_{Z} &{} &\text{if }\gamma\in Z,\\
\label{3.eq:4.10'}
\|W_{\gamma}\|_Y &\leq C''_W\,(\|\gamma\|_{X}+\|\gamma\|_{\mathfrak{S}_{1,\infty}}^{3/4}\|\gamma\|_{\mathfrak{S}_{1,1}}^{1/4})& &\text{if } \gamma\in X\cap \mathfrak{S}_{1,\infty},
\\
\label{3.eq:4.7}
  \|W_{\gamma}\,\,|D^0|^{-1}\|_Y &\leq C'_W\,\|\gamma\|_{\mathfrak{S}_{1,1}\cap Y} & &\text{if } \gamma\in \mathfrak{S}_{1,1}\cap Y.
\end{align}
\end{lemma}
\begin{remark}
The constants $C_W$, $C_W'$ and $C_{W}''$ are estimated in \eqref{eq:CW} in Appendix \ref{sec:B}. 
\end{remark}

Gathering together Lemmas \ref{3.lem:4.2}, \ref{3.lem:4.5}  and Corollary \ref{3.cor:4.4}, we can get some rough estimates on the self-consistent potential $V_{\gamma,\xi}$ defined in \eqref{eq:def-V}. In Appendix \ref{sec:C} we obtain much better estimates by a careful study of the structure of the operator $V_\gamma=\fint_{Q_\ell^*}^\oplus V_{\gamma,\xi}\,d\xi\,$: 

\begin{lemma}[Estimates on $V_{\gamma}$]\label{3.lem:5.1}There exist positive constants $C_{EE}=C_{EE}(\ell)>0$ and $C_{EE}'=C_{EE}'(\ell)>0$ that only depend on $\ell$ and such that, for every $\gamma\in Z$,
\begin{align}\label{3.eq:5.1'}
\|V_{\gamma}\|_Y \leq C_{EE}\,\|\gamma\|_{Z}
\end{align}
and
\begin{align}\label{3.eq:5.1}
\|V_{\gamma}\,|D^0|^{-1}\|_Y \leq C_{EE}'\,\|\gamma\|_{\mathfrak{S}_{1,1}\cap Y}.
\end{align}
For every $\xi\in Q_\ell^*$ and any $\psi_\xi \in H^{1/2}_\xi$,
\begin{align}\label{3.eq:5.1''}
  \left|\left(\psi_\xi,V_{\gamma,\xi}\psi_\xi\right)_{L^2_\xi}\right|\leq C_{EE}\|\gamma\|_{\mathfrak{S}_{1,1}\cap Y}\|\psi_\xi\|_{H^{1/2}_\xi}^2.
\end{align}
Furthermore, if $\gamma\geq 0$, for any $\psi\in L^2_\xi$, 
\begin{align}\label{3.eq:5.1'''}
   -C_{EE}''\|\gamma\|_{\mathfrak{S}_{1,1}\cap Y}\|\psi_\xi\|_{L^2_\xi}^2\leq \left(\psi_\xi,V_{\gamma,\xi}\psi_\xi\right)_{L^2_\xi}.
\end{align}
\end{lemma}

\begin{remark}
The constants $C_{EE}$, $C_{EE}'$ and $C_{EE}''$ are estimated in \eqref{eq:3.43}, \eqref{eq:3.41} and \eqref{eq:3.44} of Appendix \ref{sec:C} respectively. 
\end{remark}
\begin{remark}\label{rem:well-define}
  Using Lemmas \ref{3.lem:4.2} and  \ref{3.lem:5.1}, it is easily checked that $Z \ni \gamma \mapsto \mathcal{E}(\gamma)$ is well-defined.
\end{remark}

\subsection{Spectral properties of the mean-field Dirac--Fock operator}\label{ssec:spectral}

Recall that $\kappa:=\alpha\,\big(C_G z+C_{EE}'q^+\big)$. We start with the following.
\begin{lemma}\label{3.lem:self-adj}
Let $\gamma\in Z$. We assume that $C_G z+C_{EE}'\|\gamma\|_{\mathfrak{S}_{1,1}\cap Y}<1/\alpha$, then $D_{\gamma,\xi}$ is a self-adjoint operator on $L^2_\xi$ with domain $H^1_\xi$ and form-domain $H^{1/2}_\xi$. In addition, the following holds 
\begin{equation}\label{eq:borneDgD-1}
\left\||D_{\gamma}|^{1/2}|D^0|^{-1/2}\right\|_Y\leq \left(1+\alpha\,\big(C_G z+C_{EE}'\|\gamma\|_{\mathfrak{S}_{1,1}\cap Y}\big)\right)^{1/2}
\end{equation}
and 
\begin{equation}\label{eq:borneD-1Dg}
\left\||D^0|^{1/2}|D_{\gamma}|^{-1/2}\right\|_Y\leq \left(1-\alpha\,\big(C_G z+C_{EE}'\|\gamma\|_{\mathfrak{S}_{1,1}\cap Y}\big)\right)^{-1/2}.
\end{equation}
In particular, if $\gamma\in \Gamma_{\leq q}$ and $\kappa<1$, we have
\begin{equation}\label{eq:gap}
(1-\kappa)\,|D^0|\leq |D_{\gamma}|\leq (1+\kappa)\,|D^0|.
\end{equation}
\end{lemma}

\begin{proof} Recall $q^+=\max\{1;q\}$.
By Lemmas \ref{3.lem:4.2} and  \ref{3.lem:5.1}, we obtain
\begin{equation}\label{eq:bornepotD-1}
\|(-\alpha\,z\,G_\ell+\alpha\,V_{\gamma})\,|D^0|^{-1}\|_Y\leq  \alpha\,\big(\,C_G z+C_{EE}'\,\|\gamma\|_{\mathfrak{S}_{1,1}\cap Y}\big).
\end{equation}
In particular, $D_{\gamma}$ is self-adjoint on $\fint_{Q_\ell^*}^\oplus H^1_\xi\,d\xi$ by the Rellich--Kato theorem if $C_G z+C_{EE}'\|\gamma\|_{\mathfrak{S}_{1,1}\cap Y} < 1/\alpha$ (see \cite[Theorem XIII-85]{reed1978b}). Let now $\xi\in Q_\ell^*$ and $u_\xi\in H^1_\xi(Q_\ell)$. We have 
\begin{equation}\label{3.eq:5.3'}
  \|D_{\gamma,\xi}\,u_\xi\|_{L^2_\xi}\leq \left(1+\alpha\,C_G z+\alpha C_{EE}'\|\gamma\|_{\mathfrak{S}_{1,1}\cap Y}\right)\,\|D_\xi\, u_\xi\|_{L^2_\xi},
\end{equation}
which implies \eqref{eq:borneDgD-1}. On the other hand, 
\begin{align*}
\|D_\xi\,u_\xi\|_{L^2_\xi}&\leq \|(D_{\gamma,\xi}-D_\xi)\, u_\xi\|_{L^2_\xi} + \|D_{\gamma,\xi} u_\xi\|_{L^2_\xi}\\
&\leq \alpha\,\left(C_G z+C_{EE}'q^+\right)\,\|D_\xi\,u_\xi\|_{L^2_\xi}+\|D_{\gamma,\xi}\,u_\xi\|_{L^2_\xi}.
\end{align*}
Hence,
\begin{equation}\label{eq:normeDDg}
\|D_\xi \,u_\xi\|_{L^2_\xi}\leq (1-\alpha(C_G z+C_{EE}'\|\gamma\|_{\mathfrak{S}_{1,1}\cap Y}))^{-1}\|D_{\gamma,\xi}\, u_\xi\|_{L^2_\xi}
\end{equation}
which implies \eqref{eq:borneD-1Dg}. Since $\gamma\in \Gamma_{\leq q}$, $\|\gamma\|_{\mathfrak{S}_{1,1}\cap Y}\leq q^+$. Thus \eqref{3.eq:5.3'} and \eqref{eq:normeDDg} together give \eqref{eq:gap}. This concludes the proof.
\end{proof}
As a consequence of \eqref{eq:normeDDg}, we deduce that the spectrum of $D_{\gamma}$ (and of any $D_{\gamma,\xi}$) is included in $\mathbb{R}\setminus[-1+\kappa;1-\kappa]$. 
In order to allow for as many electrons as possible per cell, we need a more accurate estimate on the bottom of $\sigma(|D_{\gamma}|)$. 
\begin{lemma}[Further properties of the bottom of the spectrum of $|D_\gamma|$]\label{lem:bottom}
Let $\gamma\in \Gamma_{\leq q}$. Then
\[
\inf\sigma(|D_{\gamma}|)\geq \lambda_0\geq 1-\kappa,
\]
with $\lambda_0:=1-\alpha \max\{C_H z+C_{EE}''q^+; \frac{C_0}{\ell}z+C_{EE}q^+\}$, the constant $C_0$ being defined in Lemma \ref{lem:Gbound} in Appendix A~below.
\end{lemma}
\begin{proof}
Let $\psi^+_\xi=\Lambda^+_{\xi}\psi_\xi$ and $\psi^-_\xi=\Lambda^-_{\xi}\psi_\xi$. Notice that $D_{\gamma,\xi}=D_{\xi}-\alpha zG_\ell+\alpha V_{\gamma,\xi}$ and $V_{\gamma,\xi}$ satisfies \eqref{3.eq:5.1''} and \eqref{3.eq:5.1'''}. Now, combining with \eqref{3.eq:4.1} in Appendix \ref{sec:A}, we have
\begin{align*}
  \MoveEqLeft \left(\psi^+_\xi, D_{\gamma,\xi}\psi^+_\xi\right)_{H^{1/2}_\xi\times H^{-1/2}_\xi}\geq \left(1-\alpha (C_H z+C_{EE}''\|\gamma\|_{\mathfrak{S}_{1,1}\cap Y})\right)\|\psi^+_\xi\|_{H^{1/2}_\xi}^2
\end{align*}
and
\begin{align*}
  \MoveEqLeft -\left(\psi^-_\xi, D_{\gamma,\xi}\psi^-_\xi\right)_{H^{1/2}_\xi\times H^{-1/2}_\xi}\geq \left(1-\alpha \left(\frac{C_0 }{\ell}z+C_{EE}\|\gamma\|_{\mathfrak{S}_{1,1}\cap Y}\right)\right)\|\psi^+_\xi\|_{H^{1/2}_\xi}^2.
\end{align*}
We get 
\begin{align*}
\MoveEqLeft
 \|\psi_\xi\|_{H_\xi^{1/2}}\|D_{\gamma,\xi}\psi\|_{H^{-1/2}_\xi}\geq \Re\left(\psi_\xi^+-\psi_\xi^-,D_{\gamma,\xi}\psi_\xi\right)_{H^{1/2}_\xi\times H^{-1/2}_{\xi}}\\
  &\qquad\qquad=\left(\psi_\xi^+,D_{\gamma,\xi}\psi_\xi^+\right)_{H^{1/2}_\xi\times H^{-1/2}_{\xi}}-\left(\psi_\xi^-,D_{\gamma,\xi}\psi_\xi^-\right)_{H^{1/2}_\xi\times H^{-1/2}_{\xi}} \geq\lambda_0\|\psi_\xi\|^2_{H^{1/2}_\xi}.
\end{align*}
\end{proof}

Further spectral properties of the self-consistent operator $D_\gamma$ are collected in the following.

\begin{lemma}[Properties of positive eigenvalues of $D_{\gamma,\xi}$]\label{3.lem:7.1}
Assume that $\kappa<1$ and let $\gamma\in \Gamma_{\leq q}$. We denote by $\lambda_j(\xi)$, for $j\geq 1$, the $j$-th positive eigenvalue (counted with multiplicity) of the mean-field operator $D_{\gamma,\xi}$. Then $\lambda_j(\xi)$ is situated in the interval $[c_*(j)(1-\kappa),c^*(j)(1-\kappa)^{-1}]$ where $c^*(j)$ and $c_*(j)$ are the constants of Formula \eqref{bornes}.

In addition, every eigenfunction $u_{j,\xi}$ associated to $\lambda_j(\xi)$ lies in $H^1_\xi$ and satisfies 
\begin{equation}\label{eq:bd-eigenf}
\big\| u_{j,\xi}\big\|_{H^1_\xi}\leq c^*(j)\,(1-\kappa)^{-2}\,\Vert u_{j,\xi}\Vert_{L^2_\xi}.
\end{equation}
\end{lemma}

\begin{proof} We rely on a variational characterization of eigenvalues in spectral gaps (see \cite{dolbeault2023extension} and references therein). Let 
\[
\Lambda_\xi^+:=\mathbbm{1}_{\mathbb{R}^+}(D_\xi)=\frac{1}{2}+\frac{D_\xi}{2\,|D_\xi|}
\]
and
\[
\Lambda_\xi^-:=\mathbbm{1}_{\mathbb{R}^-}(D_\xi)=\frac{1}{2}-\frac{D_\xi}{2\,|D_\xi|}.
\] 
One has $\Lambda_\xi^\pm H^1_\xi\subset H^1_\xi$, that is, the domain $H^1_\xi$ of the self-adjoint operator $D_{\gamma,\xi}$ satisfies Condition (H1) of \cite{dolbeault2023extension}. To each integer $j\geq 0$ we associate the min-max level
\begin{equation}\label{3.eq:7.0}
   \hat{\lambda}_j(\xi):=\inf_{\substack{V\,\textrm{subspace of}\, \Lambda_\xi^+ H_\xi^{1}\\ \dim V=j}}\;\sup_{u_\xi\in (V\bigoplus \Lambda_\xi^-H_\xi^{1})\setminus \{0\}}\frac{\left(D_{\gamma,\xi}\,u_\xi,u_\xi\right)}{\|u_\xi\|_{L^2_\xi}^2}.
\end{equation}
Let $u_\xi\in (V\bigoplus \Lambda_\xi^-H_{\xi}^{1})\setminus\{0\}$. We write $u_\xi=u_\xi^++u_\xi^-$ with 
\[
u^+_\xi=\Lambda_\xi^+ u_\xi\in V,\quad u_\xi^-=\Lambda_\xi^- u_\xi\in \Lambda_\xi^-H_\xi^{1}.
\]
By definition of $\Lambda_\xi^\pm$, 
\[
(D_\xi u_\xi^+,u_\xi^+)=(|D_\xi| u_\xi^+,u_\xi^+), \quad (D_\xi u_\xi^-,u_\xi^-)=-(|D_\xi| u_\xi^-,u_\xi^-) \text{ and } (D_\xi u_\xi^+,u_\xi^-)=0.
\]
Therefore, 
\begin{align}
\MoveEqLeft\left(D_{\gamma,\xi}u_\xi,u_\xi\right)=\left(D_{\xi}u_\xi,u_\xi\right)+\left((D_{\gamma,\xi}-D_{\xi})u_\xi,u_\xi\right)\notag\\
&= \left(|D_{\xi}|\,u_\xi^+,u_\xi^+\right)-\left(|D_{\xi}|\,u_\xi^-,u_\xi^-\right
)+\left((D_{\gamma,\xi}-D_\xi)\,u_\xi^+,u_\xi^+\right)+\left((D_{\gamma,\xi}-D_\xi)\,u_\xi^-,u_\xi^-\right)\notag\\
&\quad + 2\Re \left((D_{\gamma,\xi}-D_{\xi})u_\xi^+,u_\xi^-\right).\label{eq:decompDg}
\end{align}
We observe that for $j\geq 1$,
\[
\hat{\lambda}_j(\xi)\geq \inf_{\substack{V\,\textrm{subspace of}\, \Lambda_\xi^+ H_\xi^{1}\\ \dim V=j}}\;\sup_{u^+_\xi\in V\setminus \{0\}}\frac{\left(D_{\gamma,\xi}\,u^+_\xi,u^+_\xi\right)}{\|u^+_\xi\|_{L^2_\xi}^2}.
\]
By \eqref{eq:decompDg}, \eqref{3.eq:4.3} in Lemma \ref{3.lem:4.2} and \eqref{3.eq:5.1} in Lemma \ref{3.lem:5.1}, for any $u^+_\xi\in \Lambda_\xi^+ H_\xi^{1}$,
\[
\left(D_{\gamma,\xi}\,u^+_\xi,u^+_\xi\right)=\left(|D_{\xi}|\,u^+_\xi,u^+_\xi\right)+\left((-\alpha\,z\,G_\ell+\alpha\,V_{\gamma,\xi})\,u^+_\xi,u^+_\xi\right)\geq (1-\kappa)\left(|D_{\xi}|\,u^+_\xi,u^+_\xi\right)
.\]
Thus,
\begin{align*}
\hat{\lambda}_j(\xi)&\geq (1-\kappa)\inf_{\substack{V\,\textrm{subspace of}\, \Lambda_\xi^+ H_\xi^{1}\\ \dim V=j}}\;\sup_{u^+_\xi\in V\setminus \{0\}}\frac{\left(|D_{\xi}|u^+_\xi,u^+_\xi\right)}{\|u^+_\xi\|_{L^2_\xi}^2}\\&=(1-\kappa)d^+_j(\xi)\geq (1-\kappa)c_*(j).
\end{align*}
In particular, $\hat{\lambda}_1(\xi)\geq (1-\kappa)c_*(1)>0$.
On the other hand, $
(D_{\gamma,\xi}\,u^-_\xi,u^-_\xi)\leq -(1-\kappa)(\vert D_{\xi}\vert\,u^-_\xi,u^-_\xi)\leq 0$ 
for every $u^-_\xi \in \Lambda_\xi^- H^{1}_\xi$, whenever $\kappa<1$. This implies that $\hat{\lambda}_0(\xi)\leq 0$, so Conditions (H2) and (H3) of \cite{dolbeault2023extension} are satisfied, and we can conclude that
\begin{equation}\label{minmax}
  \lambda_j(\xi)=\hat{\lambda}_j(\xi)\;,\quad \forall j\geq 1,
\end{equation}
hence the lower bound $\lambda_j(\xi)\geq (1-\kappa)c_*(j)$.\medskip

For the upper bound, we proceed as follows. Equations \eqref{eq:bornepotD-1} and \eqref{eq:decompDg} together yield
\begin{align*}
\MoveEqLeft\left(D_{\gamma,\xi}u_\xi,u_\xi\right)
= \left(|D_{\xi}|\,u_\xi^+,u_\xi^+\right)+\left((-\alpha\,z\,G_\ell+\alpha\,V_{\gamma,\xi})\,u_\xi^+,u_\xi^+\right)+2\Re \left((-\alpha\,z\,G_\ell+\alpha\,V_{\gamma,\xi})\,u_\xi^+,u_\xi^-\right)\\
&\quad +\left((D_{\gamma,\xi}-D_\xi)\,u_\xi^-,u_\xi^-\right)-\left(|D_{\xi}|\,u_\xi^-,u_\xi^-\right)\\
&\leq(1+\kappa)\,\left(|D_{\xi}|\,u_\xi^+,u_\xi^+\right)-(1-\kappa)\,\left(|D_{\xi}|\,u_\xi^-,u_\xi^-\right)+2\,\kappa\,\Vert |D_\xi|^{1/2}u_\xi^+\Vert_{L^2_\xi}\,\||D_\xi|^{1/2}u_\xi^-\|_{L^2_\xi}\\
&=(1+\kappa)\||D_\xi|^{1/2}u_\xi^+\|_{L^2_\xi}^2+2\,\kappa\,\Vert |D_\xi|^{1/2}u_\xi^+\Vert_{L^2_\xi}\,\||D_\xi|^{1/2}u_\xi^-\|_{L^2_\xi}-(1-\kappa)\||D_\xi|^{1/2}u_\xi^-\|_{L^2_\xi}^2\\
&\leq (1-\kappa)^{-1}\||D_\xi|^{1/2}u_\xi^+\|_{L^2_\xi}^2
\end{align*}
by Young's inequality. Moreover, $\|u^+_\xi\|_{L^2_\xi}\leq\|u_\xi\|_{L^2_\xi}$, so, recalling \eqref{3.eq:7.0} and \eqref{minmax} we see that
\begin{align*}
\lambda_j(\xi)&\leq (1-\kappa)^{-1}\inf_{\substack{V\,\textrm{subspace of}\, \Lambda_\xi^+ H_\xi^{1}\\ \dim V=j}}\;\sup_{u^+_\xi\in V\setminus \{0\}}\frac{\left(|D_{\xi}|u^+_\xi,u^+_\xi\right)}{\|u^+_\xi\|_{L^2_\xi}^2}\\
&=(1-\kappa)^{-1}d^+_j(\xi)\leq (1-\kappa)^{-1}c^*(j).
\end{align*}
Finally, using \eqref{eq:normeDDg} in Lemma \ref{3.lem:self-adj}, we obtain 
\[
\lambda_{j}(\xi)\,\Vert u_{j,\xi}\|_{L^2_\xi}=\|D_{\gamma,\xi}u_{j,\xi}\|_{L^2_\xi}\geq (1-\kappa)\|D_{\xi}u_{j,\xi}\|_{L^2_\xi}=(1-\kappa)\,\|u_{j,\xi}\|_{H^1_\xi},
\]
hence \eqref{eq:bd-eigenf}. The lemma is thus proved.

\end{proof}

Recall that ${\bf{\overline B}}:=\Set*{\gamma\in Z \,\given \;\|\gamma\|_{\mathfrak{S}_{1,\infty}}\leq j_1}$ where $j_1$ is an integer that has not been defined yet. In the rest of this paper, assuming that $\kappa<1$ and recalling our notation $e_0:=(1-\kappa)^{-1}c^*(\lceil q \rceil)\,$, we take
\begin{equation}\label{choice}
  j_1:=\min\{j\geq 0\;\vert\; (1-\kappa)c_*(j+1)>e_0\}.
\end{equation}
This integer is well-defined, since $\lim_{j\to\infty}c_*(j)=+\infty$. We also introduce the energy level
\begin{equation}\label{level1}
   e_1:=(1-\kappa)c_*(j_1+1).
\end{equation}
By construction, one has $0<e_0<e_1$ an $j_1\geq \lceil q \rceil $. Moreover, Lemma \ref{3.lem:7.1} has an immediate consequence, which will be very useful in the sequel.

\begin{corollary}\label{Bbar}
  Assuming that $\kappa<1$, with the above notation, for every $\gamma$ in $\Gamma_{\leq q}$:\medskip
  
  $\bullet$ The projector $\mathbbm{1}_{[0,e_0]}(D_{\gamma,\xi})$ has rank at least $\lceil q \rceil$ for a.e. $\xi\in Q_\ell^*$.\medskip

 $\bullet$ The projector $\mathbbm{1}_{[0,e_1)}(D_{\gamma,\xi})$ has rank at most $j_1$ for a.e. $\xi\in Q_\ell^*$.\medskip
 
 $\bullet$ If $\gamma'\in Z$ and $0\leq \gamma'\leq\mathbbm{1}_{[0,e_1)}(D_{\gamma})$, then $\gamma'\in {\bf{\overline B}}$. 
\end{corollary}

We end this subsection with the following proposition.
\begin{proposition}\label{3.prop:5.8}
Assume that $\kappa<1$. Let $\gamma,\gamma'\in\Gamma_{\leq q}$ such that 
\[
0\leq \gamma'\leq \mathbbm{1}_{[0,\nu]}(D_{\gamma})
\]
with $\nu>0$. Then,
\[
\|\gamma'\|_{Z}\leq \max\{(1-\kappa)^{-1}\,q\,\nu;1\}.
\]
\end{proposition}
\begin{proof}
By Lemma \ref{3.lem:self-adj}, we have
\begin{align*}
\fint_{Q_\ell^*}\Tr_{L^2_\xi}[D_{\gamma,\xi}\gamma'_\xi]\,d\xi&=\fint_{Q_\ell^*}\Tr_{L^2_\xi}[ |D_{\gamma,\xi}|\gamma'_\xi ]\,d\xi\geq (1-\kappa) \|\gamma'\|_X.
\end{align*}
Since $\gamma'\in\Gamma_{\leq q}$, we obtain
\[
\fint_{Q_\ell^*}\Tr_{L^2_\xi}[ D_{\gamma,\xi}\gamma'_\xi]\,d\xi\leq \nu \fint_{Q_\ell^*}\Tr_{L^2_\xi}[\gamma'_\xi] \,d\xi\leq q\,\nu.
\]
Then $\|\gamma'\|_X\leq (1-\kappa)^{-1}q\,\nu$. We deduce the desired bound since $\|\gamma'\|_Y\leq 1$.
\end{proof}
\subsection{Properties of the minimizers for a linearized problem}\label{sec:5.1}

The following lemma will be used in the next sections. 
\begin{lemma}\label{3.lem:6.4} Let $g\in \Gamma_{\leq q}$ be given, and assume $\kappa<1$. Then for each $\epsilon_P>0$, the minimization problem
\[
\inf_{\substack{\gamma\in \Gamma_{\leq q}\\\gamma=P_{g}^+\gamma P_{g}^+}}\fint_{Q_\ell^*}\Tr_{L^2_\xi}[(D_{g,\xi}-\epsilon_{P})\gamma_{\xi}]\,d\xi
\]
admits a minimizer. Every minimizer $\hat{g}$ is of the form $\hat{g}=\fint_{Q_\ell^*}^\oplus\mathbbm{1}_{[0,\nu)}(D_{g,\xi})\,d\xi+\hat{\delta}$, with $0\leq \hat{\delta} \leq \fint_{Q_\ell^*}^\oplus\mathbbm{1}_{\{\nu\}}(D_{g,\xi})\,d\xi$ for some $\nu\in (0,\min(\epsilon_{P},e_0)]$, and one has $\hat{g}\in {\bf{\overline B}}$. 

If $\epsilon_{P}> e_0$, the set of minimizers is independent of $\epsilon_{P}$, and every minimizer satisfies $\widetilde{\Tr}_{L^2}(\hat{g})=q$ and $\nu\geq \lambda_0$.
\end{lemma}
\begin{proof}
 The proof is inspired of \cite{cances2008new}. For any $\xi\in Q_\ell^*$ we can choose an orthonormal eigenbasis $\{\psi_j(\xi,\cdot)\}_{j\geq 1}$ of $D_{g,\xi}P^+_{g,\xi}$, such that
\[
D_{g,\xi}P^+_{g,\xi}=\sum_{j\geq 1}\lambda_j(\xi)\left|\psi_j(\xi)\right>\left<\psi_j(\xi)\right|,
\]
with $\lambda_j(\xi)\geq 0$. According to Lemma \ref{3.lem:7.1}, for every $\xi\in Q_\ell^*$ and for every $j\geq 1$
\[
(1-\kappa)\, c_*(j)\leq \lambda_j(\xi)\leq (1-\kappa)^{-1}\, c^*(j).\]
Let us introduce as in \cite{ghimenti2009properties,cances2008new} the non-decreasing function
\[
C: \mathbb{R}_+ \ni s \mapsto \frac{\ell^3}{(2\pi)^3}\sum_{j\geq 1}\left|\Set*{\xi\in Q_\ell^*\given 0\leq \lambda_j(\xi)\leq s}\right|.
\]
By Lemma \ref{3.lem:7.1}, for $0\leq s<(1-\kappa)c_*(\lceil q\rceil)$ one has $C(t)\leq \lceil q\rceil-1 < q$. On the other hand, $C(e_0)\geq \lceil q \rceil\geq q$. Thus, there exists $\nu_1$ with 
\begin{equation}\label{eq:bd-nu1}
1-\kappa\leq (1-\kappa)c_*(\lceil q\rceil)\leq \nu_1 \leq e_0
\end{equation} 
such that

\begin{equation}\label{3.eq:6.2}
  \lim_{s\to \nu_1^-}C(s)\leq q \leq \lim_{s\to \nu_1^+} C(s).
\end{equation}
Equivalently, 
\[
\fint_{Q_\ell^*}\Tr_{L^2_\xi}[\mathbbm{1}_{[0,\nu_1)}(D_{g,\xi})]\,d\xi\leq q
\]
and 
\[
\fint_{Q_\ell^*}\Tr_{L^2_\xi}[\mathbbm{1}_{[0,\nu_1]}(D_{g,\xi})]\,d\xi\geq q.
\]
Therefore, there exists $ 0\leq \widetilde{\delta}\leq \fint_{Q_\ell^*}^\oplus\mathbbm{1}_{\{\nu_1\}}(D_{g,\xi})\,d\xi$, such that the density matrix 
\[\widetilde{g}:=\fint_{Q_\ell^*}^\oplus \mathbbm{1}_{[0,\nu_1)}(D_{g,\xi})\,d\xi+\widetilde{\delta}
\]
satisfies
 \[
\widetilde{\Tr}_{L^2}(\widetilde{g})=q.
\]
We first consider the case $\nu_1< \epsilon_{P}$. For any $\gamma\in \Gamma_{\leq q}$ with $\gamma =P_g^+\gamma P_g^+$, we write 
\begin{align}\label{eq:4.27}
\MoveEqLeft\fint_{Q_\ell^*}\Tr_{L^2_\xi}[(D_{g,\xi}-\epsilon_{P})(\gamma_{\xi}-\widetilde{g}_{\xi})]\,d\xi\notag\\
=&\fint_{Q_\ell^*}\Tr_{L^2_\xi}[(D_{g,\xi}-\nu_1)(\gamma_{\xi}-\widetilde{g}_{\xi})]\,d\xi+\fint_{Q_\ell^*}\Tr_{L^2_\xi}[(\nu_1-\epsilon_{P})(\gamma_{\xi}-\widetilde{g}_{\xi})]\,d\xi\notag\\
=&\fint_{Q_\ell^*}\Tr_{L^2_\xi}[(D_{g,\xi}-\nu_1)(\gamma_{\xi}-\widetilde{g}_{\xi})]\,d\xi+(\nu_1-\epsilon_{P})\left( \widetilde{\Tr}_{L^2}(\gamma)-q\right)\notag\\
\geq&\fint_{Q_\ell^*}\Tr_{L^2_\xi}[(D_{g,\xi}-\nu_1)(\gamma_{\xi}-\widetilde{g}_{\xi})]\,d\xi.
\end{align}
On the other hand, we have 
\begin{align}\label{3.eq:6.10}
\MoveEqLeft\fint_{Q_\ell^*}\Tr_{L^2_\xi}[(D_{g,\xi}-\nu_1)(\gamma_{\xi}-\widetilde{g}_{\xi})]\,d\xi\notag\\ &=\fint_{Q_\ell^*}\Tr_{L^2_\xi}[(D_{g,\xi}-\nu_1)(\gamma_{\xi}-\mathbbm{1}_{[0,\nu_1)}(D_{g,\xi}))]\,d\xi\notag\\
&=\fint_{Q_\ell^*}\sum_{\lambda_j(\xi)<\nu_1}\left|\lambda_j(\xi)-\nu_1\right|\left|\left(\gamma_{\xi}\psi_j(\xi),\psi_j(\xi)\right)-1\right|\,d\xi\notag\\
&\qquad\qquad \quad {}+\fint_{Q_\ell^*}\sum_{\lambda_j(\xi)>\nu_1}(\lambda_j(\xi)-\nu_1)\left(\gamma_{\xi}\psi_j(\xi),\psi_j(\xi)\right)\,d\xi\geq 0,
\end{align}
since $0\leq \gamma\leq \mathbbm{1}_{L^2(\mathbb{R}^3)}$. Thus $\widetilde{g}$ is a minimizer. Conversely, if $\hat{g}$ is a minimizer, then it must be of the form $\fint_{Q_\ell^*}^\oplus \mathbbm{1}_{[0,\nu)}(D_{g,\xi})\,d\xi+\hat{\delta}$ with $ 0\leq \hat{\delta}\leq \fint_{Q_\ell^*}^\oplus\mathbbm{1}_{\{\nu_1\}}(D_{g,\xi})d\xi$, $\,\nu=\nu_1\in [\lambda_0,\min\{\epsilon_P,e_0\}]$ and $\widetilde{\Tr}_{L^2}(\hat{g})=q$ since all inequalities in \eqref{eq:4.27} and \eqref{3.eq:6.10} above have to be equalities for $g=\hat{g}$.\medskip 

For the case $\epsilon_{P}\leq \nu_1$, we prove that $g':=
\fint_{Q_\ell^*}^\oplus\Tr_{L^2_\xi}[\mathbbm{1}_{[0,\epsilon_P)}(D_{g,\xi})]\,d\xi
$ is a minimizer with $\widetilde{\Tr}_{L^2}(g')\leq q$, thanks to a modified version of \eqref{3.eq:6.10} with $\nu_1$ (resp. $\widetilde{g}$) being replaced by $\epsilon_P$ (resp. $g'$). As in the previous case, every minimizer $\hat{g}$ satisfies $\hat{g}=\fint_{Q_\ell^*}^\oplus\mathbbm{1}_{[0,\epsilon_{P})}(D_{g,\xi})\,d\xi+\hat{\delta}$, with $0\leq \hat{\delta} \leq \fint_{Q_\ell^*}^\oplus\mathbbm{1}_{\{\epsilon_{P}\}}(D_{g,\xi})\,d\xi$. Note that in the case $\epsilon_p\leq \nu_1$, the inequality $\widetilde{\Tr}_{L^2}(\hat{g})\leq q$ automatically holds for any such $\hat{g}$.\medskip

In both cases, thanks to \eqref{eq:bd-nu1} we have $\nu=\min\{\nu_1,\epsilon_P\}\leq e_0<e_1$, hence
$0\leq \hat{g}\leq \mathbbm{1}_{[0,e_1)}(D_{g})\,$.
Thus, Corollary \ref{Bbar} implies that $\hat{g}\in {\bf{\overline B}}$. 
\end{proof}
\begin{remark}\label{rem:4.13}
Actually, it follows from Corollary \ref{Bbar} that for every minimizer $\hat{g}$ and a.e. $\xi\in Q_\ell^*$, one has $ {\rm Rank}(\hat{g}_{\xi})\leq j_1$.
\end{remark}

\subsection{First properties of minimizing sequences for \texorpdfstring{$J_{\leq q}$}{}}
We introduce the sublevel set 
\begin{align}\label{eq:calS}
  \mathcal{S}:=\Set*{\gamma\in\Gamma_{\leq q}^+\given\mathcal{E}^{DF}(\gamma)-\epsilon_P\widetilde{\Tr}_{L^2}(\gamma)\leq 0}.
\end{align}
Note that the operator $0$ belongs to $\Gamma_{\leq q}^+$ and satisfies $\mathcal{E}^{DF}(0)-\epsilon_{P}\widetilde{\Tr}_{L^2}(0)=0$. Thus,
\begin{equation}\label{inf_in_S}
J_{\leq q}=\inf_{\gamma\in \mathcal{S}}\left[\mathcal{E}^{DF}(\gamma)-\epsilon_{P}\widetilde{\Tr}_{L^2}(\gamma)\right]
\end{equation}
and from now on we will only consider minimizing sequences for $J_{\leq q}$ lying in the sublevel set $\mathcal{S}$. These sequences satisfy {\it a priori} estimates gathered in the following lemma.
\begin{lemma}[Boundedness of $\mathcal S$]\label{3.lem:6.1}
Assume that $\kappa<1$. If $\kappa<1-\frac{\alpha}{2} C_{EE}q^+$, then, for every $\gamma\in \mathcal{S}$,
\begin{equation}\label{3.eq:5.11}
  \|\gamma\|_{Z}\leq \max\left\lbrace(1-\kappa-\frac{\alpha}{2} C_{EE}q^+)^{-1}\epsilon_{P} \,q;1\right\rbrace
\end{equation}
and
\begin{equation}\label{3.eq:5.11bis}
\max\left\{\|\gamma|D^0|^{1/2}\|_{\mathfrak{S}_{1,1}};\|\gamma\|_{Y}\right\}\leq \sqrt{\max\left\lbrace(1-\kappa-\frac{\alpha}{2} C_{EE}q^+)^{-1}\epsilon_{P} \,q;1\right \rbrace \,q^+}.
\end{equation}
\end{lemma}
\begin{proof}

As $D_{\gamma}\gamma=|D_{\gamma}|\gamma$ for any $\gamma\in \Gamma_{\leq q}^+$, we get, by \eqref{3.eq:5.1''} and \eqref{eq:gap}, 
\[
\begin{aligned}
\mathcal{E}^{DF}(\gamma)-\epsilon_{P}\widetilde{\Tr}_{L^2}(\gamma)&=\widetilde{\Tr}_{L^2}[(D_{\gamma,}-\epsilon_{P}-\frac{\alpha}{2} V_{\gamma})\gamma]= \widetilde{\Tr}_{L^2}[(|D_{\gamma}|-\epsilon_{P}-\frac{\alpha}{2} V_{\gamma})\gamma]\\
&\geq \widetilde{\Tr}_{L^2}[((1-\kappa)|D^0| -\epsilon_{P}-\frac{\alpha}{2} V_{\gamma})\gamma]\\
& \geq (1-\kappa)\|\gamma\|_{X}-\frac{\alpha}{2} C_{EE} \|\gamma\|_{\mathfrak{S}_{1,1}\cap Y} \|\gamma\|_X-\epsilon_{P}\|\gamma\|_{\mathfrak{S}_{1,1}}\\
&\geq (1-\kappa-\frac{\alpha}{2} C_{EE}q^+)\|\gamma\|_{X}-\epsilon_{P}\,q.
\end{aligned}
\]
Hence, for any $\gamma\in \mathcal{S}$,
\[
(1-\kappa-\frac{\alpha}{2} C_{EE}q^+)\|\gamma\|_{X}-\epsilon_{P}\,q\leq 0.
\]
Whenever $1-\kappa-\frac{\alpha}{2} C_{EE}q^+>0$,
\eqref{3.eq:5.11} holds since $\|\gamma\|_Y\leq 1$.
\medskip

The estimate \eqref{3.eq:5.11bis} follows from H\"older's inequality and the fact that $\gamma\geq 0$; namely 
\begin{align*}
  \MoveEqLeft\|\gamma\,|D^0|^{1/2}\|_{\mathfrak{S}_{1,1}}\leq \|\gamma^{1/2}\|_{\mathfrak{S}_{2,2}}\,\|\gamma^{1/2}\,|D^0|^{1/2}\|_{\mathfrak{S}_{2,2}}\leq \|\gamma\|_{\mathfrak{S}_{1,1}}^{1/2}\,\|\gamma\|_X^{1/2}.
\end{align*}
\end{proof}

\section{Approximation by a linearized problem}\label{4.sec:5}
The aim of this section is to show the link between a minimizing sequence $(\gamma_n)_{n\geq 1}$ in $\mathcal{S}$ and the linearized Dirac--Fock problem introduced in Lemma \ref{3.lem:6.4}.
\begin{proposition}[Link with the linearized problem]\label{3.lem:6.3} Under Assumption \ref{3.ass:2.1'}, let $(\gamma_{n})_{n\geq 1}\in \mathcal{S}^{\mathbb{N}^*}$ be a minimizing sequence for $J_{\leq q}$. Then, as $n$ goes to infinity, 
\begin{equation}\label{3.eq:7.1}
\fint_{Q_\ell^*}\Tr_{L^2_\xi}\big[(D_{\gamma_n,\xi}-\epsilon_{P})\gamma_{n,\xi}\big]\,d\xi-\inf_{\substack{\gamma\in \Gamma_{\leq q}\\\gamma=P_{\gamma_n}^+\gamma P_{\gamma_n}^+}}\fint_{Q_\ell^*}\Tr_{L^2_\xi}\big[(D_{\gamma_n,\xi}-\epsilon_{P})\gamma_{\xi}\big]\,d\xi\to 0.
\end{equation}
\end{proposition}
This property is used in Lemma \ref{3.lem:6.5} below to build a new minimizing sequence with further regularity, and it is also used at the end of Section \ref{3.sec:5} to show some properties of the minimizers for $J_{\leq q}$. 

As mentioned at the end of Section \ref{sec:3}, the main difficulty in the proof of Proposition \ref{3.lem:6.3} is to deal with the nonconvex constraint set $\Gamma_{\leq q}^+$. To do so, we adapt to our setting a retraction technique introduced in \cite{Ser09}. We are going to build an open subset $\mathcal{U}$ of $Z$ stable under the continuous map $T:\,\gamma\mapsto P^+_{\gamma}\gamma P^+_{\gamma}$ and such that the sequence \big($T^p\big)_{p\geq 1}$ converges uniformly on $\overline{\mathcal U}$ to a surjective map $\theta:\,\overline{\mathcal U}\to \overline{\mathcal U}\cap {\rm Fix}(T)$. Here, $\overline{\mathcal U}$ is the closure of ${\mathcal U}$ in $Z$ and ${\rm Fix}(T)$ is the set of fixed points of $T$. The map $\theta$ will be uniformly continuous and such that $\theta\circ\theta=\theta$. Following a classical terminology we call it a retraction of $\overline{\mathcal U}$ onto $\overline{\mathcal U}\cap {\rm Fix}(T)$. The restriction of $\theta$ to ${\mathcal U}$ will be of class $C^1$, the differential map $d\theta$ being itself uniformly continuous and bounded from ${\mathcal U}$ to the space ${\mathcal B}(Z)$ of bounded linear operators on $Z$.

Then we will consider the subset ${\mathcal V}:={\mathcal U}\cap \Gamma^+_{\leq q}$, which is relatively open in $\Gamma^+_{\leq q}$ for the topology of $Z$. We will see that $\theta(\overline{\mathcal V})\subset \overline{\mathcal V}\cap \Gamma_{\leq q}^+$ and
$\overline{\mathcal V}\cap {\rm Fix}(\theta)=\overline{\mathcal V}\cap \Gamma_{\leq q}^+$, so $\theta$ may be considered as a retraction of $\overline{\mathcal V}$ onto $\overline{\mathcal V}\cap \Gamma^+_{\leq q}$. Under our assumptions, we will prove the inclusion ${\mathcal S}\subset {\mathcal V}$ which implies, in combination with \eqref{inf_in_S}, the equality
\[
J_{\leq q}=\inf_{\gamma\in \mathcal{V}}\left\{\mathcal{E}^{DF}(\theta(\gamma))-\epsilon_P\widetilde{\Tr}_{L^2}[\theta(\gamma)]\right\}.
\]
It will even turn out that ${\mathcal U}$ is a uniform neighborhood of ${\mathcal S}$ in $Z$, that is,
\[ 
{\mathcal S}+B_{Z}(0,\rho)\subset {\mathcal U}
\]
for some positive constant $\rho$. This property, combined with a formula for the differential of $\theta$, will allow us to prove Proposition \ref{3.lem:6.3}.\medskip

Before giving the definition of $\mathcal{U}$, we take $r>0$ very small, and we introduce the following set, as in \cite{Ser09}:
\begin{align*}
  \Gamma_{\leq q,r}:=\Set*{\gamma\in Z\given \mathrm{dist}_{\mathfrak{S}_{1,1}\cap Y}(\gamma,\Gamma_{\leq q})<r}.
\end{align*}
Then analogously to Lemmas \ref{3.lem:self-adj} and \ref{lem:bottom}, we have for any $\gamma\in \Gamma_{\leq q,r}$,
\begin{align}\label{eq:kappa-r}
  (1-\kappa_r)|D^0|\leq |D_\gamma|\leq (1+\kappa_r)|D^0|
\end{align}
and
\begin{align}\label{eq:lambda0-r}
  \inf\sigma(|D_{\gamma}|)\geq \lambda_{0,r}\geq 1-\kappa_r,
\end{align}
where $\kappa_r:=\alpha\,\big(C_Gz+C'_{EE}(q^++2r)\big)$ and
\begin{align*}
  \lambda_{0,r}:=1-\alpha\max\big\{C_Hz+C_{EE}r+C_{EE}''(q^++r);\frac{C_0}{\ell}z+C_{EE}(q^++r)\big\}.
\end{align*}

\begin{definition}[Admissible set for the retraction]\label{def:5.1}
Assume that $\kappa_r<1$. Let $a_r:=\frac{\alpha}{2}C_{EE}\,(1-\kappa_r)^{-1/2}\lambda_{0,r}^{-1/2}$. Given $0<\tau<\frac{1}{2a_r}$, let $M:=\max\left\{\frac{2+a_r(q^++r)}{2};\frac{1}{1-2a_r \tau}\right\}$. We then define
\[
\mathcal{U}:=\Set*{\gamma\in \Gamma_{\leq q,r}\given \max\{\|\gamma|D^0|^{1/2}\|_{\mathfrak{S}_{1,1}};\|\gamma\|_{Y}\}+M\|T(\gamma)-\gamma\|_{Z}< \tau}.
\]
\end{definition}

We have the following result.
\begin{proposition}[Existence and differentiability of the retraction]\label{prop:5.1}
Take $\kappa_r, a_r, \tau, \mathcal{U}$ as in Definition \ref{def:5.1}. Let $k:=2a_r\tau$ and $\mathcal{V}:=\mathcal{U}\cap \Gamma_{\leq q}^+ $. Then the sequence of iterated maps $(T^p)_{p\geq 1}$ converges uniformly on $\overline{\mathcal{U}}$ to a limit $\theta$ with $\theta(\overline{\mathcal{U}})= \textrm{Fix}(T)\cap \overline{\mathcal{U}}$, $\theta(\overline{\mathcal{V}})= \Gamma_{\leq q}^+\cap \overline{\mathcal{V}}$ and $\theta\circ\theta=\theta$. We have the estimate
\begin{align}\label{geom}
  \forall\;\gamma\in\overline{\mathcal{U}},\;\|\theta(\gamma)-T^p(\gamma)\|_{Z}\leq \frac{k^p}{1-k}\|T(\gamma)-\gamma\|_{Z}.
\end{align}
Moreover $\theta\in C^{1,\textrm{unif}}(\mathcal{U},Z)$ and $d\theta(T^p)$ converges uniformly to $d\theta$ on $\mathcal{U}$.

In this way we obtain a continuous retraction $\theta$ of $\overline{\mathcal{U}}$ onto $\Gamma_{\leq q}^+\cap \overline{\mathcal{U}}$ whose restriction to $\mathcal{U}$ is of class $C^{1,\textrm{unif}}$. This map and its differential are bounded and uniformly continuous on $\mathcal{U}$.

For any $\gamma\in \textrm{Fix}(T)\cap \mathcal{U}$ and $\xi\in Q_\ell^*$, the linear operator $h\mapsto d\theta_{\xi}(\gamma) h$ satisfies 
\[
P_{\gamma,\xi}^+d\theta_{\xi}(\gamma) h P_{\gamma,\xi}^+=P_{\gamma,\xi}^+h_{\xi}P_{\gamma,\xi}^+\,\,\,\mathrm{and}\,\,P_{\gamma,\xi}^- d\theta_{\xi}(\gamma) h P_{\gamma,\xi}^-=0,
\]
where $
\theta(\gamma)=\fint_{Q_\ell^*}^\oplus \theta_\xi(\gamma)d\xi$, according to the Floquet-Bloch decomposition. 
In other words, the splitting $L^2_{\xi}=P_{\gamma,\xi}^+L^2_{\xi}\oplus P_{\gamma,\xi}^-L^2_{\xi}$ gives a block decomposition of $d\theta_{\xi}(\gamma) h$ of the form
\begin{equation}\label{3.eq:5.4'}
d\theta_{\xi}(\gamma)h=
\begin{pmatrix} 
P_{\gamma,\xi}^+h_\xi P_{\gamma,\xi}^+ & b_{\gamma,\xi}(h)^* \\
b_{\gamma,\xi}(h) & 0 
\end{pmatrix}.
\end{equation}
\end{proposition}
The proof of Proposition \ref{prop:5.1} is postponed to the end of this section.\medskip

To apply Proposition \ref{prop:5.1} to the proof of Proposition \ref{3.lem:6.3}, we need to find $\tau\in \left(0,\frac{1}{2a_r}\right)$ such that $\mathcal{U}$ is a uniform neighborhood of $\mathcal{S}$. From Lemma \ref{3.lem:6.1} and the definition of ${\mathcal U}$, we can observe that if
\[
\tau>\sqrt{\max\{(1-\kappa-\frac{\alpha}{2} C_{EE}q^+)^{-1}\epsilon_{P}q;1\}\,q^+}.
\]
then there is $\rho>0$ such that for every $\gamma\in \mathcal{S}$, one has the inclusion $B_{Z}(\gamma,\rho)\subset\mathcal{U}$.
Thus, we have the following.
\begin{lemma}\label{3.cor:5.8}
Assume that $\kappa<1-\frac{\alpha}{2} C_{EE}q^+$, and let $a_r$ be as above. Assume in addition that
\[
2a_r\,\sqrt{\max\{(1-\kappa-\frac{\alpha}{2} C_{EE}q^+)^{-1}\epsilon_{P}q;1\}\,q^+}<1.
\]
Then there exist $\tau\in (0,\frac{1}{2a_r})$ and $\rho>0$ such that $\mathcal{S}+B_{Z}(0,\rho)\subset \mathcal{U}$.
\end{lemma}

We are now in a position to prove the main result of this section. 
\begin{proof}[Proof of Proposition \ref{3.lem:6.3} (as a consequence of Proposition \ref{prop:5.1})] Under Assumption \ref{3.ass:2.1'}, we may choose $r>0$ so small that the assumptions of Lemma \ref{3.cor:5.8} hold true. Then we may take $\tau\in (0,\frac{1}{2a_r})$ and $\rho>0$ satisfying the conclusion of this lemma. To prove \eqref{3.eq:7.1}, we argue by contradiction. Otherwise, there would be an $\epsilon_0 >0$ such that, for $n$ large enough,
\[
\fint_{Q_\ell^*}\Tr_{L^2_\xi}[(D_{\gamma_{n},\xi}-\epsilon_{P})\gamma_{{n},\xi}]\,d\xi\geq \inf_{\substack{\gamma\in \Gamma_{\leq q}\\\gamma=P_{\gamma_{n}}^+\gamma} }\fint_{Q_\ell^*}\Tr_{L^2_\xi}\big[(D_{\gamma_{n},\xi}-\epsilon_P)\gamma_\xi]\,d\xi+\epsilon_0.\]
By Lemma \ref{3.lem:6.4}, there exists an operator $\widehat{\gamma}_n\in \Gamma_{\leq q}$ which solves the following minimization problem:
\[
\fint_{Q_\ell^*}\Tr_{L^2_\xi}\big[(D_{\gamma_n,\xi}-\epsilon_{P})\widehat{\gamma}_{n,\xi}\big]\,d\xi= \min_{\substack{\gamma\in \Gamma_{\leq q}\\\gamma=P_{\gamma_n}^+\gamma}}\fint_{Q_\ell^*}\Tr_{L^2_\xi}\big[(D_{\gamma_n,\xi}-\epsilon_{P})\gamma_{\xi}\big]\,d\xi.
\]
From Lemma \ref{3.lem:6.4} and Proposition \ref{3.prop:5.8}, $\|\widehat{\gamma}_n\|_{Z}$ is uniformly bounded. So according to Corollary \ref{3.cor:5.8}, there is $\sigma>0$ such that for any $n$ large enough and any $s\in [0,\sigma]$, $(1-s)\gamma_{n}+s\widehat{\gamma}_n\in \Gamma_{\leq q}\cap B_{Z}(\gamma_n,\rho) \subset \mathcal{V}$. Then, from Proposition \ref{prop:5.1}, the function $f_n: [0,\sigma]\ni s \mapsto (\mathcal{E}^{DF}-\epsilon_P\widetilde{\Tr}_{L^2})(\theta[(1-s)\gamma_n+s\widehat{\gamma}_n])$ is of class $C^1$, and the sequence of derivatives $(f'_n)$ is equicontinuous on $[0,\sigma]$. From \eqref{3.eq:5.4'}, we infer
\begin{align*}
  f_n'(0)=\widetilde{\Tr}_{L^2}\big[(D_{\gamma_n}-\epsilon_P)(\widehat{\gamma}_n-\gamma_n)\big]\leq -\frac{\epsilon_0}{2}.
\end{align*}
So there is $0<s_0<\sigma$ independent of $n$ such that for any $s\in [0,s_0]$ we have $f_n'(s)\leq -\frac{\epsilon_0}{4}$. Hence, 
\begin{align*}
  (\mathcal{E}^{DF}-\epsilon_P\widetilde{\Tr}_{L^2})(\theta[(1-s_0)\gamma_n+s_0\widehat{\gamma}_n])=f_n(s_0)\leq f_n(0)-\frac{\epsilon_0 s_0}{4}=(\mathcal{E}^{DF}-\epsilon_P\widetilde{\Tr}_{L^2})(\gamma_n)-\frac{\epsilon_0 s_0}{4}.
\end{align*}
Recalling that $(\mathcal{E}^{DF}-\epsilon_{P}\widetilde{\Tr}_{L^2})(\gamma_{n})$ converges to $ J_{\leq q}$, we conclude that $(\mathcal{E}^{DF}-\epsilon_P\widetilde{\Tr}_{L^2})(\theta[(1-s_0)\gamma_n+s_0\widehat{\gamma}_n])<J_{\leq q}$ when $n$ is large enough.
This contradicts the definition of $J_{\leq q}$, since $\theta[(1-s)\gamma_{n}+s\widehat{\gamma}_n]\in \Gamma_{\leq q}^+$. Hence the proposition.
\end{proof}

It remains to prove Proposition \ref{prop:5.1}, but before that, we need some preliminary results.
Recall that $P_{0}^+=\mathbbm{1}_{\mathbb{R}^+}(D^0-\alpha zG_\ell)$. We have the following lemma.
\begin{lemma}\label{3.lem:5.2}
Take $\kappa_r, a_r$ as in Definition \ref{def:5.1} and consider the map
\[
\begin{alignedat}{2}
Q \colon  \gamma \longmapsto P_\gamma^+-P_0^+,
\end{alignedat}
\]
that is, $
Q(\gamma):= \fint^\oplus_{Q_\ell^*} Q_\xi(\gamma)\,d\xi$ with $ Q_\xi(\gamma):=P^+_{\gamma,\xi}-P^+_{0,\xi}$.

Then $Q$ is in $C^{1,\textrm{lip}}\left(\Gamma_{\leq q,r}\,;\,|D^0|^{-1/2}\,Y\right)$ and we have the estimates
\begin{align}\label{eq:5.5a}
  \forall \gamma\in\Gamma_{\leq q,r},\quad \forall h\in Z,\quad \||D^0|^{1/2}dQ(\gamma)h\|_{Y}\leq a_r\|h\|_{Z}
\end{align}
and
\begin{align}\label{eq:5.5b}
  \forall \gamma,\gamma'\in \Gamma_{\leq q,r},\quad \||D^0|^{1/2}[dQ(\gamma)h-dQ(\gamma')h]|D^0|^{1/2}\|_Y\leq K\|\gamma-\gamma'\|_{Z}\|h\|_{Z},
\end{align}
where $K$ is a positive constant depending only on $\kappa_r$ which remains bounded when $\kappa_r$
stays away from $1$.
\end{lemma}
\begin{remark}
In \eqref{eq:5.5a}, if $\gamma\in \Gamma_{\leq q}$ one can replace $a_r$ by the constant $a$ introduced in Assumption \ref{3.ass:2.1}, since $\lim_{r\to 0} a_r=a$.
\end{remark}
\begin{proof}[Proof of Lemma \ref{3.lem:5.2}]
By Lemma \ref{3.lem:self-adj}, for every $\xi \in Q_\ell^*$ and for all $\gamma\in\Gamma_{\leq q,r}$, $D_{\gamma,\xi}$ is a self-adjoint operator, and $0$ is in its resolvent set. Then by Taylor's formula \cite[Chapter VI.5, Lemma 5.6]{kato2013perturbation} or \cite{griesemer1999minimax}, we have
\begin{equation}\label{3.eq:5.2}
  P_{\gamma,\xi}^\pm=\frac{1}{2}\pm\frac{1}{2\pi}\int_{-\infty}^{+\infty}(D_{\gamma,\xi}-iz)^{-1}dz=\frac{1}{2}\pm\frac{1}{\pi}\int_{0}^{+\infty}D_{\gamma,\xi}(|D_{\gamma,\xi}|^2+z^2)^{-1}dz
\end{equation}
and, by the second resolvent identity, 
\[
\begin{aligned}
Q_\xi(\gamma)=-\frac{\alpha}{2\pi}\int_{-\infty}^{+\infty}(D_{\gamma,\xi}-iz)^{-1}V_{\gamma,\xi}(D_{0,\xi}-iz)^{-1}dz.
\end{aligned}
\]
Hence, for every $h\in Z$, we deduce from \eqref{3.eq:5.2} and the second resolvent formula again, that
\begin{align}\label{eq:dQ}
  dQ_\xi(\gamma)h=dP_{\gamma,\xi}^+\,h=-\frac{\alpha}{2\pi}\int_{-\infty}^{+\infty}(D_{\gamma,\xi}-iz)^{-1}V_{h,\xi}(D_{\gamma,\xi}-iz)^{-1}dz.
\end{align}
Besides, for any $u_{\xi}\in L^2_{\xi}(Q_\ell)$ and any $\gamma\in \Gamma_{\leq q,r}$, we have 
\begin{align}\label{integration}
   \int_{-\infty}^{+\infty} \left(u_\xi,\, |D_{\gamma,\xi}|^{1/2}(|D_{\gamma,\xi}|^2+|z|^2)^{-1}|D_{\gamma,\xi}|^{1/2}u_{\xi}\right)_{L^2_\xi}dz
=\pi\|u_\xi\|^2_{L^2_\xi}.
\end{align}
We infer from \eqref{eq:lambda0-r} that
\begin{align}\label{estimgamma}
  \||D_{\gamma}|^{-1}\|_{Y}\leq \lambda_{0,r}^{-1}.
\end{align}
Thus, gathering \eqref{eq:dQ}, \eqref{integration}, \eqref{estimgamma} with Lemma \ref{3.lem:5.1}, for any $\phi_\xi,\psi_\xi \in L^2_\xi$ we may write
\begin{equation}\label{3.eq:5.3}
 \begin{aligned}
 \MoveEqLeft\left|\big(\psi_\xi,|D_{\xi}|^{1/2}(dQ_\xi(\gamma)h) \phi_\xi\big)_{L^2_\xi}\right|\\
&=\frac{\alpha}{2\pi} \left|\int_{-\infty}^{+\infty}\left(\psi_\xi,|D_{\xi}|^{1/2} (D_{\gamma,\xi}-iz)^{-1}V_{h,\xi}(D_{\gamma,\xi}-iz)^{-1}\phi_\xi\right)_{L^2_\xi}dz\right|\\
&\leq \frac{\alpha}{2\pi}\|V_{h,\xi}\|_{\mathcal{B}(L^2_\xi)}\left(\int_{-\infty}^{+\infty}\left\|(D_{\gamma,\xi}-iz)^{-1}|D_\xi|^{1/2}\psi_{\xi}\right\|_{L^2_\xi}^2dz\right)^{1/2}\left(\int_{-\infty}^{+\infty}\left\|(D_{\gamma,\xi}-iz)^{-1}\phi_\xi\right\|_{L^2_\xi}^2dz\right)^{1/2}\\
&\leq \frac{\alpha}{2} \|V_{h,\xi}\|_{\mathcal{B}(L^2_\xi)}\left\||D_\xi|^{1/2}|D_{\gamma,\xi}|^{-1/2}\right\|_{\mathcal{B}(L^2_\xi)}\left\||D_{\gamma,\xi}|^{-1/2}\right\|_{\mathcal{B}(L^2_\xi)}\|\psi_\xi\|_{L^2_\xi}\|\phi_\xi\|_{L^2_\xi}\\
&\leq \frac{\alpha}{2}C_{EE}(1-\kappa_r)^{-1/2}\lambda_{0,r}^{-1/2}\|h\|_{Z}\|\psi_\xi\|_{L^2_\xi}\|\phi_\xi\|_{L^2_\xi}\,
\end{aligned}
\end{equation}
hence we obtain \eqref{eq:5.5a}; namely,
\begin{align*}
\||D^0|^{1/2}dQ(\gamma)h\|_Y\leq\frac{\alpha}{2} C_{EE}(1-\kappa_r)^{-1/2}\lambda_{0,r}^{-1/2}\|h\|_{Z}.
\end{align*}
For the second inequality, we write 
\begin{align*}
 dQ_\xi(\gamma)h-dQ_\xi(\gamma')h
=&-\frac{\alpha^2}{2\pi}\int_{-\infty}^{+\infty}(D_{\gamma,\xi}-iz)^{-1}V_{\gamma'-\gamma,\xi}(D_{\gamma',\xi}-iz)^{-1}V_{h,\xi}(D_{\gamma,\xi}-iz)^{-1}dz\\
&-\frac{\alpha^2}{2\pi}\int_{-\infty}^{+\infty}(D_{\gamma',\xi}-iz)^{-1}V_{h,\xi}(D_{\gamma,\xi}-iz)^{-1}V_{\gamma'-\gamma,\xi}(D_{\gamma',\xi}-iz)^{-1}dz.
\end{align*}
Proceeding as above, we get \eqref{eq:5.5b}. The fact that $Q\in C^{1,\textrm{lip}}\left(\Gamma_{\leq q,r}\,;\,|D^0|^{-1/2}\,Y\right)$ follows from \eqref{eq:5.5a} and \eqref{eq:5.5b}.
\end{proof}

\begin{lemma}\label{3.lem:5.2'}
Take $\kappa_r, a_r$ as in Definition \ref{def:5.1}. Then the map $T:\,\gamma\to P^+_{\gamma}\gamma P^+_\gamma$ is well-defined and of class $\mathcal{C}^{1,1}$ on $ \Gamma_{\leq q,r}$ with values in $\Gamma_{\leq q,r}\subset Z$. Moreover, for any $\gamma\in \Gamma_{\leq q,r}$, 
\begin{align}\label{3.eq:4.6a}
  \|T^2(\gamma)-T(\gamma)\|_{Z}&\leq 2a_r\Big(\max\big\{\|T(\gamma)|D^0|^{1/2}\|_{\mathfrak{S}_{1,1}};\|T(\gamma)\|_{Y}\big\}\Big.\notag\\
  &\quad\Big.+\frac{a_r(q^++r)}{2}\|\gamma-T(\gamma)\|_{Z}\Big)
\|T(\gamma)-\gamma\|_{Z}.
\end{align}
Moreover, there are two positive constants $C_{\kappa,r},\; L_{\kappa,r}$ such that
\begin{equation}\label{3.eq:4.6b}
  \forall\;\gamma\in \Gamma_{\leq q,r},\quad \|dT(\gamma)\|_{\mathcal{B}(Z)}\leq C_{\kappa,r}\,\Big(1+\max\{\|\gamma|D^0|^{1/2}\|_{\mathfrak{S}_{1,1}};\|\gamma\|_Y\}\Big)
\end{equation}
and
\begin{equation}\label{3.eq:4.6c}
 \forall\,\gamma,\gamma'\in \Gamma_{\leq q,r},\;\|dT(\gamma')-dT(\gamma)\|_{\mathcal{B}(Z)}\leq L_{\kappa,r}\Big(1+\max\{\|\gamma|D^0|^{1/2}\|_{\mathfrak{S}_{1,1}};\|\gamma\|_Y\}\Big)\|\gamma'-\gamma\|_{Z}.
\end{equation}
\end{lemma}
\begin{proof}
If $\gamma\in \Gamma_{\leq q,r}$, there is $\gamma_0\in \Gamma_{\leq q,r}$ such that $\Vert \gamma-\gamma_0\Vert_{\mathfrak{S}_{1,1}\cap Y}<r$. Then $P^+_{\gamma}\gamma_0 P^+_{\gamma}\in \Gamma_{\leq q}$, $T(\gamma)\in Z$ and
\[
\Vert T(\gamma) - P^+_{\gamma}\gamma_0 P^+_{\gamma}\Vert_{\mathfrak{S}_{1,1}\cap Y}=\Vert P^+_{\gamma}(\gamma-\gamma_0)P^+_{\gamma}\Vert_{\mathfrak{S}_{1,1}\cap Y}\leq
\Vert \gamma-\gamma_0\Vert_{\mathfrak{S}_{1,1}\cap Y}<r,
\]
so $T(\gamma)\in \Gamma_{\leq q,r}$.\medskip

Let $\gamma,\gamma'\in \Gamma_{\leq q,r}$. Then $P^+_{\gamma}-P^+_{\gamma'}$ can be written as
\begin{align*}
  P^+_{\gamma}-P^+_{\gamma'}=\int_{0}^1 dQ(\gamma'+t(\gamma-\gamma'))(\gamma-\gamma')dt.
\end{align*}
From \eqref{eq:5.5a},
\begin{align*}
  \||D^0|^{1/2}(P^+_{\gamma}-P^+_{\gamma'})\|_{Y}\leq a_r\,\|\gamma-\gamma'\|_{Z}.
\end{align*}
For the estimate \eqref{3.eq:4.6a}, we write 
\begin{align*}
  \MoveEqLeft T^2(\gamma)-T(\gamma)=(P^+_{T(\gamma)}-P^+_\gamma)T(\gamma)\left(P^+_{T(\gamma)}-P^+_{\gamma}+P^+_{\gamma}\right)+P^+_\gamma T(\gamma)(P^+_{T(\gamma)}-P^+_\gamma)\\
  &\quad=(P^+_{T(\gamma)}-P^+_\gamma)T(\gamma)+ T(\gamma)(P^+_{T(\gamma)}-P^+_\gamma)+(P^+_{T(\gamma)}-P^+_\gamma)T(\gamma)(P^+_{T(\gamma)}-P^+_\gamma).
\end{align*}
Then
\begin{align*}
  \MoveEqLeft \|T^2(\gamma)-T(\gamma)\|_{Z}\leq \|(P^+_{T(\gamma)}-P^+_\gamma)T(\gamma)\|_{X\cap Y}\\
  &\qquad\qquad+\| T(\gamma)(P^+_{T(\gamma)}-P^+_\gamma)\|_{X\cap Y}+\|(P^+_{T(\gamma)}-P^+_\gamma)T(\gamma)(P^+_{T(\gamma)}-P^+_\gamma)\|_{Z}.
\end{align*}
We have
\begin{align*}
\|T(\gamma)(P^+_{T(\gamma)}-P^+_\gamma)\|_{X\cap Y}\leq \||D^0|^{1/2}(P^+_{T(\gamma)}-P^+_{\gamma})\|_Y\max\{\|T(\gamma)|D^0|^{1/2}\|_{\mathfrak{S}_{1,1}};\|T(\gamma)\|_{Y}\}
\end{align*}
and
\begin{align*}
\MoveEqLeft\|(P^+_{T(\gamma)}-P^+_\gamma)T(\gamma)(P^+_{T(\gamma)}-P^+_\gamma)\|_{Z}\leq \||D^0|^{1/2}(P^+_{T(\gamma)}-P^+_{\gamma})\|_Y^2\|T(\gamma)\|_{\mathfrak{S}_{1,1}\cap Y}.
\end{align*}
Notice that $\|T(\gamma)\|_{\mathfrak{S}_{1,1}\cap Y}\leq \|\gamma\|_{\mathfrak{S}_{1,1}\cap Y}\leq q^++r$. Gathering together these estimates, we obtain \eqref{3.eq:4.6a}.\medskip

We now turn to the proof of \eqref{3.eq:4.6b} and \eqref{3.eq:4.6c}. From Lemma \ref{3.lem:5.2}, $T$ is in $C^1(\Gamma_{\leq q,r})$ with
\[
dT(\gamma)h=(dQ_{\gamma}h)\gamma P_\gamma+P_\gamma\gamma(dQ_{\gamma}h)+P_\gamma hP_\gamma.
\]
Notice that for any $\gamma\in \Gamma_{\leq q,r}$,
\begin{align}\label{eq:DP}
  \||D^{0}|^{1/2}P^+_{\gamma}|D^{0}|^{-1/2}\|_Y\leq (1-\kappa_r)^{-1/2}\||D_{\gamma}|^{1/2}P^+_{\gamma}|D^{0}|^{-1/2}\|_Y\leq \frac{(1+\kappa_r)^{1/2}}{(1-\kappa_r)^{1/2}}.
\end{align}
Thus, using \eqref{eq:5.5a}, one finds a constant $C_{\kappa,r}$ such that for any $h\in Z$,

\begin{align*}
\|dT(\gamma)h\|_{Z}&&\\
\leq 
  \max\Big\{& 
  2\,\big\||D^0|^{1/2}(dQ_{\gamma}h)\big\|_{Y}
\big\|\gamma|D^0|^{1/2}\big\|_{\mathfrak{S}_{1,1}}
\big\||D^{0}|^{1/2}P^+_{\gamma}|D^{0}|^{-1/2}\big\|_Y\;+&\!\!\!\big\||D^{0}|^{1/2}P^+_{\gamma}|D^{0}|^{-1/2}\big\|_Y^2\big\|h\big\|_{X}\,;\\
&& 2\|\gamma\|_Y\|dQ_{\gamma}h\|_{Y}+\|h\|_{Y}\Big\}\\
 \leq C_{\kappa,r}\Big(1&+\max\big\{\|\gamma|D^0|^{1/2}\|_{\mathfrak{S}_{1,1}};\|\gamma\|_Y\big\}\Big)\|h\|_{Z},&
\end{align*}
so \eqref{3.eq:4.6b} is proved. Finally, for the term $dT(\gamma')-dT(\gamma)$, we have
\begin{align*}
 dT_\xi(\gamma') h-dT_\xi(\gamma) h &=(dQ_{\gamma,\xi}h)\gamma_\xi P_{\gamma,\xi}+P_{\gamma,\xi}\gamma_\xi (dQ_{\gamma,\xi}h)+P_{\gamma,\xi} h_\xi P_{\gamma,\xi}\\
&\quad {}-(dQ_{\gamma',\xi}h)\gamma'_\xi P_{\gamma',\xi}-P_{\gamma',\xi}\gamma'_\xi (dQ_{\gamma',\xi}h)-P_{\gamma',\xi} h_\xi P_{\gamma',\xi}.
\end{align*}
Proceeding in the same way as for \eqref{3.eq:4.6b}, we can get \eqref{3.eq:4.6c}.
\end{proof}

We now show that $\mathcal{U}$ and $T$ satisfy all the assumptions in \cite[Proposition 2.2]{Ser09} in the Banach space $Z$.
\begin{proposition}\label{3.prop:5.5}
Let $\kappa_r,\,a_r,\,\tau,\,{\mathcal U}$ be as in Definition \ref{def:5.1}. Then $T$ is in $C^0(\overline{\mathcal{U}})\cap C^{1,\textrm{lip}}(\mathcal{U},Z)$. Moreover $T(\mathcal{U})\subset \mathcal{U}$ and the following estimates are satisfied:
\begin{align*}
 \sup_{\gamma\in\mathcal{U}}\|dT(\gamma)\|_{\mathcal{B}(Z)}<\infty,\quad \sup_{\gamma\in\mathcal{U}}\|T(\gamma)-\gamma\|_{Z}<\infty
\end{align*}
and
\begin{align*} \forall\;\gamma\in \mathcal{U},\quad \|T^2(\gamma)-T(\gamma)\|_{Z}\leq k\|T(\gamma)-\gamma\|_{Z}
\end{align*}
with $k:=2a_r\tau<1$.
\end{proposition}
\begin{proof}
For any $\gamma\in \mathcal{U}$, we have
\begin{align*}
  \MoveEqLeft \|T(\gamma)|D^0|^{1/2}\|_{\mathfrak{S}_{1,1}}\leq \|\gamma|D^0|^{1/2}\|_{\mathfrak{S}_{1,1}}+\|(\gamma-T(\gamma))|D^0|^{1/2}\|_{\mathfrak{S}_{1,1}}\leq \|\gamma|D^0|^{1/2}\|_{\mathfrak{S}_{1,1}}+\|\gamma-T(\gamma)\|_{X}
\end{align*}
and
\begin{align*}
  \MoveEqLeft \|T(\gamma)\|_{Y}\leq \|\gamma\|_Y.
\end{align*}
As a result, as $M\geq \frac{2+a_r(q^++r)}{2}$, \eqref{3.eq:4.6a} implies that
\[
\|T^2(\gamma)-T(\gamma)\|_{Z}\leq k\|T(\gamma)-\gamma\|_{Z}.
\]
Moreover, using the inequality $M\geq \frac{1}{1-2 a_r\tau}\,$, we get
\begin{align*}
  \MoveEqLeft \max\big\{\|T(\gamma)|D^0|^{1/2}\|_{\mathfrak{S}_{1,1}};\|T(\gamma)\|_{Y}\big\}+M\|T^2(\gamma)-T(\gamma)\|_{Z}\\
  &\qquad\qquad\leq \max\big\{\|\gamma|D^0|^{1/2}\|_{\mathfrak{S}_{1,1}};\|\gamma\|_{Y}\big\}+(1+Mk)\|T(\gamma)-\gamma\|_{Z}<\tau,
\end{align*}
so $T(\gamma)\in \mathcal{U}$.

The fact that $\sup_{\gamma\in{\mathcal{U}_r}}\|dT(\gamma)\|_{\mathcal{B}(Z)}<\infty$ and that $d T$ is Lipschitz continuous on $\mathcal{U}$ follows from \eqref{3.eq:4.6b} and \eqref{3.eq:4.6c}. Besides, for $\gamma\in \mathcal{U}$, we have $\|T(\gamma)-\gamma\|_{Z}<\frac{\tau}{M}$. This ends the proof of Proposition \ref{3.prop:5.5}.
\end{proof}
We can now prove Proposition \ref{prop:5.1}, which implies Proposition \ref{3.lem:6.3}, as we have already seen.
\begin{proof}[Proof of Proposition \ref{prop:5.1}]
By Proposition \ref{3.prop:5.5}, we may apply \cite[Proposition 2.2]{Ser09} to our map $T$ and 
our open set $\mathcal U$ in the Banach space $Z$. This allows us to construct $\theta\in C^0(\overline{\mathcal{U}},Z)\cap C^{1,\textrm{unif}}(\mathcal{U},Z)$
with the properties $\theta(\overline{\mathcal{U}})= \textrm{Fix}(T)\cap \overline{\mathcal{U}}$, $\theta\circ\theta=\theta$ and the convergence estimate \eqref{geom}. By our definition of $T$, we have $T(\Gamma_{\leq q})\subset \Gamma_{\leq q}$ and $\Gamma^+_{\leq q}=\textrm{Fix}(T)\cap \Gamma_{\leq q}$, hence the additional property $\theta(\overline{\mathcal{V}})= \Gamma_{\leq q}^+\cap \overline{\mathcal{V}}$.
 The proof of \eqref{3.eq:5.4'} is exactly the same as in \cite[Theorem 2.10]{Ser09}. This ends the proof of Proposition \ref{prop:5.1}.
\end{proof}

\section{Proof of Theorem \ref{3.th:2.7'}}\label{3.sec:5}

Throughout this section, we assume that Assumption~\ref{3.ass:2.1'} is satisfied and that $\epsilon_P>e_0$. Let $(\gamma_n)_{n\geq 1}$ be a minimizing sequence for $J_{\leq q}$ lying in $\mathcal S$. According to Lemma \ref{3.lem:6.1}, this sequence is uniformly bounded in $Z$. We split $(\gamma_n)_{n\geq 1}$ into two parts: $(\widetilde{\gamma}_n)_{n\geq 1}$ and $(\gamma_n-\widetilde{\gamma}_n)_{n\geq 1}$ where, for each $n$,
\begin{align}\label{eq:gamma-tilde-pn}
\widetilde{\gamma}_n :=p_n\gamma_np_n\quad \textrm{with}\quad p_n:=\mathbbm{1}_{[0,e_1)}(D_{\gamma_n})
\end{align}
where $e_1$ has been defined in Formula \eqref{level1}. Thanks to Corollary \ref{Bbar}, for almost every $\xi\in Q_\ell^*$ the rank of $p_{n,\xi}$, and therefore of $\widetilde{\gamma}_{n,\xi}$, is at most $j_1$, so that $ \widetilde{\gamma}_n\in {\bf{\overline B}}$.

Actually, we prove in Lemma \ref{3.lem:5.7.1} that, for each $n\geq 1$, $\widetilde{\gamma}_n\in X^2_\infty$ whereas $\gamma_n\in X$; roughly speaking, we reach a $L^\infty(Q^*_\ell; H^1_\xi(Q_\ell))$ regularity instead of a $L^2(Q^*_\ell; H_\xi^{1/2}(Q_\ell))$ regularity for the associated eigenfunctions. Moreover $\widetilde{\gamma}_n$ is close to $\gamma_n$ in $X$ (Lemma \ref{3.lem:6.5}), so $(\widetilde{\gamma}_n)_{n\geq 1}$ is a modified minimizing sequence with higher regularity than $(\gamma_n)_{n\geq 1}$. 
\medskip

The structure of the proof of Theorem \ref{3.th:2.7'} is as follows. In Subsection \ref{sec:6.1}, we will show that $\|\gamma_n-\widetilde{\gamma}_n\|_X\rightarrow0$ when $n$ goes to infinity. In Subsection \ref{sec:6.2}, we first study the convergence of the kernels of $(W_{\widetilde{\gamma}_n,\xi})_{n\geq 1}$. Then we deduce the strong convergence of $(V_{\widetilde{\gamma}_n,\xi})_{n\geq 1}$. As a result, $\|P^+_{\gamma_*}-P^+_{\widetilde{\gamma}_n}\|_Y\to 0$. On the other hand, for any $\gamma\in \Gamma_{\leq q}$, we also have $\|(P^+_{\gamma_n}-P^+_{\widetilde{\gamma}_n})\gamma\|_{\mathfrak{S}_{1,1}}\leq C\|\widetilde{\gamma}_n-\gamma_n\|_{X}^{1/2}\|\gamma\|_{Z}\to 0$. Hence in Subsection \ref{sec:6.3}, we can pass to the limit in the energy and in the constraints.

\subsection{Decomposition of minimizing sequences}\label{sec:6.1}
We start with some regularity bounds on $\widetilde{\gamma}_n$ that will be needed in the sequel. 
\begin{lemma}\label{3.lem:5.7.1} 
The sequence $(\widetilde{\gamma}_n)_{n\geq 1}$ and the sequence of kernels $(\widetilde{\gamma}_{n,\xi}(\cdot,\cdot))_{n\geq 1}$ are uniformly bounded in $X_\infty^2$ and in $L^\infty(Q_\ell^*;H^1(Q_\ell\times Q_\ell))$, respectively. More precisely, we have the following, for every $n\geq 1$ and for almost every $\xi$ in $Q_\ell^*$, 
\begin{equation}\label{eq:bd-x2infty}
\|\widetilde{\gamma}_{n}\|_{X^2_\infty}\leq j_1(1-\kappa)^{-4}c^*(j_1)^2
\end{equation}
and 
\begin{equation}\label{eq:bound-gamma}
  \|\widetilde{\gamma}_{n,\xi}(\cdot,\cdot)\|_{H^1(Q_\ell\times Q_\ell)}\leq 2j_1(1-\kappa)^{-4}c^*(j_1)^2.
\end{equation}
\end{lemma}
\begin{proof}
We first prove that $\|p_{n}\|_{X^2_\infty}$ is bounded. Let $(u_{n,j}(\xi))_{j\geq 1}$ be the normalized eigenfunctions of the operator $D_{\gamma_n,\xi}$ with the corresponding eigenvalues $\lambda_{n,j}(\xi)$ counted with multiplicity. Hence,
\[
p_{n,\xi}=\sum_{j=1}^{+\infty}\delta_{n,j}(\xi)\left|u_{n,j}(\xi)\right>\left<u_{n,j}(\xi)\right|
\]
with $\delta_{n,j}(\xi)=1$ if $0\leq \lambda_{n,j}(\xi)< e_1$ and $\delta_{n,j}(\xi)=0$ otherwise.\medskip

By Corollary \ref{Bbar}, we know that $|\Set*{j \in \mathbb{N}^*\given \delta_{n,j}(\xi)=1}|\leq j_1$. By \eqref{eq:bd-eigenf}, 
for any eigenfunction $u_{n,j}(\xi)$, we have
$\delta_{n,j}(\xi)\|u_{n,j}(\xi)\|^2_{H^1_\xi(Q_\ell)} \leq (1-\kappa)^{-4}c^*(j_1)^2$, for every $\xi \in Q_\ell^*$. Now,
\[
\|p_{n,\xi}\|_{X^2(\xi)}=\sum_{ j=1}^{j_1}\delta_{n,k}(\xi)\|u_{n,k}(\xi)\|_{H^1_\xi}^2 \leq j_1 \sup_{j\geq 1}\,\delta_{n,j}(\xi)\|u_{n,j}(\xi)\|_{H^1_\xi}^2.
\]
Hence,
\[
\|p_{n}\|_{X^2_\infty}\leq j_1(1-\kappa)^{-4}c^*(j_1)^2.
\]
Since $p_{n}=p_{n}^{\,2}$, $\widetilde{\gamma}_{n}=p_{n}\widetilde{\gamma}_{n}p_{n}$ and $0\leq \widetilde{\gamma}_n\leq \mathbbm{1}_{L^2(\mathbb{R}^3)}$, we have
\[
\begin{aligned}
\|\widetilde{\gamma}_{n}\|_{X^2_\infty}&=\||D^0|p_{n}\widetilde{\gamma}_{n}p_{n}|D^0|\|_{\mathfrak{S}_{1,\infty}}\leq \|\widetilde{\gamma}_{n}\|_{Y}\|p_{n}\|_{X^2_\infty}\\
&\leq \|p_{n}\|_{X^2_\infty}\leq j_1(1-\kappa)^{-4}c^*(j_1)^2.
\end{aligned}
\]
In terms of kernels, this implies that 
\[
\begin{aligned}
 \||D_{\xi,x}|\widetilde{\gamma}_{n,\xi}(\cdot,\cdot)\|_{L^2(Q_\ell\times Q_\ell)}=\||D_{\xi}|\widetilde{\gamma}_{n,\xi}\|_{\mathfrak{S}_2(\xi)}\leq \|\widetilde{\gamma}_{n,\xi}\|_{X^2(\xi)}\leq j_1(1-\kappa)^{-4}c^*(j_1)^2, 
\end{aligned}
\]
the same holding for $|D_{\xi,y}|\widetilde{\gamma}_{n,\xi}(\cdot,\cdot)$. Thus, $\widetilde{\gamma}_{n,\xi}(\cdot,\cdot)\in L^\infty(Q_\ell^*;H^1(Q_\ell\times Q_\ell))$ and \eqref{eq:bound-gamma} holds.

\end{proof}
We begin the proof of Theorem \ref{3.th:2.7'} by showing the following result as in the case of molecules \cite[Lemma 3.4]{Ser09}. 
\begin{lemma}\label{3.lem:6.5}
Under Assumption~\ref{3.ass:2.1'}, whenever $\epsilon_{P}> e_0$, for any minimizing sequence $(\gamma_n)_{n\geq 1}$ of (\ref{3.eq:5.0'}) in $\Gamma_{\leq q}^+$ we have
\[
\widetilde{\Tr}_{L^2}[\gamma_{n}]\to q\quad \text{and}\quad 
\|\gamma_n-\widetilde{\gamma}_n\|_X\to0. \]
\end{lemma}
\begin{proof}
According to Proposition \ref{3.lem:6.3}, any minimizing sequence $(\gamma_n)_{n\geq 1}$ in $\Gamma_{\leq q}^+$ satisfies \eqref{3.eq:7.1} ; namely
\[
 \fint_{Q_\ell^*}\Tr_{L^2_\xi}\big[(D_{\gamma_n,\xi}-\epsilon_{P})\gamma_{n,\xi}\big]\,d\xi-\inf_{\substack{\gamma\in \Gamma_{\leq q}\\\gamma=P_{\gamma_n}^+\gamma P_{\gamma_n}^+}}\fint_{Q_\ell^*}\Tr_{L^2_\xi}\big[(D_{\gamma_n,\xi}-\epsilon_{P})\gamma_{\xi}\big]\,d\xi\to 0.
\] 
By Lemma \ref{3.lem:6.4}, for every $n$, minimizers of the above minimization problem are of the form $\gamma'_n:=\fint_{Q_\ell^*}^\oplus \mathbbm{1}_{[0,\nu_n)}(D_{\gamma_n,\xi})\,d\xi+\delta_n$ for some $\nu_n\in [\lambda_0,e_0]$ and some $0\leq \delta_n \leq \fint_{Q_\ell^*}^\oplus\mathbbm{1}_{\nu_n}(D_{\gamma_n,\xi})\,d\xi\,$ such that $\widetilde{\Tr}_{L^2}(\gamma'_n)=q$. In particular, $\limsup_{n\to+\infty}\nu_n<\epsilon_P$, for every $n$.
We define 
\[\pi_n:=\fint_{Q_\ell^*}^\oplus \mathbbm{1}_{[e_1,\infty)}(D_{\gamma_n,\xi})\,d\xi,\quad \pi'_{n}:=\fint_{Q_\ell^*}^\oplus \mathbbm{1}_{(\nu_n,e_1)}(D_{\gamma_n,\xi})\,d\xi,\quad\pi''_n:=\fint_{Q_\ell^*}^\oplus
\mathbbm{1}_{[0,\nu_n]}(D_{\gamma_n,\xi})\,d\xi.\] 
We can write $p_n=\pi_n'+\pi_n''$ and we observe that $\gamma_{n}'=\pi''_n\gamma_{n}'\pi''_n$. Proceeding as for \eqref{eq:4.27} and \eqref{3.eq:6.10}, and since $\gamma_n\in \Gamma_{\leq q}^+$ we have
\[
\begin{aligned}
\MoveEqLeft\widetilde{ \Tr}_{L^2}[(D_{\gamma_n}-\epsilon_{P})\gamma_{n}]-\widetilde{ \Tr}_{L^2}[(D_{\gamma_n}-\epsilon_{P})\gamma'_{n}]\\
& =\widetilde{ \Tr}_{L^2}[(D_{\gamma_n}-\nu_n)\pi_{n}\gamma_{n}\pi_{n}]+\widetilde{ \Tr}_{L^2}[(D_{\gamma_n}-\nu_n)\pi'_{n}\gamma_{n}\pi'_{n}]\\
&\quad{}+\widetilde{ \Tr}_{L^2}[(D_{\gamma_n}-\nu_n)(\pi''_{n}\gamma_{n}\pi''_{n}-\mathbbm{1}_{[0,\nu_n]}(D_{\gamma_n}))]+(\epsilon_{P}-\nu_n)\left(q-\widetilde{\Tr}_{L^2}(\gamma_{n})\right).
\end{aligned}
\]
We observe that the four terms in the right-hand side of the above equation are non-negative whereas, from Proposition \ref{3.lem:6.3}, their sum goes to $0$ as $n$ goes to infinity. Therefore,
\[
\widetilde{ \Tr}_{L^2}[\gamma_{n}]\to q \text{ and } \widetilde{ \Tr}_{L^2}[(D_{\gamma_n}-\nu_n)\pi_{n}\gamma_{n}\pi_{n}]\to 0,\]
since $\liminf_{n\to+\infty}(\epsilon_P-\nu_n)\geq \epsilon_P-e_0>0$. But $\pi_{n}(D_{\gamma_n}-\nu_n)\pi_{n}\geq (e_1-\nu_n)\pi_{n}\geq \big(e_1-e_0)\,\pi_n$ and $\pi_{n}(D_{\gamma_n}-\nu_n)\pi_{n}\geq\pi_{n}(|D_{\gamma_n}|-e_0)\pi_{n}$. Thus, taking a convex combination of these two estimates leads to
\[
 \frac{e_1}{e_1-e_0}\,\pi_{n}(D_{\gamma_n}-\nu_n)\pi_{n}\geq \pi_{n}|D_{\gamma_n}|\pi_{n}.
\]
Hence 
\[
\begin{aligned}
\|\pi_n\gamma_n\pi_n\|_X&=\widetilde{ \Tr}_{L^2}[\pi_{n}|D^0|\pi_{n}\gamma_{n}]\leq (1-\kappa)^{-1}\widetilde{ \Tr}_{L^2}[\pi_{n}|D_{\gamma_n}|\pi_{n}\gamma_{n}],
\end{aligned}
\]
and the right-hand side goes to $0$ by \eqref{eq:gap}. It remains to study the limit of $h_{n}:=\pi_n\gamma_n p_n$ as $n$ goes to infinity. Since $(\gamma_n)^2\leq \gamma_n$, we have
\[
(\pi_n\gamma_n\pi_n)^2+h_n h_n^*=\pi_n(\gamma_n)^2\pi_n\leq \pi_n\gamma_n\pi_n.
\]
Hence 
\[
\widetilde{ \Tr}_{L^2}(|D_{\gamma_n}|^{1/2}h_{n}h_{n}^*|D_{\gamma_n}|^{1/2})\to 0.
\]
In other words, $\Vert |D_{\gamma_n}|^{1/2}h_{n}\Vert_{\mathfrak{S}_{2,2}}\to 0$. Taking any operator $A$ in $Y$, by the Cauchy--Schwarz inequality,
\begin{align}
\left|\widetilde{ \Tr}_{L^2}\big[A|D_{\gamma_n}|^{1/2}h_n^*|D_{\gamma_n}|^{1/2}\big]\right|&=\left|\widetilde{\Tr}_{L^2}\big[A|D_{\gamma_n}|^{1/2}p_n\,h_n^*|D_{\gamma_n}|^{1/2}\big]\right|\notag\\
&\leq \Vert A\,|D_{\gamma_n}|^{1/2} p_n\Vert_{\mathfrak{S}_{2,2}}\,\Vert |D_{\gamma_n}|^{1/2}h_{n}\Vert_{\mathfrak{S}_{2,2}}.\label{dualiteHS}
\end{align}
We have already seen that $p_{n,\xi}$ has rank at most $j_1$. Therefore, 
\[\Vert A\,|D_{\g_n}|^{1/2}p_n\Vert_{\mathfrak{S}_{2,2}}\leq \Vert |D_{\g_n}|^{1/2}p_n\Vert_{\mathfrak{S}_{2,2}}\Vert A\Vert_Y\leq j_1\,e_1\Vert A\Vert_Y,
\]
so we deduce immediately from \eqref{dualiteHS} that 
\[
\left\||D_{\gamma_n}|^{1/2}h_n|D_{\gamma_n}|^{1/2}\right\|_{\mathfrak{S}_{1,1}}\to 0,
\]
since $A$ is arbitrary. Hence, thanks to \eqref{eq:borneDgD-1}, $\|h_n\|_X=\left\||D^0|^{1/2}h_n|D^0|^{1/2}\right\|_{\mathfrak{S}_{1,1}}\to 0$. Finally, we obtain that $\|\gamma_n-\widetilde{\gamma}_n\|_X\leq \|\pi_n \gamma_n\pi_n\|_X+2\,\|h_n\|_X\to 0$.
\end{proof}

By Lemma~\ref{3.lem:5.7.1}, up to the extraction of a subsequence, there is $\gamma_*$ in $X^2_\infty\cap Y$, such that 
\begin{equation}\label{3.eq:5.3.1}  \widetilde{\gamma}_n\overset{\ast}{\rightharpoonup} \gamma_*\quad \text{ for the weak-}^*\text { convergence in } X^2_\infty\cap Y,
\end{equation}
since $X^2_\infty$ is a subspace of $\mathfrak{S}_{1,\infty}$ which is the dual space of $\mathfrak{S}_{\infty,1}$ and $Y$ is the dual space of $\mathfrak{S}_{1,1}$. We immediately get the following.

\begin{lemma}[Strong convergence of the density]\label{lem:6.3}
The sequence $\rho_{\widetilde{\gamma}_n}^{1/2}$ converges strongly to $\rho_{\gamma_*}^{1/2}$ in $H^{s}(Q_\ell)$ with $0\leq s<1$, thus in $L^p(Q_\ell)$ for every $1\leq p <6$. In particular, whenever $\epsilon_{P}> e_0$, we have $\int_{Q_\ell}\rho_{\gamma_*}\,dx=q$.
\end{lemma}

\begin{proof}
The proof of the strong convergence of $\rho_{\widetilde{\gamma}_n}^{1/2}$ to $\rho_{\gamma_*}^{1/2}$ in $L^p(Q_\ell)$ for every $1\leq p<6$ is the same as in \cite[p. 730]{catto2001thermodynamic} and relies on the fact that $\widetilde{\gamma}_n\in X^2_\infty$ (see the proofs of \cite[Eqs. (4.51) and  (4.55)]{catto2001thermodynamic}). When $\epsilon_P>e_0$, by Lemma \ref{3.lem:6.5}, $\widetilde{\gamma}_n-\gamma_n$ converges to $0$ in $X$, thus in $\mathfrak{S}_{1,1}$, whereas $\widetilde{\Tr}_{L^2}[\gamma_n]$ converges to $q$. Thus, $\widetilde{\Tr}_{L2}[\gamma_*]=q$.
\end{proof}

\subsection{Convergence of \texorpdfstring{$\left(P^+_{\gamma_n}\right)_{n\geq 1}$}{}}\label{sec:6.2}
We now study the differences between $P^+_{\gamma_n}$ and $P^+_{\widetilde{\gamma}_n}$ and between $P^+_{\widetilde{\gamma}_n}$ and $P^+_{\gamma_*}$ separately. We do not know whether $\gamma_n-\widetilde{\gamma}_n$ goes to $0$ in $Y$ and we do not even know whether $\widetilde{\gamma}_n-\gamma_*$ goes to $0$ in $\mathfrak{S}_{1,1}$, so we cannot rely on the continuity of the map $Q$ introduced in Lemma~\ref{3.lem:5.2}. The proof is therefore more involved than in \cite{Ser09}.
\subsubsection{Convergence of \texorpdfstring{$\left(P^+_{\widetilde{\gamma}_n}-P^+_{\gamma_*}\right)_{n\geq 1}$}{}}
The main result of this section is Corollary~\ref{3.cor:5.12} which states that the sequence $\left(P^+_{\widetilde{\gamma}_n}-P^+_{\gamma_*}\right)_{n\geq 1}$ converges strongly in $Y$.

Recall that 
\[
W_{\gamma}=W_{\geq m,\gamma}+W_{<m,\gamma},\quad\forall\, m\in\mathbb{N}, m\geq 2
\]
where for $\xi\in Q_\ell^*$ and $x,\, y\in Q_\ell$, the kernels of $W_{\geq m,\gamma,\xi}$ and $W_{< m,\gamma,\xi}$ are respectively
\begin{align}\label{eq:W>=m}
  W_{\geq m,\gamma,\xi}(x,y):=\fint_{Q_\ell^*}W_{\geq m,\ell}^\infty(\xi'-\xi,x-y)\,\gamma_{\xi'}(x,y)\,d\xi'
\end{align}
and
\begin{align}\label{eq:W<m}
  W_{< m,\gamma,\xi}(x,y):=\fint_{Q_\ell^*}W_{< m,\ell}^\infty(\xi'-\xi,x-y)\,\gamma_{\xi'}(x,y)\,d\xi'.
\end{align}
We first prove the following.
\begin{lemma}\label{lem:cvge} The two sequences of operators $\big(W_{<m,\widetilde{\gamma}_n}\big)_n$ and $\big(W_{\geq m,\widetilde{\gamma}_n}\big)_n$ are bounded in $\mathfrak{S}_{2,\infty}$. Thus, up to the extraction of a subsequence, we may assume that 
\begin{equation}\label{eq:converg} W_{<m,\widetilde{\gamma}_n}\overset{\ast}{\rightharpoonup}W_{< m,\gamma_*}\quad\textrm{ and }\quad
W_{\geq m,\widetilde{\gamma}_n}\overset{\ast}{\rightharpoonup}W_{\geq m,\gamma_*} \quad \text{ in } \quad \mathfrak{S}_{2,\infty}.
\end{equation}
\end{lemma}
Note that saying that the operator $A=\fint_{Q_\ell^*}^\oplus A_\xi\,d\xi$ belongs to $ \mathfrak{S}_{2,p}$ with $1\leq p\leq +\infty$ is equivalent to saying that the family of  kernels $\xi\mapsto A_\xi(\cdot ,\cdot)$ is in $L^p(Q_\ell^*;L^2(Q_\ell\times Q_\ell))$.

\begin{proof}The first claim follows from  \eqref{eq:bd-Winfm-S2infty} and \eqref{bd:Wsupm-X2infty}
in Appendix~B, thanks to \eqref{eq:bd-x2infty}. Let $g=\fint_{Q_\ell^*}^\oplus g_\xi\,d\xi \in \mathfrak{S}_{2,1}$. From \eqref{eq:bd-Winfm-S2infty}, we know that $W_{<m,g}\in \mathfrak{S}_{2,1}$.
Then, by \eqref{3.eq:5.3.1}, and since $\mathfrak{S}_{2,1}^*=\mathfrak{S}_{2,\infty}$, we get 
\begin{align*}
\MoveEqLeft \widetilde{\Tr}_{L^2}[g \,W_{< m,\widetilde{\gamma}_n}]=\widetilde{\Tr}_{L^2}[W_{<m,g}^*\,\widetilde{\gamma}_n ]\to\widetilde{\Tr}_{L^2}[W_{<m,g}^*\,{\gamma}_* ]=\widetilde{\Tr}_{L^2}[g\,W_{<m,\g_*}].
\end{align*}
The argument is similar for $W_{\geq m, \wg_n}$. Notice that by \eqref{bd:Wsupm-X2infty}, $W_{\geq m, \wg_n}\in \mathfrak{S}_{2,1}$. Then we write 
\[
\widetilde{\Tr}_{L^2}[W_{\geq m, \wg_n}\,g]=\widetilde{\Tr}_{L^2}[W_{\geq m, g}\,\wg_n]=\widetilde{\Tr}_{L^2}[|D^0|^{-1}W_{\geq m, g}|D^0|^{-1}\,|D^0|\wg_n|D^0|], \]
with $|D^0|\wg_n|D^0|$ converging to $|D^0|\g_*|D^0|$ for the weak-$^*$ topology of $\mathfrak{S}_{1,\infty}$ thanks to \eqref{3.eq:5.3.1}. Therefore, it remains to show that the operator $|D^0|^{-1}W_{\geq m, g}|D^0|^{-1}$ belongs to $\mathfrak{S}_{\infty,1}$ whenever $g\in \mathfrak{S}_{2,1}$. Actually, we prove that $W_{\geq m,g}|D^0|^{-1}\in Y$, which gives $|D^0|^{-1}W_{\geq m,g}|D^0|^{-1}\in \mathfrak{S}_{\infty,\infty}$ since $\Vert |D^0|^{-1}\Vert_{\mathfrak{S}_{\infty,\infty}}\leq 1$. We conclude since  $\mathfrak{S}_{\infty,\infty}\subset \mathfrak{S}_{\infty,1} $.  We use Proposition \ref{prop:3.10} on $W_{\geq m,\ell}^\infty$ and focus on the singularity introduced by the potential $G_\ell$, the difference being easy to deal with. For every $\xi\in Q_\ell^*$ and every $\psi_\xi$ and $\phi_\xi$ in  $L^2_\xi(Q_\ell)$ with $\Vert \psi_\xi\Vert_{L^2_\xi}=1$, we have  
\begin{align*} \MoveEqLeft
\left|\iint_{Q_\ell\times Q_\ell}\fint_{Q_\ell^*}\psi_\xi^*(x)\,G_\ell(x-y)\,g_{\xi'}(x,y)\,|D_\xi|^{-1}\phi_\xi(y)\,dxdy d\xi'\right|\\
&\leq \fint_{Q_\ell^*}d\xi'\left(\iint_{Q_\ell\times Q_\ell}|g_{\xi'}(x,y)|^2\,dxdy\right)^{1/2} \left(\iint_{Q_\ell\times Q_\ell}|\psi_\xi(x)|^2\,|G_{\ell}(x-y)|^2\,\big||D_\xi|^{-1}\phi_\xi(y)\big|^2\,dxdy\right)^{1/2}\\
&\leq  C_G\,\Vert g\Vert_{\mathfrak{S}_{2,1}}\Vert\phi_\xi\Vert_{L^2_\xi}\Vert\psi_\xi\Vert_{L^2_\xi},
\end{align*}
thanks to \eqref{3.eq:4.3}. Hence the result.
\end{proof}

We are now proving stronger convergence results  in \eqref{eq:converg}, by improving the bounds on the kernels of $\big(W_{\widetilde{\gamma}_n,\xi}\big)_n.$
\begin{lemma}[Convergence of the sequence $(W_{\geq m,\widetilde{\gamma}_n})_{n\geq 1}$]\label{en:3}
The sequence of kernels 
$\big(W_{\geq m,\widetilde{\gamma}_n,\xi}\big)_n$ is bounded  in $W^{1,\infty}(Q_\ell^*;L^2(Q_\ell\times Q_\ell))$.  Therefore, up to the extraction of a subsequence, we have
\[\Vert |D^0|^{-1/2}W_{\geq m,\widetilde{\gamma}_n-\gamma
_*}|D^0|^{-1/2}\Vert_{\mathfrak{S}_{2,\infty}}\to 0.\]
\end{lemma}
\begin{proof}
  Let us start with  the boundedness of the sequence. We already proved in Lemma~\ref{lem:cvge} that \big($W_{\geq m,\widetilde{\gamma}_n}\big)_n$ is bounded in $\mathfrak{S}_{2,\infty}$, which follows from \eqref{bd:Wsupm-X2infty}. Let us now check that the sequence of norms $\|\nabla_\xi W_{\geq m,\wg_n,\xi}(\cdot,\cdot)\|_{L^\infty(Q_\ell^*;L^2(Q_\ell\times Q_\ell))}$ is also bounded. Thanks to \eqref{bd:nablaWsupm}, for every $\xi\in Q_\ell^*$, 
\begin{equation*}
   \|\nabla_\xi W_{\geq m,\widetilde{\gamma}_n,\xi}(\cdot,\cdot)\|_{L^2(Q_\ell\times Q_\ell)}\leq C\,\fint_{Q_\ell^*} \|\widetilde{\gamma}_{n,\xi'}\|_{\mathfrak{S}_{2}(\xi')}\,d\xi'
  \leq C\, \|\widetilde{\gamma}_{n}\|_{\mathfrak{S}_{2,\infty}},
\end{equation*} 
and we conclude with the help of \eqref{eq:bd-x2infty}. 
Therefore, the kernels $|D_\xi|^{-1/2}W_{\geq m,\widetilde{\gamma}_n,\xi}|D_\xi|^{-1/2}(\cdot,\cdot)$  are   bounded in $W^{1,\infty}(Q_\ell^*;H^{1/2}(Q_\ell\times Q_\ell))$. Thus, according to the Rellich--Kondrachov   and the Arzel\`a--Ascoli theorems, up to the extraction of a subsequence, 
\begin{align*}
    |D_\xi|^{-1/2}W_{\geq m,\widetilde{\gamma}_n,\xi}|D_\xi|^{-1/2}(\cdot,\cdot)\to |D_\xi|^{-1/2}W_{\geq m,\widetilde{\gamma}_*,\xi}|D_\xi|^{-1/2}(\cdot,\cdot)\qquad \textrm{in}\qquad L^\infty(Q_\ell^*;L^2(Q_\ell\times Q_\ell))
\end{align*}
which yields
\(
\,\Vert |D^0|^{-1/2}W_{\geq m,\widetilde{\gamma}_n-\gamma
_*}|D^0|^{-1/2}\Vert_{\mathfrak{S}_{2,\infty}}\to 0,
\)
whence the lemma.
\end{proof}
\begin{lemma}[Convergence of the sequence  $(W_{<m,\widetilde{\gamma}_n})_{n\geq 1}$]\label{en:2}
The sequence of kernels 
$(W_{< m,\widetilde{\gamma}_n,\xi})_{n}$ is bounded in $C^{0,\mu}(Q_\ell^*;H^1_\xi(Q_\ell\times Q_\ell))$ for any $\mu\in (0,1)$. In particular, up to the extraction of a subsequence, 
  \[
\Vert W_{<m,\widetilde{\gamma}_n-\gamma_*}\Vert_{\mathfrak{S}_{2,\infty}}\to 0.\]
\end{lemma}
\begin{proof}
Let $\mu\in (0,1)$. We first show the uniform boundedness of $W_{< m,\widetilde{\gamma}_n,\xi}$ in $C^{0,\mu}(Q_\ell^*;H^1(Q_\ell\times Q_\ell))$. It is based on Lemma \ref{3.lem:5.7.1}, particularly Eq.\eqref{eq:bound-gamma}. In the formulas below, $C$ is a positive constant which depends only on $m$, $\ell$ and $\mu$. Recall that
\[
W_{<m,\ell}^\infty(\eta,z)=\frac{4\pi}{\ell^3}\sum_{\substack{|k|_\infty< m\\ k\in \mathbb{Z}^3}}\frac{1}{\left|\frac{2\pi }{\ell}k-\eta\right|^2}\,e^{i\left(\frac{2\pi }{\ell}k-\eta\right)\cdot z}.
\]
Thus,
\begin{align*}
\MoveEqLeft \|W_{< m,\widetilde{\gamma}_n,\eta}-W_{< m,\widetilde{\gamma}_n,\eta'}\|_{H^1(Q_\ell\times Q_\ell)}\\
\leq& \frac{4\pi}{\ell^3}\sum_{\substack{|k|_\infty\leq m-1\\ k\in\mathbb{Z}^3}}\left\|\fint_{Q_\ell^*} \frac{e^{i\left(\frac{2\pi }{\ell}k-(\eta-\xi')\right)\cdot(x-y)}}{\left|\frac{2\pi}{\ell}k-(\eta-\xi')\right|^2}\,\widetilde{\gamma}_{n,\xi'}\,d\xi'-\fint_{Q_\ell^*} \frac{e^{i\left(\frac{2\pi}{\ell}k-(\eta'-\xi')\right)\cdot(x-y)}}{\left|\frac{2\pi}{\ell}k-(\eta'-\xi')\right|^2}\,\widetilde{\gamma}_{n,\xi'} \,d\xi'\right\|_{H^1(Q_\ell\times Q_\ell)}.
\end{align*}
For each term on the right-hand side, we have
\begin{align*}
\MoveEqLeft\left\|\fint_{Q_\ell^*} \frac{e^{i\left(\frac{2\pi}{\ell}k-(\eta-\xi')\right)\cdot(x-y)}}{\left|\frac{2\pi}{\ell}k-(\eta-\xi')\right|^2}\,\widetilde{\gamma}_{n,\xi'} \,d\xi'-\fint_{Q_\ell^*} \frac{e^{i\left(\frac{2\pi}{\ell}k-(\eta'-\xi')\right)\cdot(x-y)}}{\left|\frac{2\pi}{\ell}k-(\eta'-\xi')\right|^2}\,\widetilde{\gamma}_{n,\xi'} \,d\xi'\right\|_{H^1(Q_\ell\times Q_\ell)}\\
\leq&\fint_{Q_\ell^*}\left|\frac{1}{\left|\frac{2\pi}{\ell}k-(\eta'-\xi')\right|^2}-\frac{1}{\left|\frac{2\pi}{\ell}k-(\eta-\xi')\right|^2}\right| \left\|e^{i\left(\frac{2\pi}{\ell}k-(\eta-\xi')\right)\cdot(x-y)}\widetilde{\gamma}_{n,\xi'}\right\|_{H^1_\eta(Q_\ell\times Q_\ell)}\!\!\!\!d\xi'\\
&+\left\|\fint_{Q_\ell^*} \frac{\big(e^{i\left(\frac{2\pi}{\ell}k-(\eta-\xi')\right)\cdot(x-y)}-e^{i\left(\frac{2\pi}{\ell}k-(\eta'-\xi')\right)\cdot(x-y)}\big)}{\left|\frac{2\pi}{\ell}k-(\eta'-\xi')\right|^2}\,\widetilde{\gamma}_{n,\xi'}\,d\xi'\right\|_{H^1(Q_\ell\times Q_\ell)}.
\end{align*}
As $\eta,\xi'\in Q_\ell^*$, according to \eqref{eq:bound-gamma} we get 
\[
\left\|e^{i\left(\frac{2\pi}{\ell}k-(\eta-\xi')\right)\cdot(x-y)}\widetilde{\gamma}_{n,\xi'}\right\|_{H^1_\eta(Q_\ell\times Q_\ell)}\leq C\left\|\widetilde{\gamma}_{n,\xi'}\right\|_{H_{\xi'}^1(Q_\ell\times Q_\ell)}\leq 2\,C\,j_1(1-\kappa)^{-4}\,c^*(j_1)^2.
\]
By the H\"older continuity of the function $\eta\mapsto \int_{Q_\ell^*}\frac{d\xi'}{|\eta-\xi'|^2}$,
\begin{align*}
 \MoveEqLeft  \fint_{Q_\ell^*}\left|\frac{1}{\left|\frac{2\pi}{\ell}k-(\eta'-\xi')\right|^2}-\frac{1}{\left|\frac{2\pi}{\ell}k-(\eta-\xi')\right|^2}\right| \left\|e^{i\left(\frac{2\pi}{\ell}k-(\eta-\xi')\right)\cdot(x-y)}\widetilde{\gamma}_{n,\xi'}\right\|_{H^1_\eta(Q_\ell\times Q_\ell)}d\xi'\\
 &\leq 2Cj_1(1-\kappa)^{-4}c^*(j_1)^2|\eta-\eta'|^\mu .
\end{align*}
For the last term, note that 
\[
|e^{-i\eta\cdot z}-e^{-i\eta'\cdot z}|\leq \|\nabla_\eta e^{i\eta\cdot z}\|_{L^\infty(Q_\ell^*\times 2Q_\ell)}|\eta-\eta'|\leq C|\eta-\eta'|\] 
and $|\nabla_z (e^{-i\eta\cdot z}- e^{-i\eta'\cdot z})|\leq C|\eta-\eta'|$. We get
\begin{align*}
 \MoveEqLeft\left\|\fint_{Q_\ell^*} \frac{e^{i\left(\frac{2\pi}{\ell}k+\xi'\right)\cdot(x-y)}\big(e^{-i\eta'\cdot (x-y)}-e^{-i\eta'\cdot(x-y)}\big)}{\left|\frac{2\pi}{\ell}k-(\eta'-\xi')\right|^2}\,\widetilde{\gamma}_{n,\xi'}\,d\xi'\right\|_{H^1(Q_\ell\times Q_\ell)}\\
&\leq 2\,C\,j_1(1-\kappa)^{-4}c^*(j_1)^2\,|\eta-\eta'|.
\end{align*}
We finally get
\[
\|W_{< m,\widetilde{\gamma}_n,\xi}\|_{C^{0,\mu}(Q_\ell^*;H^1_\xi(Q_\ell\times Q_\ell))}\leq 2\,C\,j_1(1-\kappa)^{-4}c^*(j_1)^2.
\]
Lastly, thanks to the Arzel\`a–Ascoli  theorem and the Rellich--Kondrachov theorem, the sequence of kernels converges strongly  in $L^\infty(Q_\ell^*;L^2(Q_\ell\times Q_\ell))$,  up to the extraction of a subsequence, and the limit is the operator $W_{< m,\gamma_*}$ thanks to  \eqref{eq:converg}. This concludes the proof of the lemma. 
\end{proof}
Then we have the following.
\begin{lemma}[Strong convergence of the electron--electron interaction] \label{3.lem:5.9} As $n$ goes to infinity, we have 
\[
\||D^0|^{-1/2}V_{\widetilde{\gamma}_n-\gamma_*}|D^0|^{-1/2}\|_{Y}\to0.
\]
\end{lemma}
\begin{proof}
As $V_{\gamma}=G_\ell*\rho_{\gamma}-W_{\gamma}$, we have
\begin{align*}
 \MoveEqLeft \||D^0|^{-1/2}V_{\widetilde{\gamma}_n-\gamma_*}|D^0|^{-1/2}\|_{Y}\\
 &\leq \||D^0|^{-1/2}G_\ell*(\rho_{\widetilde{\gamma}_n}-\rho_{\gamma_*})|D^0|^{-1/2}\|_Y+\||D^0|^{-1/2}(W_{\widetilde{\gamma}_n}-W_{\gamma_*})|D^0|^{-1/2}\|_{Y}.
\end{align*}
Notice that, from Lemma \ref{lem:6.3}, we infer $\rho_{\widetilde{\gamma}_n}\to \rho_{\gamma_*}$ in $L^2(Q_\ell)$. This, together with the fact that $G_\ell\in L^2(Q_\ell)$, yield
\begin{align*}
  \| G_\ell*(\rho_{\widetilde{\gamma}_n}-\rho_{\gamma_*})\|_{L^\infty(Q_\ell)}\to 0.
\end{align*}
Then, using $|D^0|^{-1}\leq 1$, we infer
\begin{align}\label{eq:gamma-n1}
  \||D^0|^{-1/2}G_\ell*(\rho_{\gamma_n}-\rho_{\gamma_*})|D^0|^{-1/2}\|_Y\to 0.
\end{align}
We consider now the second term that we split into two parts. We already proved  in Lemma \ref{en:3} that $\||D^0|^{-1/2}W_{\geq m,{\wg_n}-{\gamma_*}}|D^0|^{-1/2}\|_{\mathfrak{S}_{2,\infty}}\to 0$, which implies the strong convergence in $Y$. The proof for the  other term is even simpler, since 
\[
\||D^0|^{-1/2}W_{< m,{\wg_n}-{\gamma_*}}|D^0|^{-1/2}\|_{\mathfrak{S}_{2,\infty}}\leq \|W_{< m,{\wg_n}-{\gamma_*}}\|_{\mathfrak{S}_{2,\infty}},
\]
for $\Vert |D^0|^{-1/2}\Vert_Y\leq 1$. We conclude by Lemma \ref{en:2}. Finally, we infer 
\begin{align}\label{eq:gamma-n3}
\||D^0|^{-1/2}W_{\widetilde{\gamma}_n-\gamma_*}|D^0|^{-1/2}\|_{Y}\to0.
 \end{align}
The lemma follows gathering together \eqref{eq:gamma-n1} and \eqref{eq:gamma-n3}.
\end{proof}
As a result, we have the following. 
\begin{corollary}[Strong convergence of the spectral projectors]\label{3.cor:5.12}
As $n$ goes to infinity, we have 
\[
\|P_{\gamma_*}^+-P_{\widetilde{\gamma}_n}^+\|_Y\to0.
\]
\end{corollary}
\begin{proof}
For any $\phi_\xi$ and $\psi_\xi\in L^2_\xi$, by \eqref{3.eq:5.2} and the second resolvent identity, we obtain
\begin{align}
\MoveEqLeft\left|\left(\psi_\xi, (P_{\gamma_{*},\xi}^+-P_{\widetilde{\gamma}_n,\xi}^+) \phi_\xi\right)\right|\notag\\
&\leq \frac{1}{2\pi}\int_{-\infty}^{+\infty}\left|\left(\psi_\xi, (D_{\gamma_{*},\xi}-iz)^{-1}V_{\widetilde{\gamma}_n-\gamma_{*},\xi}(D_{\widetilde{\gamma}_n,\xi}-iz)^{-1}\phi_\xi\right)_{L^2_\xi}\right|\,dz\notag\\
&\leq \frac{1}{2\pi}\,\||D^0|^{-1/2}V_{\widetilde{\gamma}_n-\gamma_{*}}|D^0|^{-1/2}\|_Y\notag\\
&\quad {}\times \left(\int_{-\infty}^{+\infty}\|(D_{\gamma_{*},\xi}-iz)^{-1}|D_\xi|^{1/2} \psi_\xi\|_{L^2_\xi}^2\,dz\right)^{1/2}\, \left(\int_{-\infty}^{+\infty}\|(D_{\gamma_{*},\xi}-iz)^{-1}|D_\xi|^{1/2} \phi_\xi\|_{L^2_\xi}^2\,dz\right)^{1/2}\notag\\
&\leq\frac{1}{2}(1-\kappa)^{-1}\||D^0|^{-1/2}V_{\widetilde{\gamma}_n-\gamma_{*}}|D^0|^{-1/2}\|_Y\,\|\phi_\xi\|_{L^2_\xi}\,\|\psi_\xi\|_{L^2_\xi},\label{borne-proj}
\end{align}
in virtue of \eqref{eq:borneD-1Dg}. Therefore, 
\[
\|P_{\gamma_*}^+-P_{\widetilde{\gamma}_n}^+\|_Y\leq \frac{1}{2}(1-\kappa)^{-1}\||D^0|^{-1/2}V_{\widetilde{\gamma}_n-\gamma_{*}}|D^0|^{-1/2}\|_Y.
\]
The right-hand side goes to $0$ by Lemma \ref{3.lem:5.9}.
\end{proof}

\subsubsection{Relationship between \texorpdfstring{$P^+_{\widetilde{\gamma}_n}$}{} and \texorpdfstring{$P^+_{\gamma_n}$}{}}
It remains to analyze the behaviour of  $P^+_{\gamma_n}-P^+_{\widetilde{\gamma}_n}$, as $n$ goes to infinity. Our main result is the following.

\begin{lemma}[Relationship between $P^+_{\widetilde{\gamma}_n}$ and $P^+_{\gamma_n}$]\label{cor:6.8}
  For every $\gamma\in Z$ and for some $C>0$ independent of $n$ and $Z$, it holds
  that
  \begin{align*}
    \|(P^+_{\gamma_n}-P^+_{\wg_n})\gamma\|_{\mathfrak{S}_{1,1}}\leq C\,\|\gamma_n-\widetilde{\gamma}_n\|_{X}^{1/2}\|\gamma\|_{Z}.
  \end{align*}
\end{lemma}

\begin{proof}
 We first observe that $\Vert G \ast (\rho_{\gamma_n}-\rho_{\wg_n})\Vert_{Y}\to 0$, thanks to Lemma~\ref{3.lem:6.5} and \eqref{3.eq:4.0} in Lemma~\ref{3.cor:4.4}. Next, thanks to the bound \eqref{eq:bd-Wsupm}, the sequence of operators $(W_{\geq m, \g_n-\wg_n})_n$ converges strongly to $0$ in $Y$. Therefore, proceeding as for \eqref{3.eq:5.3}, we have
 \begin{equation*}
\frac{1}{2\pi}\left\Vert\int_{-\infty}^{+\infty}(D_{\gamma_{n}}-iz)^{-1}\left(G\ast (\rho_{\g_n}-\rho_{\wg_n})-W_{\geq m,\widetilde{\gamma}_n-\gamma_{n}}\right)(D_{\wg_n}-iz)^{-1}\,dz\right\Vert_{Y}
\leq C\, \Vert \g_n-\wg_n\Vert_{X},
\end{equation*}
with the right-hand side going to $0$. We now focus on the term involving $W_{<m,\g_n-\wg_n}$. For every $\gamma\in Z$, proceeding as for \eqref{borne-proj}, we have
\begin{align*}
\MoveEqLeft\frac{1}{2\pi}\left\Vert\int_{-\infty}^{+\infty}(D_{\gamma_{n}}-iz)^{-1}W_{<m,\widetilde{\gamma}_n-\gamma_{n}}(D_{\wg_n}-iz)^{-1}\,\gamma\,dz\right\Vert_{\mathfrak{S}_{1,1}}\\
&\leq \frac{1}{2\pi}\int_{-\infty}^{+\infty}\left\Vert(D_{\g_n}-iz)^{-1}W_{<m,\widetilde{\gamma}_n-\gamma_{n}}(D_{\wg_n}-iz)^{-1}\gamma\right\Vert_{\mathfrak{S}_{1,1}}\,dz \\
&\leq C \,\left\Vert W_{<m,\widetilde{\gamma}_n-\gamma_{n}}\right\Vert_{\mathfrak{S}_{2,2}}\,\left\Vert \gamma\right\Vert_{\mathfrak{S}_{2,2}}\,\int_{-\infty}^{+\infty} \|(D_{\g_n}-iz)^{-1}\|_{Y} \|(D_{\wg_n}-iz)^{-1}\|_{Y} dz \\
&\leq C\,\left\Vert \g_n-\wg_n\right\Vert_{\mathfrak{S}_{2,2}} \|\gamma\|_{\mathfrak{S}_{2,2}}\\
&\leq C \left\Vert \g_n-\wg_n\right\Vert_{\mathfrak{S}_{1,1}}^{1/2}\,\Vert \gamma\Vert_{Z} 
\end{align*}
thanks to \eqref{eq:bd-Winfm-S2infty} and in virtue of $\Vert \g_n-\wg_n\Vert_Y\leq 2$.
\end{proof}

\subsection{Existence and properties of minimizers for \texorpdfstring{$J_{\leq q}$}{}}\label{sec:6.3}
The existence of minimizers for $J_{\leq q}$ now follows by passing to the limit in the constraint and in the energy. The proof is separated into the following two lemmas.
\begin{lemma}
The limit $\gamma_*$ lies in $\Gamma_{q}^+$. 
\end{lemma}
\begin{proof}
As 
\[
\widetilde{\gamma}_n \overset{\ast}{\rightharpoonup} \gamma_*\quad\textrm{in}\quad X_\infty^2\cap Y,
\]
we get
\[
\|\gamma_*\|_{Y}\leq\liminf_{n\to\infty}\|\widetilde{\gamma}_n\|_{Y}\leq 1
\]
and
\[
\|\gamma_*\|_{X_\infty^2}\leq\liminf_{n\to\infty}\|\widetilde{\gamma}_n\|_{X_\infty^2}\leq j_1(1-\kappa)^{-4}c^*(j_1)^2.
\]
Thus, $\gamma_*\in\Gamma$. Since $\epsilon_P>e_0$, $\widetilde{\Tr}_{L^2}[\gamma_*]=q$ thanks to Lemma~\ref{lem:6.3}.

\medskip

To complete the proof, it remains to show that $P_{\gamma_*}^+\gamma_{*}=\gamma_{*}$. From Lemma \ref{cor:6.8} and since $\Vert \widetilde{\gamma}_n-\gamma_n\Vert_X\to 0$ (Lemma~\ref{3.lem:6.5}),
we first prove that
\begin{equation}\label{cvge-contrainte-tilde}
  \Vert P_{\widetilde{\gamma}_n}^+\,\widetilde{\gamma}_n-\widetilde{\gamma}_n\Vert_{\mathfrak{S}_{1,1}}\to 0.
\end{equation}
Indeed, since $P_{\gamma_n}^+\gamma_n=\gamma_n$, we have
\[\Vert P_{\widetilde{\gamma}_n}^+\,\widetilde{\gamma}_n-\widetilde{\gamma}_n\Vert_{\mathfrak{S}_{1,1}} \leq \Vert (P_{\widetilde{\gamma}_n}^+-P_{\gamma_n}^+)\,\widetilde{\gamma}_n\Vert_{\mathfrak{S}_{1,1}}+ \Vert P_{\gamma_n}^+\,(\widetilde{\gamma}_n-\gamma_n)\Vert_{\mathfrak{S}_{1,1}}+\Vert \gamma_n-\widetilde{\gamma}_n\Vert_{\mathfrak{S}_{1,1}}, 
\]
using also $\Vert P_{\gamma_n}^+\Vert_Y\leq 1$. Then, the right-hand side goes to $0$. Next, by Corollary~\ref{3.cor:5.12},
\begin{equation}\label{eq:contrainte-tilde}
\Vert (P_{\widetilde{\gamma}_n}^+-P_{\gamma_*}^+)\,\widetilde{\gamma}_n\Vert_{\mathfrak{S}_{1,1}}\to 0.
\end{equation}
Let $g\in Y$. Let us show that 
\[
\widetilde{\Tr}_{L^2}\big[(P_{\gamma_*}^+\gamma_{*}-\gamma_{*})\,g\big]=0.
\]
Notice that
\begin{align*}
\left|\widetilde{\Tr}_{L^2}[(P_{\gamma_*}^+\gamma_{*}-\gamma_{*})\,g]\right|&\leq\left|\widetilde{\Tr}_{L^2}[(P_{\widetilde{\gamma}_n}^+-P_{\gamma_*}^+)\widetilde{\gamma}_n\,g]\right|+\left|\widetilde{\Tr}_{L^2} [(P_{\widetilde{\gamma}_n}^+\,\widetilde{\gamma}_n-\widetilde{\gamma}_n)\,g]\right|\\
&\qquad+\left|\widetilde{\Tr}_{L^2}[P_{\gamma_*}^+(\gamma_{*}-\widetilde{\gamma}_{n})\,g]\right|+\left|\widetilde{\Tr}_{L^2}[(\widetilde{\gamma}_{n}-\gamma_{*})\,g]\right|.
\end{align*}
The first two terms in the right-hand side goes to $0$ thanks to  \eqref{cvge-contrainte-tilde} and \eqref{eq:contrainte-tilde}. For the last two terms, we use the weak-$^*$ convergence of $\widetilde{\gamma}_n$ to $\gamma_*$ in $X^2_\infty$ and the fact that $|D^0|^{-1}g|D^0|^{-1}$ and $|D^0|^{-1}P_{\gamma_*}^+g|D^0|^{-1}$ both lie en $\mathfrak{S}_{\infty,1}$. Hence $\gamma_*\in \Gamma_{ q}^+$. This concludes the proof.
\end{proof}
\begin{lemma}\label{3.lem:5.16}
We have 
\begin{equation}\label{eq:lim-energy}
\lim_{n\to +\infty}\left(\mathcal{E}^{DF}(\gamma_n)-\epsilon_P\,\widetilde{\Tr}_{L^2}[\gamma_n]\right)=
\mathcal{E}^{DF}(\gamma_*)-\epsilon_P\,\widetilde{\Tr}_{L^2}[\gamma_*].\end{equation}
Therefore,  $\gamma_*$ is a minimizer of $J_{\leq q}$.
\end{lemma}

\begin{proof}
For the kinetic energy term, we write 
\[
\widetilde{\Tr}_{L^2}[D^0(\gamma_{n}-\gamma_{*})]=\widetilde{\Tr}_{L^2}[D^0(\gamma_{n}-\widetilde{\gamma}_{n})]+\widetilde{\Tr}_{L^2}[D^0(\widetilde{\gamma}_{n}-\gamma_{*})].
\]
Since $\|\gamma_n-\widetilde{\gamma}_n\|_X\to0$ thanks to Lemma~\ref{3.lem:6.5},  $\widetilde{\Tr}_{L^2}[D^0(\gamma_{n}-\widetilde{\gamma}_{n})]\to 0$. On the other hand, by definition of the weak-$^*$ convergence in $X^2_\infty$, $|D^0|(\widetilde{\gamma}_{n}-\gamma_*)|D^0|$ converges to $0$ for the weak-$^*$ topology in $\mathfrak{S}_{1,\infty}$. As $|D^0|^{-1}D^0|D^0|^{-1}\in \mathfrak{S}_{\infty,1}$, we obtain that
\[
\widetilde{\Tr}_{L^2}[D^0(\widetilde{\gamma}_{n}-\gamma_{*})]=\widetilde{\Tr}_{L^2}[|D^0|(\widetilde{\gamma}_{n}-\gamma_{*})|D^0| |D^0|^{-1}D^0|D^0|^{-1}]\to 0,
\]
hence
\[
\widetilde{\Tr}_{L^2}[D^0(\gamma_{n}-\gamma_{*})]\to 0.
\]
In addition, thanks again to Lemma~\ref{lem:6.3},
\[\epsilon_P\,\widetilde{\Tr}_{L^2}[\widetilde{\gamma}_{n}-\gamma_*]\to 0.
\]
The proof for the attractive potential is similar. We start with 
\[
\widetilde{\Tr}_{L^2}[G_\ell(\gamma_{n}-\gamma_{*}) ]=\widetilde{\Tr}_{L^2}[G_\ell(\gamma_{n}-\widetilde{\gamma}_{n}) ]+\widetilde{\Tr}_{L^2}[G_\ell(\widetilde{\gamma}_{n}-\gamma_{*})].
\]
The second term goes to $0$ as $n$ goes to infinity since $\rho_{\widetilde{\gamma}_n}$ converges to $\rho_{\gamma_*}$ in $L^2(Q_\ell)$ and $G\in L^2(Q_\ell)$ (Lemma~\ref{lem:6.3}). For the first term, we use the fact that 
\[
\left|\widetilde{\Tr}_{L^2}[G_\ell(\gamma_{n}-\widetilde{\gamma}_{n}]\right|\leq C_H\|\gamma_n-\widetilde{\gamma}_n\|_X\to0,
\]
and we conclude by Lemma~\ref{3.lem:6.5}.\medskip

\noindent
For the repulsive potential, using the fact that $\widetilde{\Tr}_{L^2}[V_{\widetilde{\gamma}}\,\gamma']=\widetilde{\Tr}_{L^2}[V_{\gamma'}\,{\widetilde{\gamma}}]$ whenever $\gamma$ and $\gamma'$ are in $Z$, we have 
\[
\begin{aligned}
\MoveEqLeft \left|\widetilde{\Tr}_{L^2}[V_{\gamma_n}\gamma_{n}-V_{\gamma_*}\gamma_{*}]\right|\\
&\leq \left|\widetilde{\Tr}_{L^2}[V_{\gamma_n}\gamma_{n}-V_{\widetilde{\gamma}_n}\widetilde{\gamma}_n]\right|+\left|\widetilde{\Tr}_{L^2}[V_{\widetilde{\gamma}_n}\widetilde{\gamma}_n-V_{\gamma_*}\gamma_{*}]\right|\\
&= \left|\widetilde{\Tr}_{L^2}[V_{\widetilde{\gamma}_n+\gamma_{n}}(\gamma_{n}-\widetilde{\gamma}_n)]\right| + \left |\widetilde{\Tr}_{L^2}[V_{\widetilde{\gamma}_n-\gamma_*}(\widetilde{\gamma}_n+\gamma_*)]\right|\\
&\leq C_{EE}\,(\|\widetilde{\gamma}_n\|_{Z}+\|\gamma_n\|_{Z})\|\widetilde{\gamma}_n-\gamma_n\|_{X}+\||D^0|^{-1/2}V_{\widetilde{\gamma}_n-\gamma_*}|D^0|^{-1/2}\|_Y(\|\widetilde{\gamma}_n\|_{X}+\|\gamma_*\|_{X})
\end{aligned}
\]
using the bound \eqref{3.eq:5.1'} in Lemma~\ref{3.lem:5.1}. Finally, the right-hand side  in the above string of inequalities goes to $0$, according to Lemmas \ref{3.lem:6.5} and  \ref{3.lem:5.9}. The lemma follows.
\end{proof}
Since $\gamma_*$ is a minimizer of $J_{\leq q}$, we may apply Proposition \ref{3.lem:6.3} to the constant minimizing sequence equal to $\gamma_*$ , and we obtain that \[
\int_{Q_\ell^*}\Tr_{L^2_\xi}\big[(D_{\gamma_{*},\xi}-\epsilon_{P})\gamma_{*,\xi}\big]\,d\xi=\inf_{\substack{\gamma\in \Gamma_{\leq q}\\\gamma=P_{\gamma_{*}}^+\gamma}}\int_{Q_\ell^*}\Tr_{L^2_\xi}\big[(D_{\gamma_{*},\xi}-\epsilon_{P})\gamma_{\xi}\big]\,d\xi.
\]
We may then apply Lemma~\ref{3.lem:6.4} to conclude that 
 $\gamma_{*}=\fint_{Q_\ell^*}^\oplus\mathbbm{1}_{[0,\nu)}(D_{\gamma_*,\xi})\,d\xi+\delta$ with $0\leq \delta \leq \mathbbm{1}_{\{\nu\}}(D_{\gamma_*})$ for some $\nu\in [\lambda_0,e_0]$ that is independent of $\epsilon_P$. This ends the proof of Theorem \ref{3.th:2.7'}. Thus, Theorem \ref{3.th:2.7} holds true.

\medskip

{\bf Acknowledgments.} The authors are grateful to the referees for a careful reading of the manuscript and for remarks and suggestions that greatly improved the quality of this paper. L.M. acknowledges support from the European Research Council (ERC) under the European Union's Horizon 2020 research and innovation program (Grant agreement No. 810367).\bigskip

{\bf Declarations.}\bigskip

{\bf Conflict of interest} The authors declare no conflict of interest.\bigskip

{\bf Data Availability} Data sharing is not applicable to this article as it has no associated data.

\appendix

\section{Proof of Lemma \ref{3.lem:4.2}}\label{sec:A}
It suffices to prove \eqref{3.eq:4.3}. By interpolation, we can choose $C_H=C_G$.

To deal with \eqref{3.eq:4.3}, the idea is to decompose the potential $G_\ell$ on $Q_\ell$ into two parts, namely $\frac{1}{|x|}$ and $G_\ell-\frac{1}{|x|}$. The first term can be treated as the Hardy inequality on $Q_\ell$, whereas the second is bounded. We begin with the second term and prove the following.
\begin{lemma}\label{lem:Gbound}
There is a constant $C_0>0$ independent of $\ell$ such that
\begin{align*}
  \sup_{x\in Q_\ell}\left|G_\ell(x)-\frac{1}{|x|}\right|\leq \frac{C_0}{\ell}.
\end{align*}
This implies that, for any $x\in \mathbb{R}^3$,
\begin{equation}\label{3.eq:4.1}
  G_\ell(x)\geq -\frac{C_0}{\ell}.
\end{equation}
In particular, we have 
\[
C_0\leq \inf_{0<R<\frac{1}{2}}\left(\frac{3}{2R}+\frac{2\pi R^2}{5}+\frac{3}{4\pi^2 R^3}\min\left\{\frac{4\pi\,R^3}{3};1-\frac{4\pi\,R^3}{3}\right\}^{1/2}\left(\sum_{k\in\mathbb{Z}^3\setminus\{0\}}\frac{1}{|k|^4}\right)^{1/2}\right).
\]
\end{lemma}
\begin{proof}
As $G_1(x)=\ell \,G_\ell(\ell x)$, it suffices to consider the case $\ell=1$. Let $f(x)=G_1(x)-\frac{1}{|x|}$. Equation \eqref{eq:def-G} yields 
\[
-\Delta f= 4\pi\Big(-1+\sum_{k\in \mathbb{Z}^3\setminus\{0\}}\delta_k\Big ).
\]
Let $B(z,R)$ be a ball of center $z$ and radius $R$ chosen such that $\big(\bigcup_{z\in Q_1}B(z,R)\big)\cap (\mathbb{Z}^3\setminus\{0\})=\emptyset$. Obviously, we can assume $0< R<1/2$.  On the one hand, by the divergence theorem, for $0\leq r\leq R$ and $z\in Q_1$ we obtain
\begin{align}
 \frac{1}{4\pi}\frac{d}{dr}\int_{\mathbb{S}^2}f(z+r\omega)d\omega=\frac{1}{4\pi r^2}\int\limits_{\partial B(z,r)}\frac{\partial f(s)}{\partial n}\,ds=\frac{1}{4\pi r^2}\int\limits_{B(z,r)}\Delta_x f(x)\,dx\label{eq:A2}
\end{align}
with $\mathbb{S}^2$ denoting the unit sphere. For any $z\in Q_1$,
\[
\frac{1}{4\pi r^2} \int_{B(z,r)}\Delta_x f\,dx = \frac{1}{r^2} 
 \int_{B(z,r)} 1\,dx
= \frac{4\pi}{3}r,
\]
where the first equation holds since 
\[
\Big(\bigcup_{z\in Q_1}B(z,r)\Big)\cap (\mathbb{Z}^3\setminus\{0\})=\emptyset,\quad \textrm{for}\, 0\leq r\leq R.
\]
Therefore, integrating \eqref{eq:A2} with respect to $r$,
\[ \int_{\mathbb{S}^2}f(z+r\omega)\,d\omega -4\pi\,f(z)= \frac{8\,\pi^2}{3}\,r^2.
\]
Since $\int_{B(z,R)}f(x)\,dx=\int_0^R r^2\left(\int_{\mathbb{S}^2}f(z+r\omega)\,d\omega\right)\,dr$, integration over $[0,R]$ leads to
\begin{align*}
  |f(z)|\leq \frac{3}{4\pi R^3}\left|\int_{B(z,R)}f(x)\, dx\right|+\frac{2\pi R^2}{5}\leq \frac{3}{4\pi R^3}\left|\int_{B(z,R)}G_1(x)\, dx\right|+\frac{3}{4\pi R^3}\left|\int_{B(z,R)}\frac{1}{|x|}\, dx\right|+\frac{2\pi R^2}{5}.
\end{align*}
On the other hand,
\begin{align*}
\left|\int_{B(z,R)}G_1(x)\, dx\right|&\leq |B(z,R)|^{1/2}\|G_1\|_{L^2(B(z,R))}\\&\leq \left(\frac{4\pi\,R^3}{3}\right)^{1/2}\|G_1\|_{L^2(Q_1)}=\frac{1}{\pi}\left(\frac{4\pi\,R^3}{3}\sum_{k\in\mathbb{Z}^3\setminus\{0\}}\frac{1}{|k|^4}\right)^{1/2}.
\end{align*}
Using \eqref{eq:constante-G} and the periodicity of $G_1$, we also have
\begin{align*}
 \left|\int_{B(z,R)}G_1(x)\, dx\right|&=\left|\int_{(z+Q_1)\setminus B(z,R)}G_1(x)\, dx\right|\\&\leq \frac{1}{\pi}\left(1-\frac{4\pi\,R^3}{3}\right)^{1/2}\left(\sum_{k\in\mathbb{Z}^3\setminus\{0\}}\frac{1}{|k|^4}\right)^{1/2}.
\end{align*}
Thus,
\[
\left|\int_{B(z,R)}G_1(x)\, dx\right|\leq \frac{1}{\pi}\min\left\{\left(\frac{4\pi\,R^3}{3}\right)^{1/2},\left(1-\frac{4\pi\,R^3}{3}\right)^{1/2}\right\}\left(\sum_{k\in\mathbb{Z}^3\setminus\{0\}}\frac{1}{|k|^4}\right)^{1/2}.
\]
Furthermore, 
\begin{align*}
  \left|\int_{B(z,R)}\frac{1}{|x|} \,dx\right|\leq \int_{B(0,R)}\frac{1}{|x|}\,dx=2\pi R^2.
\end{align*}
Therefore, the bound holds for $\ell=1$ and any $0<R<\frac{1}{2}$ with 
\begin{align*}
  \MoveEqLeft C_0\leq \frac{3}{2R}+\frac{2\pi R^2}{5}+\frac{3}{4\pi^2 R^3}\min\left\{\frac{4\pi\,R^3}{3};1-\frac{4\pi\,R^3}{3}\right\}^{1/2}\left(\sum_{k\in\mathbb{Z}^3\setminus\{0\}}\frac{1}{|k|^4}\right)^{1/2}.
\end{align*}
\end{proof}

We now consider the Hardy inequality on $Q_\ell$ for the potential $\frac{1}{|x|}$.
\begin{lemma}\label{hardy} Let $u\in H^1(Q_\ell)$, then
\begin{align*}
  \left\|\frac{u}{|x|}\right\|_{L^2(Q_\ell)}^2\leq \frac{4\ell+12}{\ell}\|\nabla u\|_{L^2(Q_\ell)}^2+\frac{24+12\ell}{\ell^2}\|u\|_{L^2(Q_\ell)}^2.
\end{align*}
\end{lemma}
\begin{proof}
We start with the relationship:
\[
0\leq \int_{Q_\ell}\left|\nabla u+ \frac{x \,u}{2|x|^2}\right|^2d x.
\]
Thus,
\[
0\leq \int_{Q_\ell}|\nabla u|^2 \,dx+\frac{1}{4}\int_{Q_\ell}\frac{|u|^2}{|x|^2}\,dx+\frac{1}{2}\int_{Q_\ell}\frac{\nabla |u|^2 \cdot x}{|x|^2}
\,dx.\]
By the divergence theorem for $\int_{Q_\ell}\nabla\cdot(\frac{|u|^2 x}{|x|^2})\,dx$, we obtain
\[
\int_{Q_\ell}\frac{\nabla |u|^2 \cdot x}{|x|^2}\,dx= \int_{\partial Q_\ell}\frac{\Vec{n} x\,|u|^2}{|x|^2}\,dx-\int_{Q_\ell}\frac{|u|^2}{|x|^2}\,dx
\]
where $\Vec{n}$ is the outward pointing unit normal at each point on the boundary $\partial Q_\ell$. To end this proof, it suffices to estimate $\int_{\partial Q_\ell}\frac{\Vec{n} x|u|^2}{|x|^2}$.

Let 
\[
A^{2,3}(x_1)=\int_{(-\frac{\ell}{2},\frac{\ell}{2}]^2}|u|^2(x_1,x_2,x_3)dx_2dx_3.
\]
As $|\Vec{n}\cdot x|=\frac{\ell}{2}$ and $|x|\geq \frac{\ell}{2}$ for any $x\in\partial Q_\ell$, we have
\begin{align}
  &\left|\;\int\limits_{\partial Q_\ell}\frac{\Vec{n} x|u|^2}{|x|^2}\,dx\right|\leq \frac{2}{\ell}\int\limits_{\partial Q_\ell}|u|^2\,dx\notag \\
  &\qquad=\frac{2}{\ell}\left( A^{2,3}\left(-\frac{\ell}{2}\right)+A^{2,3}\left(\frac{\ell}{2}\right)+ A^{1,3}\left(-\frac{\ell}{2}\right)+A^{1,3}\left(\frac{\ell}{2}\right)+A^{1,2} \left(-\frac{\ell}{2}\right)+A^{1,2}\left(\frac{\ell}{2}\right)\right).\label{eq:A3}
\end{align}
Let $x_1^{(0)}\in (-\frac{\ell}{2},\frac{\ell}{2}]$ such that 
\[
 A^{2,3}(x_1^{(0)})\leq\fint_{(-\frac{\ell}{2},\frac{\ell}{2}]}A^{2,3}(x_1)\,dx_1= \frac{1}{\ell}\int_{(-\frac{\ell}{2},\frac{\ell}{2}]}A^{2,3}(x_1)\,dx_1=\frac{1}{\ell}\int_{Q_\ell}|u|^2\,dx.
\]
Then we have
\begin{align*}
  A^{2,3}\left(-\frac{\ell}{2}\right)+A^{2,3}\left(\frac{\ell}{2}\right)&=\left[\int_{x_1^{(0)}}^{\frac{\ell}{2}}-\int_{-\frac{\ell}{2}}^{x_1^{(0)}}\right] \frac{d}{dx_1} A^{2,3}dx_1+2A^{2,3}(x_1^{(0)})\\
  &\leq 2 A^{2,3}(x_1^{(0)})+\int_{(-\frac{\ell}{2},\frac{\ell}{2}]}\left|\frac{d}{dx_1}A^{2,3}\right|dx_1.
\end{align*}
As
\[
\int_{(-\frac{\ell}{2},\frac{\ell}{2}]}\left|\frac{d}{dx_1}A^{2,3}(x_1)\right|dx_1\leq 2\int_{Q_\ell} |u|\left|\frac{\partial}{\partial x_1}u\right|\leq 2\|u\|_{L^2(Q_\ell)}\|\nabla u\|_{L^2(Q_\ell)},
\]
we get
\[
A^{2,3}\left(-\frac{\ell}{2}\right)+A^{2,3}\left(\frac{\ell}{2}\right)\leq \frac{2}{\ell}\|u\|_{L^2(Q_\ell)}^2+2\|u\|_{L^2(Q_\ell)}\|\nabla u\|_{L^2(Q_\ell)}.
\]
Inserting this into \eqref{eq:A3}, we can conclude that
\[
0\leq \|\nabla u\|_{L^2(Q_\ell)}^2-\frac{1}{4}\left\|\frac{u}{|x|}\right\|_{L^2(Q_\ell)}^2+\frac{6}{\ell^2}\|u\|^2_{L^2(Q_\ell)}+\frac{6}{\ell}\|u\|_{L^2(Q_\ell)}\|\nabla u\|_{L^2(Q_\ell)}.
\]
As a result, by the Cauchy--Schwarz inequality
\begin{align*}
  \left\|\frac{u}{|x|}\right\|_{L^2(Q_\ell)}^2\leq \frac{4\ell+12}{\ell}\|\nabla u\|_{L^2(Q_\ell)}^2+\frac{24+12\ell}{\ell^2}\|u\|_{L^2(Q_\ell)}^2.
\end{align*}
\end{proof}

Actually, Lemma \ref{hardy} implies that for any $\xi\in Q_\ell^*$,
\begin{align*}
\||x|^{-1}u_\xi\|_{L^2_\xi}&\leq \left(\max\left\{\frac{24+12\ell}{\ell^2};\frac{4\ell+12}{\ell}\right\}\right)\|(1-\Delta_\xi)^{1/2}u_\xi\|_{L^2_\xi}\\
&= 2\max\left\{\sqrt{\frac{6}{\ell^2}+\frac{3}{\ell}};\sqrt{1+\frac{3}{\ell}}\right\}\|u_\xi\|_{H^1_\xi}.
\end{align*}
Thus combining Lemmas \ref{lem:Gbound} and \ref{hardy}, we know that
\begin{align*}
  \left\|G_\ell u_\xi\right\|_{L^2_{\xi}}&\leq \frac{C_0}{\ell}\left\|u_\xi\right\|_{L^2_{\xi}}+2\max\left\{\sqrt{\frac{6}{\ell^2}+\frac{3}{\ell}};\sqrt{1+\frac{3}{\ell}}\right\}\|u_\xi\|_{H^1_\xi}\\
  &\leq C_G\|u_\xi\|_{H^1_\xi}.
\end{align*}
with
\begin{align}\label{const:CG}
  C_G:=2\left( 1+\frac{C_0}{\ell}\right)\max\left\{\sqrt{1+\frac{3}{\ell}};\sqrt{\frac{3}{\ell}+\frac{6}{\ell^2}}\right\}.
\end{align}

We now turn to the estimates on the operator $W_{\gamma}$. \section{Proof of Lemma \ref{3.lem:4.5}}\label{sec:B}

We first study the properties of $W_\ell^\infty$, then we prove Lemma \ref{3.lem:4.5}. In passing, we correct wrong estimates in \cite{catto2001thermodynamic}. Indeed, contrary to what was claimed, the function $W_{\infty}(\eta, x)-e^{-i \eta \cdot x} G(x)-4 \pi \frac{e^{-i \eta \cdot x}}{|\eta|^2}$ introduced in \cite[p. 745, Eq. (5.3)]{catto2001thermodynamic} is not harmonic in $(1+\epsilon)\,Q_\ell$ with respect to the $x$-variable. Therefore, some arguments need to be modified.
\subsubsection{Properties of \texorpdfstring{$W_\ell^\infty$}{}}
Recall that $W_\ell^\infty=W_{\geq m,\ell}^\infty+W_{<m,\ell}^\infty$ is given by \eqref{3.eq:4.7'}. We are going to prove the Hardy type inequalities for $W_{\geq m,\ell}^\infty$. A natural idea is to compare it with the potential $G_\ell$.
\begin{proposition}[Singularities for the potential $W_{\geq m,\ell}^\infty$]\label{prop:minW}
For every $m\geq 2$, there exists a positive constant $C_{\geq m}$ such that, for any $\ell>0$, we have
\begin{equation}\label{prop:3.10}
  \sup_{\substack{\eta\in 2Q_\ell^*\\x\in Q_\ell}} \Big|W_{\geq m,\ell}^\infty(\eta,x)-G_\ell(x)\Big|\leq \frac{C_{\geq m}}{\ell}
\end{equation}
with
\begin{align*}
  \MoveEqLeft C_{\geq m}\leq \inf_{0<R<1/2}\left\{\frac{\sqrt{3}}{2(\pi R)^{3/2}}\frac{m^2+2}{(m-1)^2}\left(\sum_{|k|_\infty\geq m}\frac{1}{|k|^4}\right)^{1/2}+\frac{2\pi[(2m-1)^3+1] R^2}{5}\right.\\
  &\qquad+\left.\frac{3}{4\pi^2 R^3}\min\left\{\frac{4\pi\,R^3}{3};1-\frac{4\pi\,R^3}{3}\right\}^{1/2}\left(\sum_{k\in\mathbb{Z}^3\setminus\{0\}}\frac{1}{|k|^4}\right)^{1/2}\right\}.
\end{align*}
\end{proposition}

\begin{proof}
The proof is similar to Lemma \ref{lem:Gbound}. Notice that 
\[ W_\ell^\infty(\eta,x)=\lambda W_{\lambda \ell}^\infty\left(\frac{\eta}{\lambda},\lambda x\right),\quad\eta\in \mathbb{R}^3,\,x\in \mathbb{R}^3
.\]
We therefore take $\ell=1$. Observe, from \eqref{3.eq:2.8}, that
\[
-\Delta_z W_1^\infty(\eta,x)=4\pi\sum_{k\in\mathbb{Z}^3}e^{-i\eta\cdot k}\delta_{k}(x).
\]
Let $f(\eta,x)=W_{\geq m,1}^\infty(\eta,x)-G_1(x)$. Then
\[
-\Delta_x f(\eta,x)=4\pi\sum_{\substack{k\neq 0\\ k\in\mathbb{Z}^3}}(e^{-i\eta\cdot k}-1)\delta_k(x)+4\pi-4\pi\sum_{\substack{|k|_\infty < m\\ k\in\mathbb{Z}^3}}e^{i(2\pi k-\eta)\cdot x}.
\]
Let $B(z,R)$ be a ball of center $z$ and radius $R$ chosen such that $\big(\bigcup_{z\in Q_1}B(z,R)\big)\cap (\mathbb{Z}^3\setminus\{0\})=\emptyset$. Obviously, we can assume $0< R<1/2$. Analogously to \eqref{eq:A2}, for $0\leq r\leq R$ and $z\in Q_1$ we obtain
\begin{equation}\label{eq:divthm-f}
  \begin{aligned}
 \MoveEqLeft  \frac{d}{dr}\left(\frac{1}{4\pi r^2}\int_{\partial B(z,r)}f(\eta,s)\, ds\right)=\frac{1}{4\pi r^2}\int_{B(z,r)}\Delta_x f(\eta,x)\,dx.
\end{aligned}
\end{equation}
On the one hand, for any $z\in Q_1$,
\[
\frac{1}{4\pi r^2}\left\vert \int_{B(z,r)}\Delta_x f\,dx\right\vert = \frac{1}{r^2}\left\vert 
 \int_{B(z,r)} \Big(1-\sum_{\substack{k\in\mathbb{Z}^3\\|k|_\infty \leq m-1 }} e^{i(2\pi k-\eta)\cdot x}\Big)\,dx\right \vert 
\leq \frac{4\pi[(2m-1)^3+1]}{3}\,r,
\]
where the first equality holds since 
\[
\Big(\bigcup_{z\in Q_1}B(z,r)\Big)\cap (\mathbb{Z}^3\setminus\{0\})=\emptyset,\quad \textrm{for}\, 0\leq r\leq R.
\]
Therefore, integrating \eqref{eq:divthm-f} with respect to $r$,
\[
-\frac{8\,\pi^2\,[(2m-1)^3+1]}{3}\,r^2\leq \int_{\mathbb{S}^2}f(\eta,z+r\omega)\,d\omega -4\pi\,f(\eta,z)\leq \frac{8\,\pi^2\,[(2m-1)^3+1]}{3}\,r^2.
\]
Then integration over $[0,R]$ leads to
\begin{align*}
  |f(\eta,z)|\leq \frac{3}{4\pi R^3}\left|\int_{B(z,R)}f(\eta,x)\, dx\right|+\frac{2\pi[(2m-1)^3+1] R^2}{5}.
\end{align*}
On the other hand,
\[
\left|\int_{B(z,R)}G_1(x)\, dx\right|\leq \frac{1}{\pi}\min\left\{\left(\frac{4\pi\,R^3}{3}\right)^{1/2},\left(1-\frac{4\pi\,R^3}{3}\right)^{1/2}\right\}\left(\sum_{k\in\mathbb{Z}^3\setminus\{0\}}\frac{1}{|k|^4}\right)^{1/2}.
\]
Furthermore, according to the quasi-periodicity of $W_{\geq m,1}$ with respect to $z\in \mathbb{R}^3$, for any $\eta\in 2Q_1^*$,
\begin{align*}
 \MoveEqLeft \left|\int_{B(z,R)}W_{\ge m,1}^\infty(\eta,x)\, dx\right|
 \leq |B(z,R)|^{1/2}\|W_{\geq m,1}^\infty\|_{L^2(B(z,R))}\leq \left(\frac{4\pi\,R^3}{3}\right)^{1/2}\|W_{\geq m,1}^\infty\|_{L^2(Q_1)}\\
&\leq 4\pi\left(\frac{4\pi\,R^3}{3}\right)^{1/2}\,\left(\sum_{|k|_\infty\geq m}\frac{1}{|2\pi k-\eta|^4}\right)^{1/2}\\
&\leq 4\pi \left(\frac{4\pi\,R^3}{3}\right)^{1/2}\,\sup_{\substack{|k|_\infty\geq m\\ \eta\in 2Q_1^*}}\frac{|2\pi k|^2}{|2\pi k-\eta|^2}\left(\sum_{|k|_\infty\geq m}\frac{1}{|2\pi k|^4}\right)^{1/2}\\
&= \left(\frac{4\pi\,R^3}{3}\right)^{1/2}\,\frac{m^2+2}{\pi(m-1)^2}\left(\sum_{|k|_\infty\geq m}\frac{1}{| k|^4}\right)^{1/2}.
\end{align*}
Therefore, the bound \eqref{prop:3.10} holds for $\ell=1$ with 
\begin{align*}
 C_{\geq m}&\leq \frac{\sqrt{3}}{2(\pi R)^{3/2}}\frac{m^2+2}{(m-1)^2}\left(\sum_{|k|_\infty\geq m}\frac{1}{|k|^4}\right)^{1/2}+\frac{2\pi[(2m-1)^3+1] R^2}{5}\\
  &\quad+\frac{3}{4\pi^2 R^3}\min\left\{\left(\frac{4\pi\,R^3}{3}\right)^{1/2},\left(1-\frac{4\pi\,R^3}{3}\right)^{1/2}\right\}\left(\sum_{k\in\mathbb{Z}^3\setminus\{0\}}\frac{1}{|k|^4}\right)^{1/2},
\end{align*}
for any $0<R<\frac{1}{2}$.
The corresponding result for any $\ell>0$ follows immediately by a scaling argument.
\end{proof}

We can immediately conclude from Lemma \ref{3.lem:4.2} and Proposition \ref{prop:minW} the following.
\begin{corollary}[Hardy-type inequalities for the potential $W_{\geq m,\ell}^\infty$]\label{cor:hardy}
For $m\geq 2$, we have
\begin{equation}\label{eq:W1}
 \||W_{\geq m,\ell}^\infty|^{1/2}|D_\xi|^{-1/2}\|_Y\leq \left(C_H+\frac{C_{\geq m}}{\ell}\right)
\end{equation}
and
\begin{equation}\label{eq:W2}
  \|W_{\geq m,\ell}^\infty |D_\xi|^{-1}\|_Y \leq \left( C_G+\frac{C_{\geq m}}{\ell}\right).
\end{equation}
\end{corollary}

We also have the following estimate on $W_{\geq m,\ell}^\infty$.
\begin{lemma}\label{lem:3.12}
Let $m\geq 2$. There is a positive constant $C=C(\ell,m)$ such that
\begin{equation}
  \label{bd:nablaWsupm}
\sup_{\eta\in2Q_\ell^*}\|\nabla_\eta W_{\geq m,\ell}^\infty(\eta,\cdot)\|_{L^\infty(Q_\ell)}\leq C.
\end{equation}
\end{lemma}
\begin{proof}
Take $\ell=1$ for simplicity. Note that 
\[
-\Delta_x \nabla_\eta W_{\geq m,1}^\infty(\eta,x)=-4\pi \sum_{k\in\mathbb{Z}^3\setminus\{0\}}ik e^{-i\eta\cdot k}\delta_k(x)+4\pi\sum_{\substack{|k|_\infty<m\\k\in\mathbb{Z}^3}}ixe^{i(2\pi k-\eta)\cdot x},
\]
from which we obtain 
\[
|\Delta_x \nabla_\eta W_{\geq m,1}^\infty(\eta,x)|\leq C
\]
for any $\eta\in 2Q_\ell^*$ and $x\in Q_\ell$.
Following the proof of Lemma \ref{prop:minW}, we know 
\[
|\nabla_\eta W_{\geq m,1}^\infty(\eta,x)|\leq C.
\]
The corresponding result for any $\ell>0$ follows immediately by a scaling argument as for Lemma \ref{prop:minW}.
\end{proof}

\subsubsection{Estimates for the exchange term}
We consider now the exchange term. Let $\psi_\xi\in H^{1/2}_\xi$. As 
\begin{equation}\label{eq:ps-W}
\|W_{\gamma,\xi}\psi_\xi\|_{L_\xi^{2}}=\sup_{\phi_\xi\in L_\xi^{2},\; \|\phi_\xi\|_{L^2_\xi}=1}|(\phi_\xi,W_{\gamma,\xi}\psi_\xi)|,
\end{equation}
we only need to study the inner product $(W_{\gamma,\xi}\psi_\xi,\phi_\xi)$. For $m\geq 2$, $\eta\in \mathbb{R}^3$ and $z\in \mathbb{R}^3$,
\begin{equation}\label{3.eq:4.7''}
W_\ell^\infty(\eta,z)=W_{\geq m,\ell}^\infty(\eta,z)+W_{<m,\ell}^\infty(\eta,z).
\end{equation}

For the term that carries all singularities in the $x$ variable (i.e., $W_{\geq m,\ell}^\infty$), we use the decomposition \eqref{eq:decomp-gamma} and Corollary \ref{cor:hardy}. Let $\gamma\in X$ with $\gamma^*=\gamma$. For any $\xi\in Q_\ell^*$ we have
\begin{equation}
\label{eq:decomp-gamma-X}
|D_\xi|^{1/2}\gamma_{\xi}|D_\xi|^{1/2}=\sum_{n\geq 1}\lambda_n(\xi)\left|v_n(\xi,\cdot)\rangle\right.\left.\langle v_n(\xi,\cdot)\right|
\end{equation}
with $\left(v_{n}(\xi,\cdot),v_{m}(\xi,\cdot)\right)_{L^2_\xi}=\delta_{m,n}$ and $\|\gamma\|_{X}=\fint_{Q_\ell^*}\sum_{n\geq 1}|\lambda_n(\xi)|\,d\xi$. Hence  
\[
\gamma_{\xi}=\sum_{n\geq 1}\lambda_n(\xi)\left|u_n(\xi,\cdot)\left>\right<u_n(\xi,\cdot)\right|
\]
with $u_n(\xi,\cdot)=|D_\xi|^{-1/2}v_n(\xi,\cdot)$.
Now, we have
\begin{align}\label{eq:B9}
\MoveEqLeft\left\vert\fint_{Q^*_\ell}d\xi'\iint\limits_{Q_\ell\times Q_\ell}W_{\geq m,\ell}^\infty(\xi-\xi',x-y)\phi_\xi^*(x)\gamma_{\xi'}(x,y)\psi_\xi(y)\,dxdy\right\vert\notag\\
& \leq \fint\limits_{Q^*_\ell}d\xi' \sum_{n\geq 1}|\lambda_n(\xi')|\!\!\!\!\!\!\int\limits_{Q_\ell\times Q_\ell}\!\!\!|W_{\geq m,\ell}^\infty(\xi-\xi',x-y)|\,|u_n(\xi',x)||u_n(\xi',y)|\,|\psi_\xi(y)|\,|\phi_\xi(x)|\,dxdy\notag\\
&\leq \fint_{Q^*_\ell}d\xi' \sum_{n\geq 1}|\lambda_n(\xi')| \left(\int_{Q_\ell}|\psi_\xi(y)|^2 \,dy \int_{Q_\ell}|W_{\geq m,\ell}^\infty(\xi-\xi',x-y)|\,|u_n(\xi',x)|^2\,dx \right)^{1/2}\notag\\
&\quad {}\times\left(\int_{Q_\ell}|\phi_\xi(x)|^2\,dx \int_{Q_\ell}|W_{\geq m,\ell}^\infty(\xi-\xi',x-y)|\,|u_n(\xi',y)|^2\, dy\right)^{1/2} \notag\\
&\leq \left(\fint_{Q^*_\ell}\left(C_H+\frac{C_{\geq m}}{\ell}\right) \sum_{n\geq 1}|\lambda_n(\xi')|\,\||D_{\xi'}|^{1/2}u_n(\xi',\cdot)\|_{L^2(\xi')}^2\, d\xi'\right)\|\psi_\xi\|_{L^2_\xi}\,\|\phi_\xi\|_{L^2_\xi}\notag\\
&\leq \left(C_H+\frac{C_{\geq m}}{\ell}\right)\,\|\gamma\|_X\,\|\psi_\xi\|_{L^2_\xi}\,\|\phi_\xi\|_{L^2_\xi},
\end{align}
with the help of \eqref{eq:W1}. Thus, for every $\gamma\in X$ with $\gamma=\gamma^*$, we have 
\begin{equation}\label{eq:bd-Wsupm}
\Vert W_{\geq m,\g}\Vert_{Y}\leq \big(C_H+\frac{C_{\geq m}}{\ell}\big)\,\|\gamma\|_X.
\end{equation}
Using the Cauchy--Schwarz inequality and \eqref{eq:W2}, we can also argue as follows whenever $\gamma\in\mathfrak{S}_{1,1}$ with $\gamma^*=\gamma$:
\begin{align}\label{eq:B10}
&\left\vert\fint\limits_{Q^*_\ell}d\xi'\!\!\iint\limits_{Q_\ell\times Q_\ell}W_{\geq m,\ell}^\infty\Big(\xi-\xi',x-y\Big)\phi_\xi^*(x)\gamma_{\xi'}(x,y)\,(|D_\xi|^{-1}\psi_\xi)(y)\,dxdy\right\vert\notag\\
&\leq\fint\limits_{Q^*_\ell} \left(\iint\limits_{Q_\ell\times Q_\ell}\!\!\!\rho_{|\gamma_{\xi'}|}(y)\,|\phi_\xi(x)|^2\,dxdy\right)^{1/2}\!\!\!\!\left(\iint\limits_{Q_\ell\times Q_\ell}\!\!\!\rho_{|\gamma_{\xi'}|}(x)\, \Big|W_{\geq m,\ell}^\infty\big(\xi-\xi',x-y\big)\Big|^2 \,||D_\xi|^{-1}\psi_\xi(y)|^2\,dxdy\right)^{1/2} \!\!\!\!d\xi' \notag\\
&\leq \;\left(C_G+\frac{C_{\geq m}}{\ell}\right)\,\|\gamma\|_{\mathfrak{S}_{1,1}}\,\|\psi_\xi\|_{L^2_\xi}\,\,\|\phi_\xi\|_{L^2_\xi}.
\end{align}
Thus, for every $\gamma \in \mathfrak{S}_{1,1}$ with $\gamma^*=\gamma$, since $|D^0|^{-1}\in \mathfrak{S}_{\infty,\infty}$, we have 
\begin{equation}
\Vert W_{\geq m,\gamma}|D^0|^{-1}\Vert_{Y}\leq \left(C_G+\frac{C_{\geq m}}{\ell}\right)\,\|\gamma\|_{\mathfrak{S}_{1,1}}.
\end{equation}
 Additionally, when $\gamma \in X^2_1$ with $\gamma=\gamma^*$ (which includes the case $\gamma\in X^2_\infty)$, one can prove that $W_{\geq m,\gamma}$ belongs to $\mathfrak{S}_{2,\infty}$. Recall that 
$ \Vert W_{\geq m,\gamma}\Vert_{\mathfrak{S}_{2,p}}= \Vert W_{\geq m,\gamma}(\cdot,\cdot)\Vert_{L^p(Q_\ell^*;L^2(Q_\ell\times Q_\ell))}$. The main ingredient is that $\sqrt{\rho_{|\gamma|}}\in H^1(Q_\ell)$ when $\gamma\in X^2_1$. Indeed, analogously to \eqref{eq:decomp-gamma-X}, for any $\gamma\in X^2_1$, we have 
\begin{equation}
\label{eq:decomp-gamma-X21}
|D_\xi|\gamma_{\xi}|D_\xi|=\sum_{n\geq 1}\mu_n(\xi)\left|v_n(\xi,\cdot)\rangle\right.\left.\langle v_n(\xi,\cdot)\right|\in \mathfrak{S}_{1,p}
\end{equation}
with $\left(v_{n}(\xi,\cdot),v_{m}(\xi,\cdot)\right)_{L^2_\xi}=\delta_{m,n}$ and $\|\gamma\|_{X^2_p}=\left(\fint_{Q_\ell^*}\big(\sum_{n\geq 1}|\mu_n(\xi)|\big)^p\,d\xi\right)^{1/p}$. Hence  
\[
\gamma_{\xi}=\sum_{n\geq 1}\mu_n(\xi)\left|u_n(\xi,\cdot)\left>\right<u_n(\xi,\cdot)\right|
\]
with $u_n(\xi,\cdot)=|D_\xi|^{-1}v_n(\xi,\cdot)\in H^1_\xi(Q_\ell)$, from which we deduce, adapting the  proof in  \cite[Eq. (4.42)]{catto2001thermodynamic},  that
\begin{equation}\label{eq:rhoH1}
\Vert\sqrt{\rho_{|\g|}}\Vert_{H^1(Q_\ell)} \leq   \|\gamma\|_{X^2_1}^{1/2},
\end{equation}
since $\rho_{|\g|}=\fint_{Q_\ell^*}\sum_{n\geq 1} |\mu_n(\xi)|\,|u_n(\xi,\cdot)|^2d\xi$.
For every $\xi\in Q_\ell^*$, the kernel of $W_{\geq m, \g, \xi}$ writes 
\begin{equation}\label{eq:noyau-Wsupm}
 W_{\geq m,\gamma,\xi}(x,y)=\fint_{Q_\ell^*}W_{\geq m,\ell}^\infty(\xi'-\xi,x-y)\,\gamma_{\xi'}(x,y)\,d\xi'.
\end{equation}
We now use the bound \eqref{prop:3.10} on $W_{\geq m,\ell}^\infty$ and split the kernel $W_{\geq m, \gamma,\xi}$ in two terms~:
\begin{align}
 W_{\geq m,\gamma,\xi}(x,y)&=\fint_{Q_\ell^*}G_{\ell}(x-y)\,\gamma_{\xi'}(x,y)\,d\xi'+ \fint_{Q_\ell^*}f(\xi'-\xi,x-y)\,\gamma_{\xi'}(x,y)\,d\xi'\notag\\
 &= G_{\ell}(x-y)\,\gamma(x,y)+ \fint_{Q_\ell^*}f(\xi'-\xi,x-y)\,\gamma_{\xi'}(x,y)\,d\xi',\label{eq:noyau-Wsupm2}
\end{align}
with $f\in L^\infty(2Q_\ell^*\times2Q_\ell)$.
The second term in \eqref{eq:noyau-Wsupm2} lies in $L^\infty(Q_\ell^*;L^2(Q_\ell\times Q_\ell))$ as soon as $\gamma \in \mathfrak{S}_{2,1}$. Therefore, the delicate contribution comes from the term involving the potential $G_\ell$ that we bound from above by using \eqref{eq:CS-gamma}:
\begin{align*}\label{eq:Wsupm-S2infty}
\MoveEqLeft\iint_{Q_\ell\times Q_\ell}|G_\ell(x-y)|^2\rho_{|\gamma|}(x)\rho_{|\gamma|}(y)dxdy 
\leq \Vert |G_\ell|^2\Vert_{L^1(Q_\ell)},\Vert \rho\Vert_{L^3(Q_\ell)}
    \Vert \rho_{|\gamma|}\Vert_{L^{3/2}(Q_\ell)}\\
    &\leq \Vert G_\ell\Vert_{L^2(Q_\ell)}^2\,\Vert \sqrt{\rho_{|\gamma|}}\Vert_{L^6(Q_\ell)}^2
    \Vert \sqrt{\rho_{|\gamma|}}\Vert_{L^{3}(Q_\ell)}^2\\
    &\leq C\, \Vert G_\ell\Vert_{L^2(Q_\ell)}^2\, \Vert \sqrt{\rho_{|\gamma|}}\Vert_{H^{1}(Q_\ell)}^4.  
\end{align*}
Finally, using \eqref{eq:rhoH1}, we have
\begin{equation}
\label{bd:Wsupm-X2infty}
\Vert W_{\geq m,\gamma}\Vert_{\mathfrak{S}_{2,\infty}} \leq C\,\Vert \g\Vert_{X^2_1},
\end{equation}
where $C$ is a positive constant that depends only on $m$ and $\ell$.

\medskip

We now study the contribution of the term involving $W_{<m, \ell}^\infty$, that carries the singularities in the $\eta$ variable. For this part of the proof, we only need that $\gamma \in Y$. Then, we may write
\begin{align}
\MoveEqLeft\left\vert\fint\limits_{Q^*_\ell}d\xi'\!\!\iint\limits_{Q_\ell\times Q_\ell}W_{< m,\ell}^\infty\Big(\xi-\xi',x-y\Big)\phi_\xi^*(x)\gamma_{\xi'}(x,y)\psi_\xi(y)dxdy\right\vert\notag\\
&\leq\left|\fint\limits_{Q^*_\ell}\iint_{Q_\ell\times Q_\ell}\frac{4\pi}{\ell^3}\sum_{\substack{k\in \mathbb{Z}^3\\|k|_\infty\leq m-1}}\frac{e^{-i\big(\xi'-\xi-\frac{2\pi k}\ell\big)\cdot (x-y)}}{\big|\xi'-\xi-\frac{2\pi k}\ell\big|^2}\,\phi_\xi^*(x)\,\gamma_{\xi'}(x,y)\,\psi_\xi(y)\,dxdyd\xi'\right|\notag\\
& \leq\frac{4\pi}{\ell^3}\left|\sum_{\substack{k\in \mathbb{Z}^3\\|k|_\infty\leq m-1}}\fint\limits_{Q^*_\ell+\frac{2 k\pi}{\ell}}\iint_{Q_\ell\times Q_\ell}\frac{e^{-i(\xi'-\xi)\cdot (x-y)}}{\big|\xi'-\xi\big|^2}\,\phi_\xi^*(x)\,\gamma_{\xi'}(x,y)\,\psi_\xi(y)\,dxdyd\xi'\right|\notag\\
&\leq \frac{4\pi}{\ell^3}\sum_{\substack{k\in \mathbb{Z}^3\\|k|_\infty\leq m-1}}\fint\limits_{Q^*_\ell+\frac{2 k\pi}{\ell}}\frac{d\xi'}{\big|\xi'-\xi\big|^2}\left|\left( e^{i(\xi'-\xi)\cdot(\cdot)}\phi_\xi(\cdot),\gamma_{\xi'}e^{i(\xi'-\xi)\cdot(\cdot)}\psi_\xi(\cdot)\right)_{L^2}\right| \notag\\
&\leq C_{\leq m,\ell}\,\supess_{\xi'\in Q_\ell^*}\|\gamma_{\xi'}\|_{\mathcal{B}(L^2_{\xi'})}\; \|\psi_{\xi}\|_{L^2_\xi}\,\|\phi_{\xi}\|_{L^2_\xi}\notag\\
&= C_{\leq m,\ell}\,\|\gamma\|_{Y}\,\|\psi_{\xi}\|_{L^2_\xi}\,\|\phi_{\xi}\|_{L^2_\xi}\label{eq:3.39}
 \end{align}
with
\begin{align}
 C_{\leq m,\ell}&=\frac{4\pi}{\ell^3}\adjustlimits\sup_{\xi\in Q_\ell^*}\sum_{\substack{k\in \mathbb{Z}^3\\|k|_\infty\leq (m-1)}}\;\fint\limits_{Q^*_\ell+\frac{2 k\pi}{\ell}}\frac{d\xi'}{|\xi'-\xi|^2}=\frac{(2m-1)}{2\pi\ell}\,\int\limits_{\,[-1,1)^3}\frac{d\xi'}{|\xi'|^2}\label{eq:C-m-l}.
\end{align}
Thus, for every $\gamma\in Y$, $W_{<m,\g}\in Y$ and 
\begin{equation}\label{eq:bd-Winfm}
\Vert W_{<m,\g}\Vert_Y\leq C_{\leq m,\ell}\,\|\gamma\|_{Y}.
\end{equation}
We now make a further assumption on $\gamma$; namely, $\gamma\in \mathfrak{S}_{1,\infty}$. (Actually, we need $\gamma\in \mathfrak{S}_{1,4} $.) We first observe that \[
\gamma_{\xi'+\frac{2k\pi}{\ell}}(x,y)=e^{\frac{2ik\pi}{\ell}\cdot(x-y)}\,\gamma_{\xi'}(x,y)\quad \text{for every }\xi'\in Q_\ell^*, \, k\in \mathbb{Z}^3 \text{ and }
x,\,y\in \mathbb{R}^3.\]
In particular, $\rho_{\gamma_{\xi'+\frac{2k\pi}{\ell}}}=\rho_{\gamma_{\xi'}}\quad \text{for every }\xi'\in Q_\ell^* \text{ and } k\in \mathbb{Z}^3$, and the function of $\xi\mapsto \Tr_{L^2_{\xi}}(\gamma_{\xi})$ is $Q_\ell^*$-periodic.
Next, we write
\begin{align}
\MoveEqLeft
\left|\fint\limits_{Q^*_\ell}d\xi'\iint\limits_{Q_\ell\times Q_\ell}\frac{4\pi}{\ell^3}
\sum_{\substack{k\in \mathbb{Z}^3\\|k|_\infty\leq m-1}}
\frac{e^{-i\big(\xi'-\xi-\frac{2\pi k}\ell\big)\cdot (x-y)}}{\big|\xi'-\xi-\frac{2\pi k}\ell\big|^2}\phi_\xi^*(x)\,\gamma_{\xi'}(x,y)\,\psi_\xi(y)\,dxdy
\right|
\notag \\
&\leq \frac{4\pi}{\ell^3}\smashoperator[l]{\sum_{\substack{k\in \mathbb{Z}^3\\|k|_\infty\leq m-1}}}\,\fint\limits_{Q^*_\ell+\frac{2k\pi}{\ell}}\frac{d\xi'}{\big|\xi'-\xi\big|^2}\smashoperator[l]{\iint\limits_{Q_\ell\times Q_\ell}}\,\rho^{1/2}_{|\gamma_{\xi'}|}(x)\,\rho^{1/2}_{|\gamma_{\xi'}|}(y)\,|\psi_\xi(y)|\,|\phi_\xi^*(x)|\,dxdy \notag \\
&=\frac{1}{2\pi^2}\int\limits_{(2m-1)Q^*_\ell}\frac{d\xi'}{\big|\xi'-\xi\big|^2}\smashoperator[l]{\iint\limits_{Q_\ell\times Q_\ell}}\,\rho^{1/2}_{|\gamma_{\xi'}|}(x)\,\rho^{1/2}_{|\gamma_{\xi'}|}(y)\,|\psi_\xi(y)|\,|\phi_\xi(x)|\,dxdy \notag \\
&\leq\frac{1}{2\pi^2}\left(\int\limits_{(2m-1)Q^*_\ell}\frac{\|\gamma_{\xi'}\|_{\mathfrak{S}_1(\xi')}}{\big|\xi'-\xi\big|^2}\,d\xi'\right)\, \|\psi_\xi\|_{L^2_\xi}\,\|\phi_\xi\|_{L^2_\xi},\label{eq:last-line}
\end{align}
where the last estimate follows from the Cauchy--Schwarz inequality. Here and below we use the fact that 
\[
\sum_{\substack{k\in \mathbb{Z}^3\\|k|_\infty\leq m-1}}\fint_{Q^*_\ell} f\big(\xi'-\frac{2\pi k}{\ell}\big)\,d\xi'= \sum_{\substack{k\in \mathbb{Z}^3\\|k|_\infty\leq m-1}} \fint_{\frac{2\pi k}\ell+Q^*_\ell}f(\xi') \,d\xi'=\frac{\ell^3}{(2\pi)^3}\int_{(2m-1)Q^*_\ell} f(\xi')\,d\xi',
\]
since $(Q_\ell^*+2\pi k/\ell)\cap (Q_\ell^*+2\pi k'/\ell)=\emptyset$ 
whenever $k,k'\in\mathbb{Z}^3$ with $k\neq k'$.
We focus on the quantity inside the brackets in the last inequality. By H\"older's inequality, for $\gamma\in \mathfrak{S}_{1,\infty}$, and for some constant $C_{\leq m,\ell}'$, we obtain
\begin{align}
\Vert W_{< m,\gamma}\Vert_Y\notag
&\leq\frac{1}{2\pi^2}\supess_{\xi\in Q_{\ell}^*}\int_{(2m-1)Q^*_\ell}\frac{\|\gamma_{\xi'}\|_{\mathfrak{S}_1(\xi')}}{\big|\xi'-\xi\big|^2}\,d\xi'\\
&\leq C\,\supess_{\xi\in Q_{\ell}^*}\left(\int_{(2m-1)Q^*_\ell}
\frac{d\xi'}{\big|\xi'-\xi\big|^{8/3}}\right)^{3/4}\,\left(\int_{(2m-1)Q^*_\ell}\|\gamma_{\xi'}\|_{\mathfrak{S}_1(\xi')}^4\,d\xi'\right)^{1/4}\notag\\
&=(2m-1)^{3/4}\,C\,\left(\int_{(2m-1)Q^*_\ell}
\frac{d\xi'}{\big|\xi'-\xi\big|^{8/3}}\right)^{3/4}\,\Vert \gamma\Vert_{\mathfrak{S}_{1,4}}\notag\\
&\leq C_{\leq m,\ell}' \|\gamma\|_{\mathfrak{S}_{1,\infty}}^{3/4}\|\gamma\|_{\mathfrak{S}_{1,1}}^{1/4}. \label{eq:last-line2}
\end{align}
Following the lines of the proof of \eqref{eq:last-line} and \eqref{eq:last-line2}, 
we obtain, for every $\xi\in Q_\ell^*$, 
\[
\Vert W_{<m,\g,\xi}(\cdot,\cdot)\Vert_{L^2(Q_\ell\times Q_\ell)}\;
\leq \;C\,
\sum_{\substack{k\in\mathbb{Z}^3
\\ |k|_\infty\leq m-1}} \fint\limits_{Q^*_\ell}
\frac{\Vert\gamma_{\xi'}(\cdot,\cdot)\Vert_{L^2(Q_\ell\times Q_\ell)}}{\big|\xi'-\xi-\frac{2\pi k}\ell\big|^2}\,\,d\xi'
\;=\;  C\,\int\limits_{(2m-1)Q^*_\ell}\frac{\|\gamma_{\xi'}\|_{\mathfrak{S}_2(\xi')}}{\big|\xi'-\xi\big|^2}\,d\xi', 
\]
and, by Young's convolution inequality, the right-hand side belongs to $L^p(Q_\ell^*)$ as soon as $\gamma\in\mathfrak{S}_{2,p}$ with $1\leq p\leq +\infty$ (which is guaranteed whenever $\gamma\in \mathfrak{S}_{1,p}$), where $C$ is a positive constant which depends only on $m$ and $\ell$.
In particular, for every $1\leq p\leq +\infty$, 
\begin{equation}\label{eq:bd-Winfm-S2infty}
\Vert W_{<m,\gamma}\Vert_{\mathfrak{S}_{2,p}}\leq C\, \Vert \gamma \Vert_{\mathfrak{S}_{2,p}}.
\end{equation}

Since $\|\gamma\|_{\mathfrak{S}_{1,1}}\leq \|\gamma\|_X$ and $|D^{0}|^{-1/2}\leq 1$, the statement of the lemma follows~: From \eqref{eq:B9} and \eqref{eq:3.39}, we obtain \eqref{3.eq:4.6}; from \eqref{eq:B9} and \eqref{eq:last-line2}, we obtain \eqref{3.eq:4.10'}; from \eqref{eq:B10} and \eqref{eq:3.39}, we obtain \eqref{3.eq:4.7}. More precisely,
\begin{align}\label{eq:CW}
  C_W=C_H+C_\ell,\quad
  C'_W=C_G+C_\ell,\quad
  C''_W=C_H+C_\ell',
\end{align}
with
\begin{align}\label{eq:Cl}
  C_\ell:=\inf_{\substack{m\in\mathbb{N}\\m\geq 2}}\left(\frac{C_{\geq m}}{\ell}+C_{\leq m,\ell}\right),\quad C_\ell':=\inf_{\substack{m\in\mathbb{N}\\m\geq 2}}\left(\frac{C_{\geq m}}{\ell}+C'_{\leq m,\ell}\right).
\end{align}

\section{Proof of Lemma \ref{3.lem:5.1}}\label{sec:C}
Analogously to \eqref{eq:ps-W}, we have
\begin{equation}
\|V_{\gamma,\xi}\psi_\xi\|_{L_\xi^{2}}=\sup_{\phi_\xi\in L_\xi^{2},\; \|\phi_\xi\|_{L^2_\xi}=1}|(\phi_\xi,V_{\gamma,\xi}\psi_\xi)|.
\end{equation}
We can rewrite as $W_{\ell}^\infty=W_{< m,\ell}^\infty+G_\ell+ (W_{\geq m,\ell}^\infty-G_\ell)$. According to Proposition \ref{prop:minW} and \eqref{eq:3.39}, the terms associated to $W_{< m,\ell}^\infty$ and $(W_{\geq m,\ell}^\infty-G_\ell)$ are easily bounded. The aim of this section is to get a better estimate on the following term~:
\[
\iint_{Q_\ell\times Q_\ell}G_\ell(x-y)\rho_{\gamma}(y)\phi_\xi^*(x)\psi_\xi(x)\,dxdy-\fint_{Q^*_\ell}d\xi'\iint_{Q_\ell\times Q_\ell}G_\ell(x-y)\phi_\xi^*(x)\gamma_{\xi'}(x,y)\psi_\xi(y)\,dxdy.
\]
From now on, for any function $f\in L^2(Q_\ell,\mathbb{C}^4)$, we denote $f:=(f^\alpha)_{1\leq \alpha\leq 4}$. We use the decomposition \eqref{eq:decomp-gamma} for $\gamma\in \mathfrak{S}_{1,1}\cap Y$ such that $\gamma=\gamma^*$. Then as $G(x)=G(-x)$, for almost every $\xi\in Q_\ell^*$, we may write
  \begin{align}\label{eq:3.38}
  \MoveEqLeft \iint_{Q_\ell\times Q_\ell}G_\ell(x-y)\left[\rho_{\gamma_{\xi'}}(y)\phi_\xi^*(x)\psi_\xi(x)-\phi_\xi^*(x)\gamma_{\xi'}(x,y)\psi_\xi(y)\right]\,dxdy\notag\\
  &=\sum_{n\geq 1}\lambda_n(\xi')\iint_{Q_\ell\times Q_\ell} G_\ell(x-y)\Big(|u_n(\xi',y)|^2\phi_\xi^*(x)\psi_\xi(x)-\phi_\xi^*(x)u_n(\xi',x)u_n^*(\xi',y)\psi_\xi(y)\Big)dxdy\notag\\
  &=\frac{1}{2}\sum_{n\geq 1}\sum_{1\leq \alpha,\beta\leq 4}\lambda_n(\xi')\iint_{Q_\ell\times Q_\ell}G_\ell(x-y)\left(u_n^\alpha(\xi',y)\phi_\xi^\beta(x)-\phi_\xi^\alpha(y) u_n^\beta(\xi',x)\right)^*\notag\\
  &\qquad\qquad\times\left(u_n^\alpha(\xi',y) \psi_\xi^\beta(x)- \psi_\xi^\alpha(y) u_n^\beta(\xi',x)\right)\,dxdy.
\end{align}

\textbf{Estimate for \eqref{3.eq:5.1}.} By Lemma \ref{lem:Gbound}, we have
\begin{align}\label{eq:C3}
 \MoveEqLeft \left|\sum_{1\leq \alpha,\beta\leq 4}\iint\limits_{Q_\ell\times Q_\ell}G_\ell(x-y)\left(u_n^\alpha(\xi',y)\phi^\beta_\xi(x)-\phi^\alpha_\xi(y) u_n^\beta(\xi',x)\right)^*\left(u_n^\alpha(\xi',y)\psi^\beta_\xi(x)-\psi^\alpha_\xi(y) u_n^\beta(\xi',x)\right)\,dxdy\right|\notag\\
 &\leq \left(\sum_{1\leq \alpha,\beta\leq 4}\iint_{Q_\ell\times Q_\ell} |G_\ell(x-y)|^2\left|u_n^\alpha(\xi',y)\psi_\xi^\beta(x)-u_n^\beta(\xi',x)\psi_\xi^\alpha(y)\right|^2\,dxdy\right)^{1/2}\notag\\
 &\quad\times \left(\sum_{1\leq \alpha,\beta\leq 4}\iint_{Q_\ell\times Q_\ell} \left|u_n^\alpha(\xi',y) \phi^\beta_\xi(x)-u_n^\beta(\xi',x) \phi^\alpha_\xi(y)\right|^2\,dxdy\right)^{1/2}.
\end{align}
Thus according to the Cauchy--Schwarz inequality, we have
\begin{align*}
\MoveEqLeft  \sum_{1\leq \alpha,\beta\leq 4}\left|\iint\limits_{Q_\ell\times Q_\ell}G_\ell(x-y)\left(u_n^\alpha(\xi',y)\phi_\xi^\beta(x)-\phi_\xi^\alpha(y) u_n^\beta(\xi',x)\right)^*\left(u_n^\alpha(\xi',y)\psi^\beta_\xi(x)-\psi^\alpha_\xi(y) u_n^\beta(\xi',x)\right)\,dxdy\right|\\
 &\leq 2\left(\sum_{1\leq \alpha,\beta\leq 4}\iint\limits_{Q_\ell\times Q_\ell} |G_\ell(x-y)|^2|u_n(\xi',y)|^2|\psi_\xi(x)|^2dxdy\right)^{1/2}\!\!\!\!\left(\iint\limits_{Q_\ell\times Q_\ell} |u_n^*(\xi',y)|^2|\phi_\xi^*(x)|^2dxdy\right)^{1/2}\\
 &\leq 2C_G\|\phi_\xi\|_{L^2_\xi}\,\||D_\xi|\psi_\xi\|_{L^2_\xi}.
\end{align*}
Substituting this inequality into \eqref{eq:3.38} and using the decomposition \eqref{3.eq:2.1}, we get
\begin{align*}
 \MoveEqLeft  \left|\iint_{Q_\ell\times Q_\ell}G_\ell(x-y)\left(\rho_{\gamma_{\xi'}}(y)\phi_\xi^*(x)\psi_\xi(x)-\phi_\xi^*(x)\gamma_{\xi'}(x,y)\psi_\xi(y)\right)\,dxdy\right|\\
  & \leq C_G\sum_{n\geq 1}|\lambda_n(\xi')|\|\phi_\xi\|_{L^2_\xi}\||D_\xi|\psi_\xi\|_{L^2_\xi} =C_G\|\gamma_{\xi'}\|_{\mathfrak{S}_1(\xi')}\|\phi_\xi\|_{L^2_\xi}\||D_\xi|\psi_\xi\|_{L^2_\xi},
\end{align*}
from which we get
  \begin{align}
  &\left|\iint_{Q_\ell\times Q_\ell}G_\ell(x-y)\rho_{\gamma}(y)\phi_\xi^*(x)\psi_\xi(x)\,dxdy-\fint_{Q^*_\ell}d\xi'\iint_{Q_\ell\times Q_\ell}G_\ell(x-y)\phi_\xi^*(x)\gamma_{\xi'}(x,y)\psi_\xi(y)\,dxdy\right|\notag\\
  &\qquad \qquad \leq C_G\|\gamma\|_{\mathfrak{S}_{1,1}}\|\phi_\xi\|_{L^2_\xi}\||D_\xi|\psi_\xi\|_{L^2_\xi}.\label{eq:3.40}
  \end{align}
Combining \eqref{eq:3.40} with Proposition \ref{prop:minW} and \eqref{eq:3.39}, we get for any $\phi_\xi\in L^2_{\xi}$ and $\psi_\xi\in H^1_{\xi}$,
\[
  \left|\left(\phi_\xi,V_{\gamma,\xi}\psi_\xi\right)\right|\leq (C_G+C_\ell)\|\gamma\|_{\mathfrak{S}_{1,1}\cap Y}\|\phi_\xi\|_{L^2_\xi}\||D_\xi|\psi_\xi\|_{L^2_\xi},
\]
hence \eqref{3.eq:5.1} with
\begin{equation}\label{eq:3.41}
C_{EE}':=C_G+C_\ell
\end{equation}
with $C_\ell$ given in \eqref{eq:Cl}.\medskip

\textbf{Estimate for \eqref{3.eq:5.1'}.} As $\gamma\in Z$, we use the decomposition \eqref{eq:decomp-gamma-X} for 
$\gamma_{\xi}$.
Analogously to \eqref{eq:C3}, we also have
\begin{align*}
 \MoveEqLeft \left|\sum_{1\leq \alpha,\beta\leq 4}\iint\limits_{Q_\ell\times Q_\ell}G_\ell(x-y)\left(u_n^\alpha(\xi',y)\phi^\beta_\xi(x)-\phi^\alpha_\xi(y) u_n^\beta(\xi',x)\right)^*\left(u_n^\alpha(\xi',y)\psi^\beta_\xi(x)-\psi^\alpha_\xi(y) u_n^\beta(\xi',x)\right)\,dxdy\right|\\
 &\leq 2\left(\iint\limits_{Q_\ell\times Q_\ell} |G_\ell(x-y)||u_n(\xi',y)|^2|\psi_\xi(x)|^2dxdy\right)^{1/2}\!\!\!\!\left(\iint\limits_{Q_\ell\times Q_\ell} |G_\ell(x-y)||u_n^*(\xi',y)|^2|\phi_\xi^*(x)|^2dxdy\right)^{1/2}
\end{align*}
from which by the decomposition \eqref{eq:decomp-gamma-X} we get
  \begin{align}\label{eq:3.42}
  \MoveEqLeft\left|\iint_{Q_\ell\times Q_\ell}G_\ell(x-y)\rho_{\gamma}(y)\phi^*_\xi(x)\psi_\xi(x)dxdy-\fint_{Q^*_\ell}d\xi'\iint_{Q_\ell\times Q_\ell}G_\ell(x-y)\phi^*_\xi(x)\gamma_{\xi'}(x,y)\psi_\xi(y)\,dxdy\right|\notag\\
  &\leq \left(\fint_{Q_{\xi'}}d\xi' \sum_{n\geq 1}|\lambda_n(\xi')|\iint_{Q_\ell\times Q_\ell} |G_\ell(x-y)||u_n(\xi',y)|^2|\psi_\xi(x)|^2dxdy\right)^{1/2}\notag\\
 &\quad\times\left(\fint_{Q_{\xi'}}d\xi' \sum_{n\geq 1}|\lambda_n(\xi')| \iint_{Q_\ell\times Q_\ell} |G_\ell(x-y)||u_n(\xi',y)|^2|\phi_\xi^*(x)|^2dxdy\right)^{1/2}\notag\\
 &\leq C_H\|\gamma\|_X\|\phi_\xi\|_{L^2_\xi}\|\psi_\xi\|_{L^2_\xi}
  \end{align}
where the last inequality holds by using Lemma \ref{3.lem:4.2}.

Combining \eqref{eq:3.42} with Proposition \ref{prop:minW} and estimate \eqref{eq:3.39}, we get for any $\phi_\xi\in L^2_{\xi}$ and $\psi_\xi\in H^1_{\xi}$,
\[
  \left|\left(\phi_\xi,V_{\gamma,\xi}\psi_\xi\right)\right|\leq (C_H+C_\ell)\|\gamma\|_{Z}\|\phi_\xi\|_{L^2_\xi}\||D_\xi|\psi_\xi\|_{L^2_\xi},
\]
hence \eqref{3.eq:5.1} and
\begin{equation}\label{eq:3.43}
C_{EE}:=C_H+C_\ell.
\end{equation}

\textbf{Estimate for \eqref{3.eq:5.1''}.} Combining with Proposition \ref{prop:minW} and estimate \eqref{eq:3.39}, analogously to \eqref{eq:3.42} it can be derived directly from:
\begin{align*}
  \MoveEqLeft\left|\iint_{Q_\ell\times Q_\ell}G_\ell(x-y)\rho_{\gamma}(y)\psi_\xi^*(x)\psi_\xi(x)dxdy-\fint_{Q^*_\ell}d\xi'\iint_{Q_\ell\times Q_\ell}G_\ell(x-y)\psi_\xi^*(x)\gamma_{\xi'}(x,y)\psi_\xi(y)\,dxdy\right|\\
  &\leq \left(\fint_{Q_{\xi'}}d\xi' \sum_{n\geq 1}|\lambda_n(\xi')|\iint_{Q_\ell\times Q_\ell} |G_\ell(x-y)||u_n(\xi',y)|^2|\psi_\xi(x)|^2dxdy\right)\\
 &\leq C_H\|\gamma\|_{\mathfrak{S}_{1,1}}\||D_\xi|^{1/2}\psi_\xi\|_{L^2_\xi}^2
\end{align*}
using the decomposition \eqref{eq:decomp-gamma} for $\gamma_\xi$. Hence \eqref{3.eq:5.1''} and $C_{EE}$.
\medskip

\textbf{Estimate for \eqref{3.eq:5.1'''}.} Notice that $|\gamma_{\xi'}(x,y)|\leq \rho_{\gamma_{\xi'}}(x)^{1/2}\rho_{\gamma_{\xi'}}(y)^{1/2}$ since $\gamma\geq 0$. Thus, according to Lemma \ref{lem:Gbound} and the Cauchy--Schwarz inequality,
\begin{align*}
\MoveEqLeft \fint_{Q^*_\ell}d\xi'\iint_{Q_\ell\times Q_\ell}G_\ell(x-y)\psi_\xi^*(x)\gamma_{\xi'}(x,y)\psi_\xi(y)\,dxdy-\iint_{Q_\ell\times Q_\ell}G_\ell(x-y)\rho_{\gamma}(x)|\psi_\xi(y)|^2\,dxdy\\
  & \leq \iint_{Q_\ell\times Q_\ell}(|G_\ell(x-y)|-G_\ell(x-y))\rho_{\gamma}(y)|\psi_\xi(x)|^2\,dxdy\leq \frac{2C_0}{\ell}\|\gamma\|_{\mathfrak{S}_{1,1}}\|\psi\|_{L^2_\xi}^2.
\end{align*}
Combining with Proposition \ref{prop:minW} and \eqref{eq:3.39}, we get
\[
\left(\psi_\xi,V_{\gamma,\xi}\psi_\xi\right)\geq - \left(\frac{C_0}{\ell}+C_\ell\right)
\|\gamma\|_{\mathfrak{S}_{1,1}\cap Y}\|\psi_\xi\|_{L^2_\xi}^2,
\]
hence \eqref{3.eq:5.1''} and 
\begin{align}\label{eq:3.44}
  C_{EE}''=\frac{2C_0}{\ell}+C_\ell.
\end{align}

\section{Numerical results about constants}\label{sec:D}
In this section, we will show the numerical results about the constants used in Remark \ref{3.rem:2.15} under the condition $\ell=1000$. Next, we show that Assumption \ref{3.ass:2.1} is satisfied for $q\leq 17$ for the neutral systems.

We compute numerically the value of the bound of the potential $G_\ell-\frac{1}{|x|}$. First of all, we calculate 
\[
\sum_{k\in \mathbb{Z}^3\setminus\{ 0\}}\frac{1}{|k|^4}\approx 16.512.
\]
Thus, $C_0\approx 5.019$ and we can choose $C_H= C_G\approx 2.011$. Concerning the estimates involving the potential $W_\ell$, we set $m=2$. When $R\approx \frac{1}{2}$,
\[
|C_{\geq 2}|\leq 20.912,\quad C_{\leq 2,1000}\approx 0.010.
\]
Thus, we get $C_W\approx 2.042$, and $C'_W\approx 2.042$. Then, $C_{EE}\approx 2.052$, $C_{EE}'\approx 2.052$ and $C_{EE}''\approx 0.041$.

Finally, we estimate $c^*(\lceil q\rceil )$ which  is given by \eqref{bornes}. 
Let $u_{p,\xi}(x)=e^{i(\frac{2\pi}{\ell}p+\xi)\cdot x}$ with $p\in\mathbb{Z}^3$. Then $(u_{p,\xi})_{p\in\mathbb{Z}^3}$ is an orthogonal basis on $L^2_\xi(Q_\ell)$. Obviously, $(\Lambda^+ u_{p,\xi})_p$ is also an orthogonal basis on $L^2_\xi(Q_\ell)$. Let 
\[
V_{\lceil q\rceil}=\text{Span}\Set*{\Lambda^+ u_{p,\xi}(x)\given p=(j, 0,0), j\in \{1,\cdots,\lceil q\rceil\} }.
\]
Then
\[
c^*(\lceil q\rceil)\leq \sup_{\xi\in Q_\ell^*}\sup_{u_\xi^+\in V_{\lceil q\rceil}}\frac{\||D_\xi|^{1/2}u^+_\xi\|^2_{L_\xi^{2}}}{\|u^+_\xi\|^2_{L_\xi^{2}}}\leq \sqrt{1+\frac{4\,\pi^2 (\lceil q\rceil+1)^2}{\ell^2}}.
\]
Now we can check Assumption \ref{3.ass:2.1} for $z=q=17$. The calculation leads to $a\approx 0.010$ and $c^*(17)\leq 1.006$.
Thus, we have
\begin{itemize}
  \item $\kappa+\frac{\alpha}{2} C_{EE}q^+\approx 0.631<1$,
  \item $2a\,\sqrt{\max\{(1-\kappa-\frac{\alpha}{2} C_{EE}q^+)^{-1}(1-\kappa)^{-1} c^*(\lceil q \rceil)q;1\}q^+}\leq 0.973<1$.
  \end{itemize}
Consequently, Assumption \ref{3.ass:2.1} is satisfied for $z=q\leq17$ {whenever $\ell=1000$}.

\medskip
\begin{refcontext}[sorting=nyt]
\printbibliography[heading=bibintoc, title={Bibliography}]
\end{refcontext}

\end{document}